\documentclass[a4paper,reqno,11pt]{amsart}

\usepackage[utf8]{inputenc}
\usepackage{microtype}

\usepackage{amssymb}
\usepackage{mathrsfs}
\usepackage{mathtools}
\usepackage{comment}

\usepackage{color}
\usepackage{tikz}
\usetikzlibrary{topaths}

\usepackage{hyperref}

\newtheorem{theorem}{Theorem}[section]
\newtheorem*{theorem*}{Theorem}
\newtheorem*{maintheorem}{Main Theorem}
\newtheorem{lemma}[theorem]{Lemma}
\newtheorem*{morselemma}{Morse Lemma}
\newtheorem{proposition}[theorem]{Proposition}
\newtheorem*{proposition*}{Proposition}
\newtheorem*{brownscriterion}{Brown's Criterion}
\newtheorem{corollary}[theorem]{Corollary}
\newtheorem*{corollary*}{Corollary}

\newtheorem*{conjecture*}{Conjecture}

\theoremstyle{definition}
\newtheorem{definition}[theorem]{Definition}
\newtheorem{remark}[theorem]{Remark}
\newtheorem{observation}[theorem]{Observation}
\newtheorem{example}[theorem]{Example}

\newcommand{\Z}{\mathbb{Z}}
\newcommand{\N}{\mathbb{N}}
\newcommand{\Poset}{\mathcal{P}}
\newcommand{\spraige}{\mathscr{S}}
\newcommand{\match}{\mathcal{M}}
\newcommand{\hatch}{\mathcal{H}}
\newcommand{\matcharc}{\mathcal{MA}}
\newcommand{\hatcharc}{\mathcal{HA}}
\newcommand{\restrarc}{\mathcal{RA}}
\newcommand{\cover}{\mathcal{U}}
\newcommand{\fiberjoin}{\mathscr{J\!V\!F}}
\newcommand{\abs}[1]{\lvert{#1}\rvert}
\DeclareMathOperator{\image}{Im}

\DeclareMathOperator{\Stab}{Stab}
\DeclareMathOperator{\lk}{lk}
\DeclareMathOperator{\st}{st}
\DeclareMathOperator{\dlk}{{\lk}{\downarrow}}
\DeclareMathOperator{\dst}{{\st}{\downarrow}}
\DeclareMathOperator{\Bot}{bot}

\newcommand{\braiges}{\mathscr{B}}
\newcommand{\pbraiges}{\mathscr{P\!B}}

\newcommand{\elbraigecpx}{\mathcal{EB}}
\newcommand{\braigecpx}{\mathcal{B}}

\newcommand{\elpbraigecpx}{\mathcal{EPB}}
\newcommand{\pbraigecpx}{\mathcal{PB}}

\newcommand{\family}{\mathscr{F}}
\newcommand{\arcfam}{\mathscr{A\!F}}
\newcommand{\braigefam}{\mathscr{B\!F}}
\newcommand{\cosets}{\mathcal{U}}
\newcommand{\nerve}{\mathcal{N}}
\DeclareMathOperator{\CC}{CC}

\DeclareMathOperator*{\bigjoin}{\text{\raisebox{-0.2ex}{\LARGE $\ast$}}}

\let\bigast\bigjoin  
\let\Bigast\Bigjoin

\newcommand{\defeq}{\mathrel{\vcentcolon =}}

\DeclareMathOperator{\id}{id}
\DeclareMathOperator{\F}{F}

\newcommand{\Fbr}%
   {F_{\operatorname{br}}}                 

\newcommand{\Vbr}%
   {V_{\operatorname{br}}}                 

\newcommand{\Tbr}%
   {T_{\operatorname{br}}}                 

\newcommand{\VbrH}%
   {\widehat{V}_{\operatorname{br}}}       

\newcommand{\FbrH}%
   {\widehat{F}_{\operatorname{br}}}       

\numberwithin{equation}{section}

\newlength{\stdbskip}
\newlength{\bskipabstract}
\newlength{\bskipmaintext}
\newlength{\deltatextwidth}
\newlength{\deltatextheight}

\begin{document}

\title{The braided Thompson's groups are of type~$\F_\infty$}
\date{\today}


\author[K.-U.~Bux]{Kai-Uwe Bux}

\author[M.~G.~Fluch]{Martin G.~Fluch}

\author[M.~Marschler]{Marco Marschler}

\author[S.~Witzel]{Stefan Witzel}

\author[M.~C.~B.~Zaremsky]{Matthew C.~B.~Zaremsky \\~\\ With an Appendix by Matthew C.~B.~Zaremsky}

\begin{abstract}
\setlength{\baselineskip}{\bskipabstract}
We prove that the braided Thompson's groups $\Vbr$ and $\Fbr$ are of type $\F_\infty$, confirming a conjecture by John Meier.
The proof involves showing that matching complexes of arcs on surfaces are highly connected.\\
\noindent In an appendix, Zaremsky uses these connectivity results to exhibit families of subgroups of the pure braid group that are highly generating, in the sense of Abels and Holz.
\end{abstract}

\maketitle
\thispagestyle{empty}


\setlength{\baselineskip}{\bskipmaintext}

\noindent
A group is of \emph{type} $F_\infty$ if it admits a classifying space
whose $n$-skeleton is compact for every $n$. The case $n=2$ shows that
such a group is in particular finitely presented. Prominent examples
of groups of type $\F_\infty$ include Thompson's groups $F$, $T$ and
$V$, and the braid groups $B_n$.  A braided variant of Thompson's
group $V$, which we will denote $\Vbr$, was introduced independently
by Brin and Dehornoy~\cite{brin07,dehornoy06}. This group contains $F$
as a subgroup, along with copies of the braid group $B_n$ for each
$n\in\N$, and was shown to be finitely presented by
Brin~\cite{brin06}.  Brady, Burillo, Cleary and Stein~\cite{brady08}
introduced another braided Thompson's group, which we denote $\Fbr$,
and which contains the pure braid groups~$PB_n$ in a similar way to how
$\Vbr$ contains $B_n$.  They also proved that~$\Fbr$ is finitely
presented.  The notation used in~\cite{brin07,brady08} is~$BV$
and~$BF$.  The relationship between~$V$ and~$\Vbr$ is in many ways
reminiscent of the relationship between a Coxeter group and its
corresponding Artin group. For example, there is a presentation of~$V$
that can be converted to a presentation for~$\Vbr$ by dropping the
relations that the generators are involutions~\cite{brin06}. 

In this paper we prove that the braided Thompson's groups are of type
$\F_\infty$. In the case of~$\Vbr$, this was conjectured by John Meier
already in~$2001$. This question was also discussed in~\cite[Remark~5.1~(1)]{funar08} and
\cite[Remark~3.3]{funar11}.

\begin{maintheorem}\label{thrm:maintheorem}
The braided Thompson's groups $\Vbr$ and $\Fbr$ are of type~$\F_\infty$.
\end{maintheorem}

Our proof is geometric. The starting point is that each braided
Thompson's group acts naturally on an associated poset complex. One
key step is to restrict this action to an invariant cubical subcomplex
that is smaller and therefore easier to understand locally. We call
this cube complex a \emph{Stein space} because a similar space was
first studied by Stein \cite{stein92} for the group $F$.

The descending links arising in the study of the local structure are
modeled by \emph{matching complexes on a surface}, which may be of
independent interest. These complexes are given by a graph together
with a surface containing the vertices of the graph, and consist of
arc systems that yield a matching of the graph. We show these
complexes to be highly connected for certain families of graphs.

\begin{theorem*}[Theorem~\ref{thrm:surface_matching_conn} and Corollary~\ref{cor:surface_line_matching_conn}]
The matching complex on a surface for the complete graph on $n$
vertices is $(\lfloor \frac{n-2}{3} \rfloor-1)$-connected. For the
linear graph on $n$ vertices it is $(\lfloor \frac{n-2}{4}
\rfloor-1)$-connected. 
\end{theorem*}

The proof for the linear graph is more subtle than that for the
complete graph. It requires new techniques to verify that the
connectivity increases as we build up the complex from a smaller, less
highly connected one. This approach was inspired by discussions with
Andy Putman about a preprint~\cite{putman13} of his. 

In the appendix, Zaremsky uses the connectivity result for matching
complexes on surfaces for linear graphs to produce examples of highly
generating families of subgroups, in the sense of Abels and
Holz~\cite{abels93}, for the pure braid group. 

The Main~Theorem can be viewed as part of a general attempt to
understand how the finiteness properties of a group change when it is
braided. Another instance of this question concerns the \emph{braided
  Houghton groups} $BH_n$. In \cite{degenhardt00} Degenhardt
conjectures that for any~$n$, $BH_n$ is of type $F_{n-1}$ but not of
type $F_n$. He proves this for $n \le 3$, and also proves that~$BH_n$
is of type $\F_2$ for all $n\ge 3$ and of type~$\F_3$ for all $n\ge
4$; see also \cite{funar07}. In the realm of braided Thompson's
groups, Funar and Kapoudjian \cite{funar08, funar11} showed that the
\emph{braided Ptolemy-Thompson groups}~$T^\sharp$ and~$T^*$ are
finitely presented and that $T^*$ is asynchronously combable, and for
that reason of type $\F_3$ and conjectured to be of type~$\F_\infty$. 

\medskip

In Section~\ref{sec:braided_V} we recall the definitions of $\Vbr$ and
$\Fbr$, using the language from \cite{brady08}.  We also introduce
``spraiges'', or ``split-braid-merge diagrams'', along with the
important notion of ``dangling''. The Stein space $X$ is constructed
in Section~\ref{sec:def_stein_space} along with an invariant,
cocompact filtration $(X^{\le n})_n$. In Section~\ref{sec:surfaces}
matching complexes on surfaces are introduced and shown to be highly
connected. These connectivity results are then used in
Section~\ref{sec:desc_link_conn} to show that the filtration~$(X^{\le
  n})_n$ is asymptotically highly connected. Finally we prove the
Main~Theorem in Section~\ref{sec:proof_main_theorem}. In the appendix,
Zaremsky further analyzes matching complexes on surfaces, and related
complexes, to deduce properties of higher generation for pure braid
groups.


\subsection*{Acknowledgments}

We are grateful to Andy Putman for suggesting a new strategy
to handle the complexes $\matcharc(\Gamma)$ in
Section~\ref{sec:surfaces}, and for referring us to his
paper~\cite{putman13}.
We also thank Matt Brin and John Meier for explaining the backstory of
this problem to us, and the anonymous referee for corrections and a helpful comment.
The project was carried out by the research
group C8 of the SFB~701 in Bielefeld, and all five authors are
grateful for the support of the SFB. The fourth and fifth named
authors also gratefully acknowledge support of the SFB~878 in M\"unster.


\section{The braided Thompson's groups}
\label{sec:braided_V}

Thompson's groups $F$ and $V$ have been studied at length, and possess
many unusual and interesting properties.  For example, $F$ is a
torsion-free group of type~$\F_\infty$ with infinite cohomological
dimension, and $V$ is an infinite simple group of type~$\F_\infty$.
An introduction to $F$ and $V$ can be found in~\cite{cannon96}, in
which the \emph{paired tree diagrams} approach to the groups is
discussed.  This is the approach that we will take here as well.  We
will follow the definitions given in~\cite{brady08}, where in addition
the braided Thompson's groups $\Vbr$ and $\Fbr$ are defined in terms
of \emph{braided paired tree diagrams}.  We will also, less formally, 
picture elements of these groups in the language of \emph{strand
diagrams}, as in~\cite{belk07}.

\subsection{The group}\label{sec:braided_V_setup}

We first recall the definition of $V$.  By a \emph{rooted binary tree} we mean
a finite tree such that every vertex has degree $3$, except the
\emph{leaves}, which have degree $1$, and the \emph{root}, which has
degree $2$ (or degree $1$ if the root is also a leaf).  Usually we
draw such trees with the root at the top and the nodes descending from
it, down to the leaves.  A non-leaf node together with the two nodes
directly below it is called a \emph{caret}.  If the leaves of a caret
in $T$ are leaves of $T$, we will call the caret \emph{elementary}.
Note that a rooted binary tree always consists of $(n-1)$ carets and
has $n$ leaves, for some $n\in\N$.

By a \emph{paired tree diagram} we mean a triple $(T_-,\rho,T_+)$
consisting of two rooted binary trees~$T_-$ and~$T_+$ with the same
number of leaves $n$, and a permutation~$\rho\in S_n$.  The leaves
of~$T_-$ are labeled $1,\dots,n$ from left to right, and for each $i$,
the $\rho(i)^{\text{th}}$ leaf of~$T_+$ is labeled $i$.  There is an equivalence
relation on paired tree diagrams given by reductions and expansions.
By a \emph{reduction} we mean the following: Suppose there is an
elementary caret in~$T_-$ with left leaf labeled $i$ and right leaf
labeled~$i+1$, and an elementary caret in~$T_+$ with left leaf labeled~$i$ and right leaf labeled~$i+1$.  Then we can ``reduce'' the diagram
by removing those carets, renumbering the leaves and replacing
$\rho$ with the permutation~$\rho'\in S_{n-1}$ that sends the new leaf
of~$T_-$ to the new leaf of~$T_+$, and otherwise behaves like~$\rho$.
The resulting paired tree diagram~$(T'_-,\rho',T'_+)$ is then said to
be obtained by \emph{reducing}~$(T_-,\rho, T_+)$.  The reverse
operation to reduction is called \emph{expansion}, so $(T_-,\rho,T_+)$
is an expansion of $(T'_-,\rho',T'_+)$.  A paired tree diagram is
called \emph{reduced} if there is no reduction possible.  Thus an
equivalence class of paired tree diagrams consists of all diagrams
having a common reduced representative.  Such reduced representatives
are unique.  See Figure~\ref{fig:reduction_V} for an idea of reduction
of paired tree diagrams.

\begin{figure}[t]
\centering
\begin{tikzpicture}[line width=1pt, scale=0.7]
  \draw
   (2,-2) -- (0,0) -- (-2,-2)
   (-1.5,-1.5) -- (-1,-2)
   (-1,-1) -- (0,-2)
   (1,-2) -- (1.5,-1.5);
  \filldraw 
   (0,0) circle (1.5pt)
   (-1,-1) circle (1.5pt)
   (-1.5,-1.5) circle (1.5pt)
   (-2,-2) circle (1.5pt)
   (1.5,-1.5) circle (1.5pt)
   (2,-2) circle (1.5pt)
   (1,-2) circle (1.5pt)
   (-1,-2) circle (1.5pt)
   (0,-2) circle (1.5pt);
  \node at (-2,-2.5) {$1$};
  \node at (-1,-2.5) {$2$};
  \node at (0,-2.5) {$3$};
  \node at (1,-2.5) {$4$};
  \node at (2,-2.5) {$5$};

  \begin{scope}[xshift=5.5cm, yscale=-1]
   \draw
    (-2,2) -- (0,0) -- (2,2)
    (-0.5,0.5) -- (1,2)    (-0.5,1.5) -- (0,1)
    (-1,2) --(-0.5,1.5) -- (0,2);
   \filldraw
    (0,0) circle (1.5pt)
    (-0.5,0.5) circle (1.5pt)
    (-2,2) circle (1.5pt)
    (0,1) circle (1.5pt)
    (-0.5,1.5) circle (1.5pt)
    (-1,2) circle (1.5pt) 
    (0,2) circle (1.5pt)
    (1,2) circle (1.5pt)
    (2,2) circle (1.5pt);
   \node at (-2,2.5) {$3$};
   \node at (-1,2.5) {$1$};
   \node at (0,2.5) {$2$}; 
   \node at (1,2.5) {$5$};
   \node at (2,2.5) {$4$};
 \end{scope}

  \begin{scope}[xshift=-1cm, yshift=-4cm]
   \draw
   (0,-1.5) -- (1.5,0) -- (3,-1.5)   (0.5,-1) -- (1,-1.5)   (2,-1.5) --
(2.5,-1);
  \filldraw 
   (0,-1.5) circle (1.5pt)
   (1.5,0) circle (1.5pt)
   (3,-1.5) circle (1.5pt)
   (0.5,-1) circle (1.5pt)
   (1,-1.5) circle (1.5pt)
   (2,-1.5) circle (1.5pt)
   (2.5,-1) circle (1.5pt);
   \node at (0,-2) {$1$};
   \node at (1,-2) {$2$};
   \node at (2,-2) {$3$};
   \node at (3,-2) {$4$};
  \end{scope}

  \begin{scope}[xshift=3.5cm, yshift=-5.5cm, yscale=-1]
   \draw
    (0,0) -- (1.5,-1.5) -- (3,0)   (1,-1) -- (2,0)   (1.5,-0.5) -- (1,0);
   \filldraw 
    (0,0) circle (1.5pt)
    (1.5,-1.5) circle (1.5pt)
    (3,0) circle (1.5pt)
    (1,-1) circle (1.5pt)
    (2,0) circle (1.5pt)
    (1.5,-0.5) circle (1.5pt)
    (1,0) circle (1.5pt);
   \node at (0,0.5) {$2$}; 
   \node at (1,0.5) {$1$};
   \node at (2,0.5) {$4$}; 
   \node at (3,0.5) {$3$};
  \end{scope}
\end{tikzpicture}
\caption{Reduction, of the top paired tree diagram to the bottom one.}
\label{fig:reduction_V}
\end{figure}
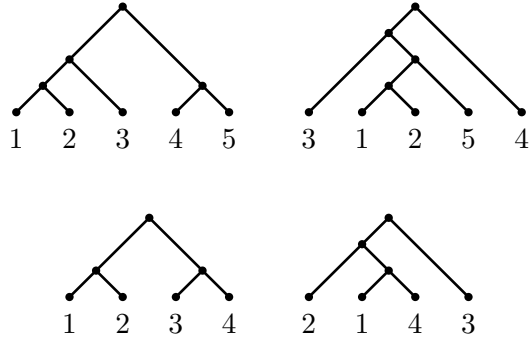

There is a binary operation $\ast$ on the set of equivalence classes
of paired tree diagrams.  Let $T=(T_-,\rho,T_+)$ and $S=(S_-,\xi,S_+)$
be reduced paired tree diagrams.  By applying repeated expansions to
$T$ and $S$ we can find representatives $(T'_-,\rho',T'_+)$ and
$(S'_-,\xi',S'_+)$ of the equivalence classes of $T$ and
$S$, respectively, such that~$T'_+ = S'_-$.  Then we
declare $T\ast S$ to be $(T'_-,\rho'\xi',S'_+)$.  This operation is well defined on the equivalence classes, and is a group operation \cite{brin07,cannon96}.

\begin{definition}
    \label{def:V_F}
    Thompson's group $V$ is the group of equivalence classes of paired
    tree diagrams with the multiplication~$\ast$.  Thompson's group $F$ is
    the subgroup of~$V$ consisting of elements where the permutation
    is the identity.
\end{definition}

To deal with braided Thompson's groups it will become convenient to
have the following picture in mind for paired tree diagrams.  Think of
the tree $T_+$ drawn beneath $T_-$ and upside down, i.e., with the
root at the bottom and the leaves at the top.  The permutation $\rho$
is then indicated by arrows pointing from the leaves of $T_-$ to the
corresponding paired leaf of $T_+$.  See Figure~\ref{fig:element_of_V}
for this visualization of (the unreduced representation of) the
element of $V$ in Figure~\ref{fig:reduction_V}.

\begin{figure}[t]
\centering
\begin{tikzpicture}[line width=1pt]
	\draw
	(2,-2) -- (0,0) -- (-2,-2)
	(-1.5,-1.5) -- (-1,-2)
	(-1,-1) -- (0,-2)
	(1,-2) -- (1.5,-1.5);
	\draw[->](-1.9,-2.1) -> (-1.1,-2.9); \draw[->](-0.9,-2.1) ->
(-0.1,-2.9); \draw[->](-0.1,-2.1) -> (-1.9,-2.9); \draw[->](1.1,-2.1) ->
(1.9,-2.9); \draw[->](1.9,-2.1) -> (1.1,-2.9);
        \filldraw 
	( 0  , 0) circle (1.5pt)
	(-1  ,-1) circle (1.5pt)
	(-1.5,-1.5) circle (1.5pt)
	(-2  ,-2) circle (1.5pt)
	( 1.5,-1.5) circle (1.5pt)
	( 2  ,-2) circle (1.5pt)
	( 1  ,-2) circle (1.5pt)
	(-1  ,-2) circle (1.5pt)
	( 0  ,-2) circle (1.5pt);

	\begin{scope}[yshift=-5cm]
	    \draw
	    (-2,2) -- (0,0) -- (2,2)
	    (-0.5,0.5) -- (1,2)    (-0.5,1.5) -- (0,1)
	    (-1,2) --(-0.5,1.5) -- (0,2);
	    \filldraw
	    ( 0  ,0) circle (1.5pt)
	    (-0.5  ,0.5) circle (1.5pt)
	    (-2  ,2) circle (1.5pt)
	    ( 0  ,1) circle (1.5pt)
	    (-0.5,1.5) circle (1.5pt)
	    (-1  ,2) circle (1.5pt) 
	    ( 0  ,2) circle (1.5pt)
	    ( 1  ,2) circle (1.5pt)
	    ( 2  ,2) circle (1.5pt);
	\end{scope}
\end{tikzpicture}
\caption{An element of $V$.}
\label{fig:element_of_V}
\end{figure}
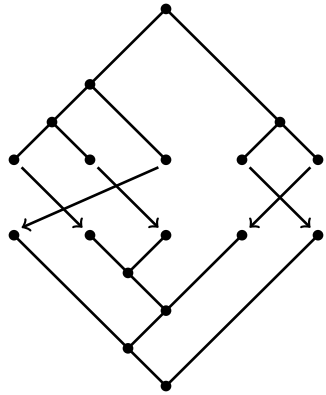

In the braided version $\Vbr$ of $V$, the permutations of leaves are
replaced by braids between the leaves.  Again following~\cite{brady08}
we will first introduce \emph{braided paired tree diagrams} and then
copy the construction of $V$ given above to define $\Vbr$.  Then we
will mention how things change for~$\Fbr$.

\begin{definition}
    A \emph{braided paired tree diagram} is a triple $(T_-,b,T_+)$
    consisting of two rooted binary trees $T_-$ and $T_+$ with the
    same number of leaves $n$ and a braid~$b \in B_n$.
\end{definition}

We draw braided paired tree diagrams with $T_+$ upside down and below
$T_-$, and the strands of the braid connecting leaves.  This is
analogous to the visualization of paired tree diagrams in
Figure~\ref{fig:element_of_V}, and examples of braided paired tree
diagrams can be seen in Figure~\ref{fig:reduction_Vbr}.

As with $V$, we can define an equivalence relation on the set of
braided paired tree diagrams using the notions of reduction and
expansion.  It is easier to first define expansion and then take
reduction as the reverse of expansion.  Let $\rho_b\in S_n$ denote the
permutation corresponding to the braid $b\in B_n$.  Let $(T_-,b,T_+)$
be a braided paired tree diagram.  Label the leaves of $T_-$ from $1$
to $n$, left to right, and for each $i$ label the
$\rho_b(i)^{\text{th}}$ leaf of~$T_+$ by~$i$.  By the~$i^{\text{th}}$
strand of the braid we will always mean the strand that begins at
the~$i^{\text{th}}$ leaf of $T_-$, i.e., we count the strands from the
top.  An \emph{expansion} of $(T_-,b,T_+)$ amounts to the following.
For some $1\le i\le n$, replace~$T_\pm$ with trees $T_\pm'$ obtained from
$T_\pm$ by adding a caret to the leaf labeled~$i$.  Then replace $b$
with a braid $b' \in B_{n+1}$, obtained by ``doubling'' the
$i^{\text{th}}$ strand of $b$.  The triple $(T_-',b',T_+')$ is an
\emph{expansion} of $(T_-,b,T_+)$. As with paired tree diagrams, 
\emph{reduction} is the reverse of expansion, so $(T_-,b,T_+)$ is a reduction of
$(T_-',b',T_+')$.  See
Figure~\ref{fig:reduction_Vbr} for an idea of reduction of braided
paired tree diagrams.

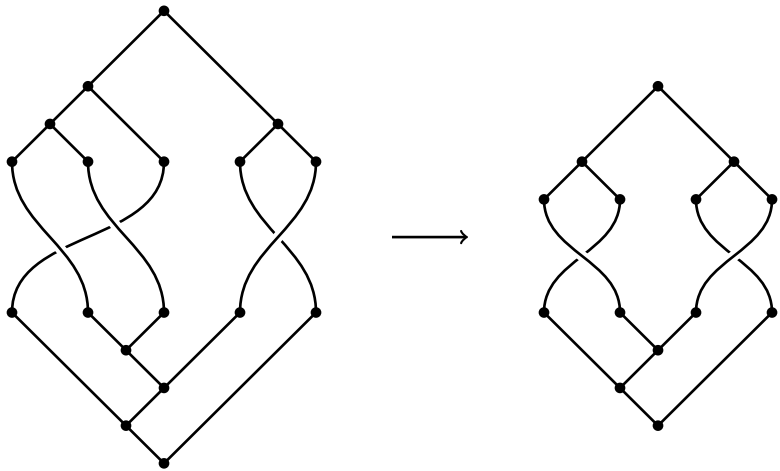
\begin{figure}
\centering
\begin{tikzpicture}[line width=1pt]
  \draw
   (0,-2) -- (2,0) -- (4,-2)   (0.5,-1.5) -- (1,-2)   (1,-1) -- (2,-2)  
(3.5,-1.5) -- (3,-2)
   (2,-2) to [out=-90, in=90] (0,-4)   (3,-2) to [out=-90, in=90] (4,-4);
  \draw[white, line width=4pt]
   (0,-2) to [out=-90, in=90] (1,-4)   (1,-2) to [out=-90, in=90] (2,-4)  
(4,-2) to [out=-90, in=90] (3,-4);
  \draw
   (0,-2) to [out=-90, in=90] (1,-4)   (1,-2) to [out=-90, in=90] (2,-4)  
(4,-2) to [out=-90, in=90] (3,-4);
  \draw
   (0,-4) -- (2,-6) -- (4,-4)   (1.5,-5.5) -- (3,-4)   (1,-4) -- (2,-5)   (2,-4)
-- (1.5,-4.5);

  \filldraw
   (0,-2) circle (1.5pt)   (2,0) circle (1.5pt)   (4,-2) circle (1.5pt)  
(0.5,-1.5) circle (1.5pt)   (1,-2) circle (1.5pt)   (1,-1) circle (1.5pt)  
(2,-2) circle (1.5pt)   (3.5,-1.5) circle (1.5pt)   (3,-2) circle (1.5pt);
  \filldraw
   (0,-4) circle (1.5pt)   (2,-6) circle (1.5pt)   (4,-4) circle (1.5pt)  
(1.5,-5.5) circle (1.5pt)   (3,-4) circle (1.5pt)   (1,-4) circle (1.5pt)  
(2,-5) circle (1.5pt)   (2,-4) circle (1.5pt)   (1.5,-4.5) circle (1.5pt);
  \draw[->]
   (5,-3) -> (6,-3);

 \begin{scope}[xshift=7cm,yshift=-1cm]
  \draw
   (0,-1.5) -- (1.5,0) -- (3,-1.5)   (0.5,-1) -- (1,-1.5)   (2,-1.5) -- (2.5,-1)
   (1,-1.5) to [out=-90, in=90] (0,-3)   (2,-1.5) to [out=-90, in=90] (3,-3);
  \draw[white, line width=4pt]
   (0,-1.5) to [out=-90, in=90] (1,-3)   (3,-1.5) to [out=-90, in=90] (2,-3);
  \draw
   (0,-1.5) to [out=-90, in=90] (1,-3)   (3,-1.5) to [out=-90, in=90] (2,-3);
  \draw
   (0,-3) -- (1.5,-4.5) -- (3,-3)   (1,-4) -- (2,-3)   (1.5,-3.5) -- (1,-3);

  \filldraw
   (0,-1.5) circle (1.5pt)   (1.5,0) circle (1.5pt)   (3,-1.5) circle (1.5pt)  
(0.5,-1) circle (1.5pt)   (1,-1.5) circle (1.5pt)   (2,-1.5) circle (1.5pt)  
(2.5,-1) circle (1.5pt);
  \filldraw
   (0,-3) circle (1.5pt)   (1.5,-4.5) circle (1.5pt)   (3,-3) circle (1.5pt)  
(1,-4) circle (1.5pt)   (2,-3) circle (1.5pt)   (1.5,-3.5) circle (1.5pt)  
(1,-3) circle (1.5pt);
 \end{scope}
\end{tikzpicture}
\caption{Reduction of braided paired tree diagrams.}
\label{fig:reduction_Vbr}
\end{figure}

Two braided paired tree diagrams are equivalent if one is obtained
from the other by a sequence of reductions or expansions.  The
multiplication operation~$\ast$ on the equivalence classes is defined
the same way as for~$V$. It is a well defined group operation \cite{brin07}.

\begin{definition}
    The braided Thompson's group $\Vbr$ is the group of equivalence
    classes of braided paired tree diagrams with the
    multiplication~$\ast$.
\end{definition}

A convenient way to visualize multiplication in~$\Vbr$ is via ``stacking'' braided paired tree diagrams. For $g,h\in \Vbr$, each pictured as a tree-braid-tree as before,~$g\ast h$ is obtained by attaching the top of $h$ to the bottom of~$g$ and then reducing the picture via certain moves. We indicate four of these moves in Figure~\ref{fig:reduction_moves}. A ``merge'' followed immediately by a ``split'', or a split followed immediately by a merge, is equivalent to doing nothing, as seen in the top two pictures. Also, splits and merges interact with braids in ways indicated by the bottom two pictures. We leave it to the reader to further inspect the details of this visualization of multiplication in~$\Vbr$. This is closely related to the \emph{strand diagram} model for Thompson's groups in \cite{belk07}. See also Section~1.2 in~\cite{brin07} and Figure~2 of~\cite{burillo09}.

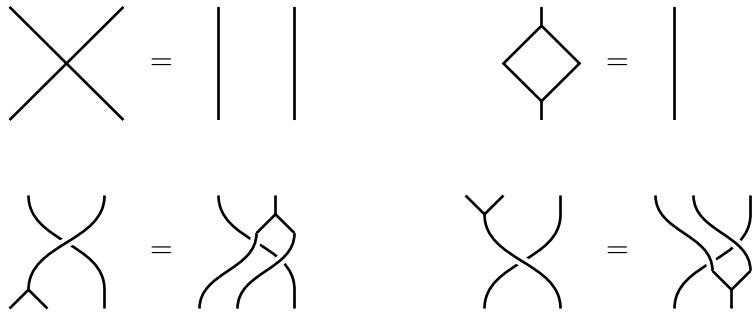
\begin{figure}[b]
\centering
\begin{tikzpicture}[line width=1pt]

  \draw
   (0,0.75) -- (1.5,-0.75)
   (1.5,0.75) -- (0,-0.75);
  \node at (2,0) {$=$};
   
 \begin{scope}[xshift=2.75cm]
  \draw
   (0,0.75) -- (0,-0.75)
   (1,0.75) -- (1,-0.75);
 \end{scope}
 \begin{scope}[xshift=6.5cm]
  \draw
   (0.5,0.75) -- (0.5,0.5) -- (0,0) -- (0.5,-0.5) -- (1,0) -- (0.5,0.5)  
(0.5,-0.5) -- (0.5,-0.75);
  \node at (1.5,0) {$=$};
  \draw
   (2.25,0.75) -- (2.25,-0.75);
 \end{scope}
 \begin{scope}[yshift=-2.5cm]
  \draw
   (0.25,0.75) to [out=-90, in=90, looseness=1] (1.25,-0.5) -- (1.25,-0.75);
  \draw[white, line width=4pt]
   (1.25,0.75) to [out=-90, in=90, looseness=1] (0.25,-0.5);
  \draw
   (1.25,0.75) to [out=-90, in=90, looseness=1] (0.25,-0.5);
  \draw
   (0,-0.75) -- (0.25,-0.5) -- (0.5,-0.75);
  \node at (2,0) {$=$};
 \end{scope}
 
 \begin{scope}[yshift=-2.5cm, xshift=2.5cm]
  \draw
   (0.25,0.75) to [out=-90, in=90, looseness=1] (1.25,-0.5) -- (1.25,-0.75);
  \draw
   (0.75,0.25) -- (1,0.5) -- (1.25,0.25)   (1,0.5) -- (1,0.75);
  \draw[white, line width=4pt]
   (0.75,0.25) to [out=-90, in=90, looseness=1] (0,-0.75)
   (1.25,0.25) to [out=-90, in=90, looseness=1] (0.5,-0.75);
  \draw
   (0.75,0.25) to [out=-90, in=90, looseness=1] (0,-0.75)
   (1.25,0.25) to [out=-90, in=90, looseness=1] (0.5,-0.75);
 \end{scope}
 \begin{scope}[yshift=-2.5cm, xshift=6cm, yscale=-1]
  \draw
   (0.25,0.75) to [out=-90, in=90, looseness=1] (1.25,-0.5) -- (1.25,-0.75);
  \draw[white, line width=4pt]
   (1.25,0.75) to [out=-90, in=90, looseness=1] (0.25,-0.5);
  \draw
   (1.25,0.75) to [out=-90, in=90, looseness=1] (0.25,-0.5);
  \draw
   (0,-0.75) -- (0.25,-0.5) -- (0.5,-0.75);
  \node at (2,0) {$=$};
 \end{scope}
 
 \begin{scope}[yshift=-2.5cm, xshift=8.5cm, yscale=-1]
  \draw
   (0.25,0.75) to [out=-90, in=90, looseness=1] (1.25,-0.5) -- (1.25,-0.75);
  \draw
   (0.75,0.25) -- (1,0.5) -- (1.25,0.25)   (1,0.5) -- (1,0.75);
  \draw[white, line width=4pt]
   (0.75,0.25) to [out=-90, in=90, looseness=1] (0,-0.75)
   (1.25,0.25) to [out=-90, in=90, looseness=1] (0.5,-0.75);
  \draw
   (0.75,0.25) to [out=-90, in=90, looseness=1] (0,-0.75)
   (1.25,0.25) to [out=-90, in=90, looseness=1] (0.5,-0.75);
 \end{scope}
\end{tikzpicture}
\caption{Moves to reduce braided paired tree diagrams after stacking.}
\label{fig:reduction_moves}
\end{figure}

From now on we will just refer to the braided paired tree diagrams as
being the elements of~$\Vbr$, though one should keep in mind that the
elements are actually equivalence classes under the reduction and
expansion operations.

We can also define $\Fbr$ as a subgroup of~$\Vbr$.  Recall that a braid $b$ is called
\emph{pure} if~$\rho_b=\id$.  The elements of~$\Fbr$ are the (equivalence classes of) diagrams where the braid is pure.  The fact that~$\Vbr$ and $\Fbr$ are finitely presented has been known for some time,
and explicit finite presentations are given in~\cite{brin07}
and~\cite{brady08}.  Our current goal is to inspect their higher
finiteness properties, though first we will need some more
language.  We now introduce a class of diagrams that will be used
throughout the rest of this paper.

\subsection{A general class of diagrams}
\label{sec:general_diagrams}

To define the spaces we will use, we need a broader class of diagrams
that generalizes braided paired tree diagrams, namely we will consider
forests instead of trees.  We will also continue to informally use the notion
of strand diagrams.  Here a forest will
always mean a finite linearly ordered union of binary rooted trees.  Given a braided paired
tree diagram $(T_-,b,T_+)$ we call a caret in~$T_-$ a
\emph{split}.  Similarly a \emph{merge} is a caret in $T_+$.
With this terminology, we can call the picture representing the
braided paired tree diagram a \emph{split-braid-merge diagram}, abbreviated
\emph{spraige}.  That is, we first picture one strand splitting up into $n$
strands in a certain way, representing $T_-$.  Then the $n$ strands braid with
each other, representing $b$, and finally according to $T_+$ we merge the
strands back together.  These special kinds of diagrams will also be
called~\emph{$(1,1)$-spraiges}.  More generally:

\begin{definition}[Spraiges]
\label{def:spraiges}
An \emph{$(n,m)$-spraige} is a spraige that begins
on $n$ strands, the \emph{heads}, and ends on $m$ strands, the
\emph{feet}.  As indicated above we can equivalently think of an
$(n,m)$-spraige as a \emph{braided paired forest diagram} $(F_-, b,
F_+)$, where~$F_-$ has $n$ roots,~$F_+$ has $m$ roots and both have
the same number of leaves.  By an~$n$-\emph{spraige} we mean an
$(n,m)$-spraige for some $m$, and by a \emph{spraige} we mean an~$(n,m)$-spraige
for some $n$ and $m$.  Let $\spraige$ denote the set of all spraiges,~$\spraige_{n,m}$ the set of all $(n,m)$-spraiges, and $\spraige_n$ the set of
all $n$-spraiges.
\end{definition}

Note that an $n$-spraige has $n$ heads, but can have any number of
feet.  A function that will be important in what follows is the
``number of feet'' function, which we define as
$f\colon\spraige\to\N$ given by $f(\sigma)=m$ if
$\sigma\in\spraige_{n,m}$ for some $n$.

The pictures in Figure~\ref{fig:spraiges-multiplication} are examples
of spraiges. One can generalize the notion of reduction and expansion of such
diagrams to arbitrary spraiges, and consider equivalence classes under reduction
and expansion.  Each such class has a unique reduced representative, as was the
case for paired tree diagrams and braided paired tree diagrams.  We will just
call an equivalence class of spraiges a spraige, so in particular the elements
of $\Vbr$ are $(1,1)$-spraiges.

\begin{figure}[t]
\centering
\begin{tikzpicture}[line width=1pt,scale=0.7]
  
  \draw
   (0,0) -- (0,-1)   (3,0) -- (3,-1)   (4,0) -- (4,-3.5)
   (1,-1) -- (1.5,0) -- (2,-1);
  \draw
   (3,-1) to [out=270, in=90] (1,-3.5);
  \draw[line width=4pt, white]
   (0,-1) to [out=270, in=90] (2,-3.5);
  \draw
   (0,-1) to [out=270, in=90] (2,-3.5);
  \draw[line width=4pt, white]
   (1,-1) to [out=270, in=90] (0,-3.5);
  \draw
   (1,-1) to [out=270, in=90] (0,-3.5);
  \draw[line width=4pt, white]
   (2,-1) to [out=270, in=90] (3,-3.5);
  \draw
   (2,-1) to [out=270, in=90] (3,-3.5) -- (3.5,-4)
   (2,-3.5) -- (3,-4.5) -- (4,-3.5)
   (0,-3.5) -- (0.5,-4.5) -- (1,-3.5);
  \node at (4.5,-2.25) {$\ast$};
  \filldraw
   (0,0) circle (1.5pt)   (1.5,0) circle (1.5pt)   (3,0) circle (1.5pt)   (4,0)
circle (1.5pt)   (0,-1) circle (1.5pt)   (1,-1) circle (1.5pt)   (2,-1) circle
(1.5pt)   (3,-1) circle (1.5pt)   (4,-1) circle (1.5pt)   (0,-3.5) circle
(1.5pt)   (1,-3.5) circle (1.5pt)   (2,-3.5) circle (1.5pt)   (3,-3.5) circle
(1.5pt)   (4,-3.5) circle (1.5pt)   (3.5,-4) circle (1.5pt)   (0.5,-4.5) circle
(1.5pt)   (3,-4.5) circle (1.5pt);

 \begin{scope}[xshift=5cm]
  \draw
   (0,-1) -- (0.5,0) -- (1,-1)
   (2,-1) -- (2.5,0) -- (3,-1) -- (3,-4.5)
   (1,-1) to [out=270, in=90] (2,-3.5);
  \draw[line width=4pt, white]
   (2,-1) to [out=270, in=90] (0,-3.5);
  \draw
   (2,-1) to [out=270, in=90] (0,-3.5);
  \draw[line width=4pt, white]
   (0,-1) to [out=270, in=90] (1,-3.5);
  \draw
   (0,-1) to [out=270, in=90] (1,-3.5)
   (1,-3.5) -- (1.5,-4.5) -- (2,-3.5)   (0,-3.5) -- (0,-4.5);
  \node at (3.75,-2.25) {$=$};
  \filldraw
   (0.5,0) circle (1.5pt)   (2.5,0) circle (1.5pt)   (0,-1) circle (1.5pt)  
(1,-1) circle (1.5pt)   (2,-1) circle (1.5pt)   (3,-1) circle (1.5pt)   (0,-3.5)
circle (1.5pt)   (1,-3.5) circle (1.5pt)   (2,-3.5) circle (1.5pt)   (3,-3.5)
circle (1.5pt)   (0,-4.5) circle (1.5pt)   (1.5,-4.5) circle (1.5pt)   (3,-4.5)
circle (1.5pt);
 \end{scope}

 \begin{scope}[xshift=9.5cm]
  \draw
   (0.2,0) -- (0.2,-4.5)   (3,0) -- (3,-1)   (4,0) -- (4,-3.5)
   (1,-3.5) -- (1,-1) -- (1.5,0) -- (2,-1)
   (3,-1) to [out=270, in=90] (2,-3.5);
  \draw[line width=4pt, white]
   (2,-1) to [out=270, in=90] (3,-3.5);
  \draw
   (2,-1) to [out=270, in=90] (3,-3.5)
   (1,-3.5) -- (1.5,-4.5) -- (2,-3.5)
   (3,-3.5) -- (3.5,-4.5) -- (4,-3.5);
  \filldraw
   (0.2,0) circle (1.5pt)   (1.5,0) circle (1.5pt)   (3,0) circle (1.5pt)  
(4,0) circle (1.5pt)   (0.2,-1) circle (1.5pt)   (1,-1) circle (1.5pt)   (2,-1)
circle (1.5pt)   (3,-1) circle (1.5pt)   (4,-1) circle (1.5pt)   (0.2,-3.5)
circle (1.5pt)   (1,-3.5) circle (1.5pt)   (2,-3.5) circle (1.5pt)   (3,-3.5)
circle (1.5pt)   (4,-3.5) circle (1.5pt)   (0.2,-4.5) circle (1.5pt)  
(1.5,-4.5) circle (1.5pt)   (3.5,-4.5) circle (1.5pt);
 \end{scope}
\end{tikzpicture}
\caption{Multiplication of spraiges.}
\label{fig:spraiges-multiplication}
\end{figure}
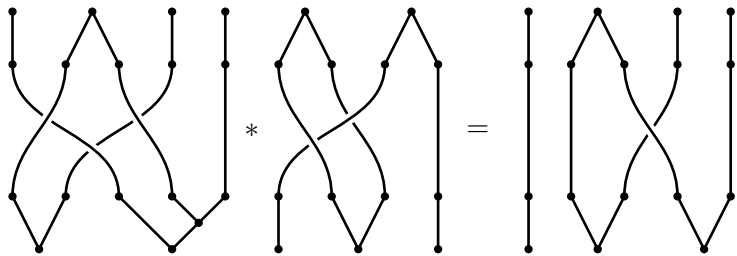

The operation $\ast$ defined for $\Vbr$ can be defined in general
for spraiges, via concatenation of diagrams.  It is only defined for
certain pairs of spraiges, namely we can multiply
$\sigma_1\ast\sigma_2$ for $\sigma_1\in\spraige_{n_1,m_1}$ and
$\sigma_2 \in \spraige_{n_2,m_2}$ if and only if $m_1=n_2$.  In this
case we obtain $\sigma_1 \ast \sigma_2 \in \spraige_{n_1,m_2}$.  The
reader may find it helpful to work out why the multiplication in
Figure~\ref{fig:spraiges-multiplication} holds. As a remark, in the figures, a single-node tree will sometimes be elongated to an edge, for aesthetic reasons.

\begin{remark}
\begin{itemize}
\item[(i)] For every $n\in \N$ there is an identity
$(n,n)$-spraige $1_n$ with respect to $\ast$, namely the spraige
represented by $(1_n,\id,1_n)$.  Here, by abuse of notation, $1_n$
also denotes the trivial forest with $n$ roots.

\item[(ii)] For every $(n,m)$-spraige $(F_-,b,F_+)$ there exists an
inverse $(m,n)$-spraige $(F_+,b^{-1},F_-)$, in the sense that
\begin{align*}
(F_-,b,F_+) \ast (F_+,b^{-1},F_-) & = 1_n \\
\intertext{and}
(F_+,b^{-1},F_-)\ast(F_-,b,F_+) & = 1_m \text{\,.}
\end{align*}

\item[(iii)] $\spraige$ is a groupoid.
\end{itemize}
\end{remark}

\medskip

There are certain forests that will be fundamental to the construction
of the Stein space~$X$ in Section~\ref{sec:def_stein_space}.  For
$n\in \N$ and $J \subseteq \{1,\dots, n\}$, define $F^{(n)}_J$
to be the forest with $n$ roots and $|J|$ carets, with a caret
attached to the $i^{\text{th}}$ root for each $i\in J$.  These forests
are characterized by the property that every caret is elementary, and
we will call any such forest \emph{elementary}.  Define the spraige
$\lambda^{(n)}_J$ to be the $(n,n+|J|)$-spraige $(F^{(n)}_J, \id,
1_{n+|J|})$, and the spraige $\mu^{(n)}_J$ to be its inverse.  If $J =
\{i\}$ write $F^{(n)}_i$, $\lambda^{(n)}_i$ and $\mu^{(n)}_i$ instead.
See Figure~\ref{fig:elem_forests} for an example of an elementary
forest and the corresponding spraiges.

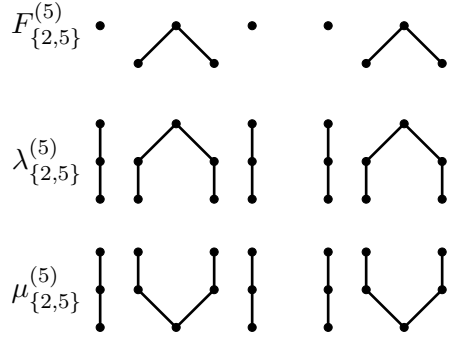
\begin{figure}[t]
\centering
\begin{tikzpicture}
  \draw[line width=1pt]
   (0.5,-0.5) -- (1,0) -- (1.5,-0.5)
   (3.5,-0.5) -- (4,0) -- (4.5,-0.5);
  \filldraw 
   (0,0) circle (1.5pt)
   (1,0) circle (1.5pt)
   (2,0) circle (1.5pt)
   (3,0) circle (1.5pt)
   (4,0) circle (1.5pt)

   (0.5  ,-0.5) circle (1.5pt)
   (1.5  ,-0.5) circle (1.5pt)
   (3.5  ,-0.5) circle (1.5pt)
   (4.5  ,-0.5) circle (1.5pt);
	
  \node at (-0.7,0) {$F_{\{2,5\}}^{(5)}$};

 \begin{scope}[yshift=-1.3cm]
  \draw[line width=1pt]
   (0.5,-1) -- (0.5,-0.5) -- (1,0) -- (1.5,-0.5) -- (1.5,-1)
   (3.5,-1) -- (3.5,-0.5) -- (4,0) -- (4.5,-0.5) -- (4.5,-1)
   (0,0) -- (0,-1)   (2,0) -- (2,-1)   (3,0) -- (3,-1);
  \filldraw 
   (0,0) circle (1.5pt)   (0,-0.5) circle (1.5pt)   (0,-1) circle (1.5pt)
   (1,0) circle (1.5pt)
   (2,0) circle (1.5pt)   (2,-0.5) circle (1.5pt)   (2,-1) circle (1.5pt)
   (3,0) circle (1.5pt)   (3,-0.5) circle (1.5pt)   (3,-1) circle (1.5pt)
   (4,0) circle (1.5pt)

   (0.5,-0.5) circle (1.5pt)   (0.5,-1) circle (1.5pt)
   (1.5,-0.5) circle (1.5pt)   (1.5,-1) circle (1.5pt)
   (3.5,-0.5) circle (1.5pt)   (3.5,-1) circle (1.5pt)
   (4.5,-0.5) circle (1.5pt)   (4.5,-1) circle (1.5pt);
	
  \node at (-0.7,-0.5) {$\lambda_{\{2,5\}}^{(5)}$};
 \end{scope}

 \begin{scope}[yshift=-4cm, yscale=-1]
  \draw[line width=1pt]
   (0.5,-1) -- (0.5,-0.5) -- (1,0) -- (1.5,-0.5) -- (1.5,-1)
   (3.5,-1) -- (3.5,-0.5) -- (4,0) -- (4.5,-0.5) -- (4.5,-1)
   (0,0) -- (0,-1)   (2,0) -- (2,-1)   (3,0) -- (3,-1);
  \filldraw 
   (0,0) circle (1.5pt)   (0,-0.5) circle (1.5pt)   (0,-1) circle (1.5pt)
   (1,0) circle (1.5pt)
   (2,0) circle (1.5pt)   (2,-0.5) circle (1.5pt)   (2,-1) circle (1.5pt)
   (3,0) circle (1.5pt)   (3,-0.5) circle (1.5pt)   (3,-1) circle (1.5pt)
   (4,0) circle (1.5pt)

   (0.5,-0.5) circle (1.5pt)   (0.5,-1) circle (1.5pt)
   (1.5,-0.5) circle (1.5pt)   (1.5,-1) circle (1.5pt)
   (3.5,-0.5) circle (1.5pt)   (3.5,-1) circle (1.5pt)
   (4.5,-0.5) circle (1.5pt)   (4.5,-1) circle (1.5pt);
	
  \node at (-0.7,-0.5) {$\mu_{\{2,5\}}^{(5)}$};
 \end{scope}
\end{tikzpicture}
\caption{The elementary forest $F^{(5)}_{\{2,5\}}$, and the spraiges
$\lambda^{(5)}_{\{2,5\}}$ and $\mu^{(5)}_{\{2,5\}}$.}
\label{fig:elem_forests}
\end{figure}

Fix an $(n,m)$-spraige $\sigma$.  For any forest~$F$ with~$m$ roots
and~$l$ leaves define the \emph{splitting of $\sigma$ by $F$} as
multiplying $\sigma$ by the spraige $(F,\id,1_l)$ from the right.
Similarly a \emph{merging of $\sigma$ by $F'$} is right multiplication
by the spraige $(1_m,\id,F')$, where $F'$ now has $l$ roots and $m$
leaves.  In the case where $F$ (respectively~$F'$) is an elementary
forest, we call this operation \emph{elementary splitting}
(respectively~\emph{elementary merging)}.  See
Figure~\ref{fig:spraige_splitting} for an idea of splitting and
Figure~\ref{fig:spraige_merging} for an idea of elementary merging.

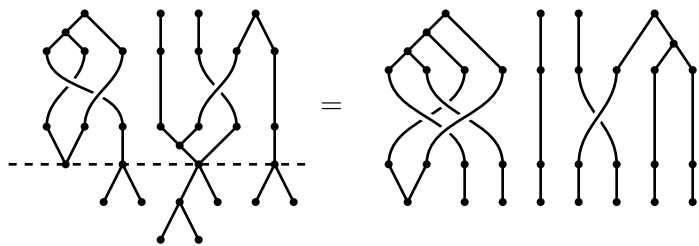
\begin{figure}[t]
\centering
\begin{tikzpicture}[line width=1pt]
  
  \draw
   (0,-0.5) -- (0.5,0) -- (1,-0.5)   (0.25,-0.25) -- (0.5,-0.5)   (1.5,0) --
(1.5,-1.5)   (2,0) -- (2,-0.5)   (2.5,-0.5) -- (2.75,0) -- (3,-0.5) -- (3,-2)
   (0.5,-0.5) to [out=-90, in=90] (0,-1.5)   (2,-0.5) to [out=-90, in=90]
(2.5,-1.5);

  \draw[white, line width=4pt]
   (0,-0.5) to [out=-90, in=90] (1,-1.5)   (2.5,-0.5) to [out=-90, in=90]
(2,-1.5);
  \draw
   (0,-0.5) to [out=-90, in=90] (1,-1.5)   (2.5,-0.5) to [out=-90, in=90]
(2,-1.5);
  \draw[white, line width=4pt]
   (1,-0.5) to [out=-90, in=90] (0.5,-1.5);
  \draw
   (1,-0.5) to [out=-90, in=90] (0.5,-1.5);

  \draw
   (0,-1.5) -- (0.25,-2) -- (0.5,-1.5)   (1,-1.5) -- (1,-2)   (1.5,-1.5) --
(2,-2) -- (2.5,-1.5)   (2,-1.5) -- (1.75,-1.75);

  \draw[dashed]
   (-0.5,-2) -- (3.5,-2);

  \draw
   (0.75,-2.5) -- (1,-2) -- (1.25,-2.5)   (1.75,-2.5) -- (2,-2) -- (2.25,-2.5)  
(2.75,-2.5) -- (3,-2) -- (3.25,-2.5)   (1.5,-3) -- (1.75,-2.5) -- (2,-3);

  \node at (3.75,-1.25) {$=$};

  \filldraw
   (0,-0.5) circle (1pt)   (0.5,-0.5) circle (1pt)   (1,-0.5) circle (1pt)  
(1.5,-0.5) circle (1pt)   (2,-0.5) circle (1pt)   (2.5,-0.5) circle (1pt)  
(3,-0.5) circle (1pt)
   (0.25,-0.25) circle (1pt)   (0.5,0) circle (1pt)   (1.5,0) circle (1pt)  
(2,0) circle (1pt)   (2.75,0) circle (1pt)
   (0,-1.5) circle (1pt)   (0.5,-1.5) circle (1pt)   (1,-1.5) circle (1pt)  
(1.5,-1.5) circle (1pt)   (2,-1.5) circle (1pt)   (2.5,-1.5) circle (1pt)  
(3,-1.5) circle (1pt)
   (0.25,-2) circle (1pt)   (1,-2) circle (1pt)   (1.75,-1.75) circle (1pt)  
(2,-2) circle (1pt)   (3,-2) circle (1pt)
   (0.75,-2.5) circle (1pt)   (1.25,-2.5) circle (1pt)   (1.5,-3) circle (1pt)  
(1.75,-2.5) circle (1pt)   (2,-3) circle (1pt)   (2.25,-2.5) circle (1pt)  
(2.75,-2.5) circle (1pt)   (3.25,-2.5) circle (1pt);

 \begin{scope}[xshift=4.5cm]
  \draw
   (0,-0.75) -- (0.75,0) -- (1.5,-0.75)   (0.5,-0.75) -- (0.25,-0.5)   (1,-0.75)
-- (0.5,-0.25)
   (2,0) -- (2,-2)   (2.5,0) -- (2.5,-0.75)   (3,-0.75) -- (3.5,0) -- (4,-0.75)
-- (4,-2)   (3.5,-2) -- (3.5,-0.75) -- (3.75,-0.4)
   (1,-0.75) to [out=-90, in=90] (0,-2)   (2.5,-0.75) to [out=-90, in=90]
(3,-2);

  \draw[white, line width=4pt]
   (0,-0.75) to [out=-90, in=90] (1,-2)   (0.5,-0.75) to [out=-90, in=90]
(1.5,-2)   (3,-0.75) to [out=-90, in=90] (2.5,-2);
  \draw
   (0,-0.75) to [out=-90, in=90] (1,-2)   (0.5,-0.75) to [out=-90, in=90]
(1.5,-2)   (3,-0.75) to [out=-90, in=90] (2.5,-2);
  \draw[white, line width=4pt]
   (1.5,-0.75) to [out=-90, in=90] (0.5,-2);
  \draw
   (1.5,-0.75) to [out=-90, in=90] (0.5,-2);

  \draw
   (0,-2) -- (0.25,-2.5) -- (0.5,-2)   (1,-2) -- (1,-2.5)   (1.5,-2) --
(1.5,-2.5)   (2,-2) -- (2,-2.5)   (2.5,-2) -- (2.5,-2.5)   (3,-2) -- (3,-2.5)  
(3.5,-2) -- (3.5,-2.5)   (4,-2) -- (4,-2.5);

  \filldraw
   (0,-0.75) circle (1pt)   (0.5,-0.75) circle (1pt)   (1,-0.75) circle (1pt)  
(1.5,-0.75) circle (1pt)   (2,-0.75) circle (1pt)   (2.5,-0.75) circle (1pt)  
(3,-0.75) circle (1pt)   (3.5,-0.75) circle (1pt)   (4,-0.75) circle (1pt)
   (0.25,-0.5) circle (1pt)   (0.5,-0.25) circle (1pt)   (0.75,0) circle (1pt)  
(2,0) circle (1pt)   (2.5,0) circle (1pt)   (3.5,0) circle (1pt)   (3.75,-0.4)
circle (1pt)
   (0,-2) circle (1pt)   (0.5,-2) circle (1pt)   (1,-2) circle (1pt)   (1.5,-2)
circle (1pt)   (2,-2) circle (1pt)   (2.5,-2) circle (1pt)   (3,-2) circle (1pt)
  (3.5,-2) circle (1pt)   (4,-2) circle (1pt)
   (0.25,-2.5) circle (1pt)   (1,-2.5) circle (1pt)   (1.5,-2.5) circle (1pt)  
(2,-2.5) circle (1pt)   (2.5,-2.5) circle (1pt)   (3,-2.5) circle (1pt)  
(3.5,-2.5) circle (1pt)   (4,-2.5) circle (1pt);
 \end{scope}
\end{tikzpicture}
\caption{A splitting of a spraige.}
\label{fig:spraige_splitting}
\end{figure}

\begin{figure}[t]
\centering
\begin{tikzpicture}[line width=1pt]
  
  \draw
   (0,-0.5) -- (0.25,0) -- (0.5,-0.5)   (1,0) -- (1,-0.5)   (1.5,-0.5) --
(1.75,0) -- (2,-0.5)   (2.5,0) -- (2.5,-0.5)
   (0.5,-0.5) to [out=-90, in=90] (0,-1.5)   (1.5,-0.5) to [out=-90, in=90]
(2,-1.5)   (2,-0.5) to [out=-90, in=90] (2.5,-1.5);

  \draw[white, line width=4pt]
   (0,-0.5) to [out=-90, in=90] (1,-1.5)   (2.5,-0.5) to [out=-90, in=90]
(1.5,-1.5);
  \draw
   (0,-0.5) to [out=-90, in=90] (1,-1.5)   (2.5,-0.5) to [out=-90, in=90]
(1.5,-1.5);
  \draw[white, line width=4pt]
   (1,-0.5) to [out=-90, in=90] (0.5,-1.5);
  \draw
   (1,-0.5) to [out=-90, in=90] (0.5,-1.5);

  \draw
   (0,-1.5) -- (0,-2)   (0.5,-1.5) -- (1,-2) -- (1.5,-1.5)   (1,-1.5) --
(1.25,-1.75)   (2,-1.5) -- (2,-2)   (2.5,-1.5) -- (2.5,-2);

  \draw[dashed]
   (-0.5,-2) -- (3,-2);

  \draw
   (0,-2) -- (0.5,-2.5) -- (1,-2)   (2,-2) -- (2.25,-2.5) -- (2.5,-2);

  \node at (3.25,-1.25) {$=$};

  \filldraw
   (0,-0.5) circle (1pt)   (0.25,0) circle (1pt)   (0.5,-0.5) circle (1pt)  
(1,0) circle (1pt)   (1,-0.5) circle (1pt)   (1.5,-0.5) circle (1pt)   (1.75,0)
circle (1pt)   (2,-0.5) circle (1pt)   (2.5,0) circle (1pt)   (2.5,-0.5) circle
(1pt)
   (0,-1.5) circle (1pt)   (0,-2) circle (1pt)   (0.5,-1.5) circle (1pt)  
(1,-2) circle (1pt)   (1.5,-1.5) circle (1pt)   (1,-1.5) circle (1pt)  
(1.25,-1.75) circle (1pt)   (2,-1.5) circle (1pt)   (2,-2) circle (1pt)  
(2.5,-1.5) circle (1pt)   (2.5,-2) circle (1pt)
   (0.5,-2.5) circle (1pt)   (2.25,-2.5) circle (1pt);

 \begin{scope}[xshift=4cm]
  \draw
   (0,-0.5) -- (0.25,0) -- (0.5,-0.5)   (1,0) -- (1,-0.5)   (1.5,-0.5) --
(1.5,0)   (2,0) -- (2,-0.5)
   (0.5,-0.5) to [out=-90, in=90] (0,-1.5)   (1.5,-0.5) to [out=-90, in=90]
(2,-1.5);

  \draw[white, line width=4pt]
   (0,-0.5) to [out=-90, in=90] (1,-1.5)   (2,-0.5) to [out=-90, in=90]
(1.5,-1.5);
  \draw
   (0,-0.5) to [out=-90, in=90] (1,-1.5)   (2,-0.5) to [out=-90, in=90]
(1.5,-1.5);
  \draw[white, line width=4pt]
   (1,-0.5) to [out=-90, in=90] (0.5,-1.5);
  \draw
   (1,-0.5) to [out=-90, in=90] (0.5,-1.5);

  \draw
   (0,-1.5) -- (0.75,-2.25) -- (1,-2)   (0.5,-1.5) -- (1,-2) -- (1.5,-1.5)  
(1,-1.5) -- (1.25,-1.75)   (2,-1.5) -- (2,-2.25);

  \filldraw
   (0,-0.5) circle (1pt)   (0.25,0) circle (1pt)   (0.5,-0.5) circle (1pt)  
(1,0) circle (1pt)   (1,-0.5) circle (1pt)   (1.5,-0.5) circle (1pt)   (1.5,0)
circle (1pt)   (2,-0.5) circle (1pt)   (2,0) circle (1pt)
   (0,-1.5) circle (1pt)   (0.5,-1.5) circle (1pt)   (1,-2) circle (1pt)  
(1.5,-1.5) circle (1pt)   (1,-1.5) circle (1pt)   (1.25,-1.75) circle (1pt)  
(2,-1.5) circle (1pt)   (2,-2.25) circle (1pt)
   (0.75,-2.25) circle (1pt);
 \end{scope}
\end{tikzpicture}
\caption{An elementary merging of a spraige.}
\label{fig:spraige_merging}
\end{figure}
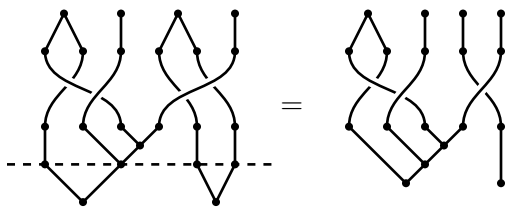

In the special case that $F=F^{(n)}_i$ for $i\in \{1,\dots,n\}$ we can
think of a splitting by~$F$ as simply attaching a single caret to the
$i^{\text{th}}$ foot of a spraige, possibly followed by reductions.
Similarly a merging by $F$ in this case can be thought of as merging
the $i^{\text{th}}$ and $(i+1)^{\text{st}}$ feet together.  In
these cases we will also speak of \emph{adding a split}
(respectively~\emph{merge}) to the spraige.

The following types of spraiges will prove to be particularly
important.  First, a \emph{braige} is defined to be a spraige where
there are no splits, i.e., a spraige of the form~$(1_n,b,F)$ for $b\in
B_n$ and $F$ having $n$ leaves.  Also, when $F$ is elementary we will
call $(1_n,b,F)$ an \emph{elementary braige}.  Analogously to
spraiges, we define \emph{$n$-braiges} and \emph{elementary
$n$-braiges}.

To deal with $\Fbr$, we make the following convention.  Whenever we
want to only consider pure braids, we will attach the modifier
``pure'', e.g., we can talk about pure $n$-spraiges, or elementary
pure $n$-braiges.

\subsection{Dangling spraiges}\label{sec:dangling}

We can identify the braid group $B_n$ with a subgroup of~$\spraige_{n,n}$ via $b\mapsto (1_n,b,1_n)$. In particular for any $n,m\in\N$ there is a right
action of the braid group $B_m$ on $\spraige_{n,m}$, by right multiplication. 
Quotienting out modulo this action encodes the idea that the feet of a spraige may ``dangle''.  See
Figure~\ref{fig:dangle} for an example of the dangling action of $B_2$
on~$\spraige_{4,2}$.

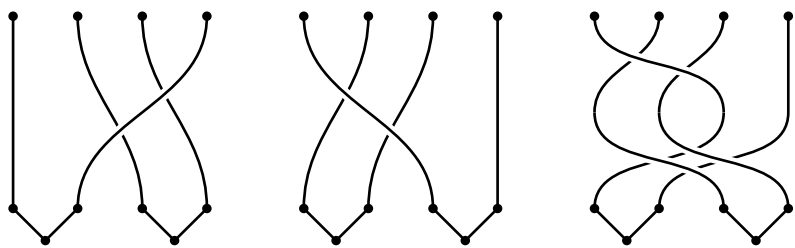
\begin{figure}[t]
\centering
\begin{tikzpicture}[line width=1pt, scale=0.85]
  
  \draw
  (0,-1.5) to [out=270, in=90] (-1,-4.5)
  (1,-1.5) to [out=270, in=90] (0,-4.5)
  (2,-1.5) -- (2,-4.5);
  \draw [line width=4pt, white]
  (-1,-1.5) to [out=270, in=90] (1,-4.5);
  \draw
  (-1,-1.5) to [out=270, in=90] (1,-4.5)
  (-1,-4.5) -- (-0.5,-5) -- (0,-4.5)   (1,-4.5) -- (1.5,-5) -- (2,-4.5);
  \filldraw
  (-1,-1.5) circle (1.5pt)   (0,-1.5) circle (1.5pt)   (1,-1.5) circle (1.5pt)  
(2,-1.5) circle (1.5pt)
  (-1,-4.5) circle (1.5pt)   (0,-4.5) circle (1.5pt)   (1,-4.5) circle (1.5pt)  
(2,-4.5) circle (1.5pt)
  (-0.5,-5) circle (1.5pt)   (1.5,-5) circle (1.5pt);

 \begin{scope}[xshift=-3.5cm,xscale=-1]
  \draw
  (0,-1.5) to [out=270, in=90] (-1,-4.5)
  (1,-1.5) to [out=270, in=90] (0,-4.5)
  (2,-1.5) -- (2,-4.5);
  \draw [line width=4pt, white]
  (-1,-1.5) to [out=270, in=90] (1,-4.5);
  \draw
  (-1,-1.5) to [out=270, in=90] (1,-4.5)
  (-1,-4.5) -- (-0.5,-5) -- (0,-4.5)   (1,-4.5) -- (1.5,-5) -- (2,-4.5);
  \filldraw
  (-1,-1.5) circle (1.5pt)   (0,-1.5) circle (1.5pt)   (1,-1.5) circle (1.5pt)  
(2,-1.5) circle (1.5pt)
  (-1,-4.5) circle (1.5pt)   (0,-4.5) circle (1.5pt)   (1,-4.5) circle (1.5pt)  
(2,-4.5) circle (1.5pt)
  (-0.5,-5) circle (1.5pt)   (1.5,-5) circle (1.5pt);
 \end{scope}

 \begin{scope}[xshift=4.5cm]
  \draw
  (0,-1.5) to [out=270, in=90] (-1,-3)
  (1,-1.5) to [out=270, in=90] (0,-3)
  (2,-1.5) -- (2,-3)
  (1,-3) to [out=270, in=90] (-1,-4.5)
  (2,-3) to [out=270, in=90] (0,-4.5);
  \draw [line width=4pt, white]
  (-1,-1.5) to [out=270, in=90] (1,-3)
  (0,-3) to [out=270, in=90] (2,-4.5)
  (-1,-3) to [out=270, in=90] (1,-4.5);
  \draw
  (-1,-1.5) to [out=270, in=90] (1,-3)
  (0,-3) to [out=270, in=90] (2,-4.5)
  (-1,-3) to [out=270, in=90] (1,-4.5)
  (-1,-4.5) -- (-0.5,-5) -- (0,-4.5)   (1,-4.5) -- (1.5,-5) -- (2,-4.5);
  \filldraw
  (-1,-1.5) circle (1.5pt)   (0,-1.5) circle (1.5pt)   (1,-1.5) circle (1.5pt)  
(2,-1.5) circle (1.5pt)
  (-1,-4.5) circle (1.5pt)   (0,-4.5) circle (1.5pt)   (1,-4.5) circle (1.5pt)  
(2,-4.5) circle (1.5pt)
  (-0.5,-5) circle (1.5pt)   (1.5,-5) circle (1.5pt);
 \end{scope}
\end{tikzpicture}
\caption{Dangling.}
\label{fig:dangle}
\end{figure}

For $\sigma\in\spraige_{n,m}$, denote by $[\sigma]$ the orbit of
$\sigma$ under this action, and call $[\sigma]$ a \emph{dangling
$(n,m)$-spraige}.  We can also refer to a dangling $n$-spraige or
dangling spraige.  The action of $B_m$ preserves the property of being
a braige or elementary braige, so we can also refer to dangling
braiges and dangling elementary braiges.

Let $\Poset$ denote the set of all dangling spraiges, with
$\Poset_{n,m}$ and $\Poset_n$ defined in the obvious way.  When $m=1$,
$B_m$ is trivial, so we will identify $\Poset_{n,1}$ with~$\spraige_{n,1}$ for each~$n$.  In particular we identify~$\Poset_{1,1}$ with $\Vbr$.  Note that if $\sigma\in\spraige_{n,m}$ and~$\tau_1, \tau_2
\in \spraige_{m,\ell}$ with $[\sigma \ast \tau_1] =
[\sigma \ast \tau_2]$, then $[\tau_1]=[\tau_2]$. We will refer to this fact as
\emph{left cancellation}.

There is also a poset structure on $\Poset$.  For $x,y\in\Poset$, with
$x=[\sigma_x]$, say that~$x\le y$ if there exists a forest~$F$
with~$m$ leaves such that~$y=[\sigma_x \ast (F,\id,1_m)]$.  In other
words,~$x\le y$ if~$y$ is obtained from~$x$ via splitting.  It is easy
to see that this is a partial ordering.  Also, if~$x\in\Poset_n$
and~$y\in\Poset$ with~$x\le y$ or~$y\le x$, then~$y\in\Poset_n$.
In other words, two elements are comparable only if they have the same number of heads.  Also define a
relation~$\preceq$ on~$\Poset$ as follows.  If~$x = [\sigma_x] \in
\Poset$ and $y\in\Poset$ such that~$y = [\sigma_x \ast
\lambda^{(n)}_J]$ for some $n\in\N$
and~$J\subseteq\{1,\dots,n\}$, write~$x\preceq y$.  That is,~$x\preceq
y$ if $y$ is obtained from~$x$ via elementary splitting, and this is a
well defined relation with respect to dangling.  If~$x\preceq y$
and~$x\neq y$ then write $x\prec y$.  Note that~$\preceq$ and~$\prec$
are not transitive, though it is true that if~$x\preceq z$ and~$x\le
y\le z$ then $x\preceq y$ and~$y\preceq z$.  This is all somewhat similar to the corresponding
situation for~$F$ and~$V$ discussed for example in Section~4
of~\cite{brown92}.

We remark that a totally analogous construction yields the notion of a
\emph{dangling pure spraige}, where the dangling is now via the action
of the pure braid group.  We also have dangling pure braiges and
dangling elementary pure braiges.  All of the essential results above
still hold.


\section{The Stein space}
\label{sec:def_stein_space}

In this section we construct a space $X$ on which $\Vbr$ acts and which
we call the \emph{Stein space} for~$\Vbr$.  This construction can also
be reproduced using pure braids to get a space $X(\Fbr)$ on which~$\Fbr$ acts; we will say more about this at the end of the section.
Similar spaces, which could be termed the Stein spaces for~$F$
and~$V$, were constructed and discussed
in~\cite{brown92,brown06,stein92}.  Also, a Stein space was
used in~\cite{fluch13} to show that the higher dimensional versions
$sV$ of~$V$ are of type~$\F_\infty$.  In the course of defining our
space~$X$, it should be clear how the Stein spaces for~$F$ and~$V$
would be described with our model.  This construction was given in
some generality in~\cite{stein92}, and was further generalized
in~\cite{farley03} to get finiteness properties (among other things)
for a class of groups called diagram groups, of which~$F$ is an
example.

Our starting point is the poset $\Poset_1$ of dangling $1$-spraiges,
i.e., dangling spraiges with a single head.  As with any poset, we
have the following terminology.  If $x\le z$ and $y\le z$ call $z$ an
\emph{upper bound} of $x$ and $y$.  The minimal elements of the set of
upper bounds of $x$ and $y$ are called \emph{minimal upper bounds}.
If $x$ and $y$ have a unique minimal upper bound $z$ call $z$ the
\emph{least upper bound} of $x$ and $y$.  Similarly define \emph{lower
bounds}, \emph{maximal lower bounds} and \emph{greatest lower bounds}.

We remark that while we construct $X$ starting with $\Poset_1$, the
following results remain essentially unchanged if we start instead
with $\Poset_n$ for some other $n$.  Since we want $\Vbr$ to act
on~$X$ though,~$\Poset_1$ is the right place to start.

\begin{proposition}
\label{prop:sups}
Let $x,y\in\Poset_1$.  Then $x$ and $y$ have a least upper bound.
Also, if~$x$ and~$y$ have a lower bound then they have a greatest
lower bound.
\end{proposition}

\begin{proof}
We first claim that $x$ and $y$ have an upper bound.  Represent $x$ by
the spraige~$\sigma_x=(T,b,F)$, where $T$ is a tree with $n$ leaves,
$b\in B_n$ and $F$ is a forest with~$k$ roots and~$n$ leaves, so $x$
is a dangling $(1,k)$-spraige.  Represent $y$ by~$\sigma_y=(U,c,G)$,
where $U$ is a tree with~$m$ leaves, $c\in B_m$, and $G$ is a forest
with~$\ell$ roots and~$m$ leaves, so $y$ is a dangling
$(1,\ell)$-spraige.  Now, $[\sigma_x\ast(F,\id,1_n)] = [(T,\id,1_n)]$,
so~$x\le[(T,\id,1_n)]$.  Similarly~$y\le[(U,\id,1_m)]$.  Since $T$ and
$U$ are both trees,~$[(T,\id,1_n)]$
and~$[(U,\id,1_m)]$ have an upper bound, and hence so do~$x$ and~$y$.

We now claim that there is even a least upper bound.  Again take
$\sigma_x=(T,b,F)$ and $\sigma_y=(U,c,G)$, and suppose~$z$ and~$w$ are
both minimal upper bounds of $x=[\sigma_x]$ and~$y=[\sigma_y]$.  Then
there is a $(k,\ell)$-spraige $(H_-,d,H_+)$ such that
$[\sigma_x\ast(H_-,\id,1_p)]=z$ and~$[\sigma_x\ast(H_-,d,H_+)]=y$, and
there is a $(k,\ell)$-spraige~$(I_-,e,I_+)$ such that
$[\sigma_x\ast(I_-,\id,1_q)]=w$ and $[\sigma_x\ast(I_-,e,I_+)]=y$.
Here $H_-$ has $p$ leaves and $I_-$ has $q$ leaves.  In particular
$[\sigma_x\ast(H_-,d,H_+)]=[\sigma_x\ast(I_-,e,I_+)]$, which by left
cancellation tells us that $[(H_-,d,H_+)]=[(I_-,e,I_+)]$. Moreover,
since~$z$ and~$w$ are minimal upper bounds of~$x$ and~$y$, the
spraiges~$(H_-,d,H_+)$ and~$(I_-,e,I_+)$ are reduced.  By uniqueness
of reduced representatives, we must have in particular that~$H_-=I_-$,
and so~$z=w$.

Finally suppose $x$ and $y$ have maximal lower bounds $z$ and $w$.  We
claim that~$w=z$.  Of course $x$ and $y$ are upper bounds for~$z$
and~$w$, so if $v$ is the least upper bound of~$z$ and~$w$ then~$v$ is
a lower bound of~$x$ and~$y$.  But then since~$z$ and~$w$ are maximal
lower bounds we conclude that~$z = v = w$.
\end{proof}

Now consider the geometric realization $|\Poset_1|$, i.e., the
simplicial complex with a~$k$-simplex for every chain $x_0<\cdots<x_k$
in $\Poset_1$.  We will refer to $x_k$ as the \emph{top} of the
simplex and $x_0$ as the \emph{bottom}.  Call such a simplex
\emph{elementary} if~$x_0\preceq x_k$.

\begin{definition}[Stein space]\label{def:stein_space}
Define the \emph{Stein space} $X$ for $\Vbr$ to be the subcomplex
of~$|\Poset_1|$ consisting of all elementary simplices. 
\end{definition}

Since faces of elementary simplices are elementary, this is a
subcomplex.  In fact there is a coarser cell decomposition of $X$, as
a cubical complex, which we describe as follows.  For $x\le y$ define
the closed interval $[x,y]\defeq \{z\mid x\le z\le y\}$.  Similarly
define the open and half-open intervals $(x,y)$, $(x,y]$ and $[x,y)$.
Note that if $x\preceq y$ then the closed interval $[x,y]$ is a
Boolean lattice, and so the simplices in its geometric realization
piece together into a cube.  The \emph{top} of the cube is $y$ and the
\emph{bottom} is~$x$.  Every elementary simplex is contained in such a
cube, and the face of any cube is clearly another cube.  Also, the
intersection of cubes is either empty or is itself a cube; this is
clear since if $[x,y]\cap[z,w]\neq\emptyset$ then $y$ and $w$ have a
lower bound, and we get that $[x,y]\cap[z,w]=[\sup(x,z),\inf(y,w)]$.
This means that~$X$ has the structure of a cubical complex, in
the sense of~\cite[p.~112, Definition~7.32]{bridson99}.  This is all
very similar to the construction of the Stein spaces for~$F$ and~$V$
given in~\cite{brown92,stein92}.

Recall that $f(x)$ is the number of feet of a spraige $x$.

\begin{lemma}
\label{lem:cube_lemma}
For $x<y$ with $x\not\prec y$, $|(x,y)|$ is contractible.
\end{lemma}

\begin{proof}
The proof is very similar to the proof of the lemma in Section~4
of~\cite{brown92}.  For any $z\in(x,y]$ let $z_0$ be the largest
element of $[x,z]$ such that $x\preceq z_0$.  By hypothesis
$z_0\in[x,y)$, and by the definition of $\preceq$ it is clear that
$z_0\in(x,y]$, so in fact $z_0\in(x,y)$.  Also, $z_0\le y_0$ for any
$z\in(x,y)$.  The inequalities $z\ge z_0\le y_0$ then imply that
$|(x,y)|$ is contractible, by Section~1.5 of~\cite{quillen78}.
\end{proof}

\begin{corollary}
\label{cor:stein_space_cible}
$X$ is contractible.
\end{corollary}

\begin{proof}
First note that $|\Poset_1|$ is contractible since $\Poset_1$ is
directed.  Similar to the situation in~\cite{brown92}, we can build up
from $X$ to $|\Poset_1|$ by attaching new subcomplexes, and we claim
that this never changes the homotopy type, so $X$ is contractible.
Given a closed interval $[x,y]$, define~$r([x,y])\defeq f(y)-f(x)$.
As a remark, if~$x\preceq y$ then~$r([x,y])$ is the dimension of
the cube given by $[x,y]$.  We attach the contractible subcomplexes
$|[x,y]|$ for $x\not\preceq y$ to $X$ in increasing order of $r$-value.
When we attach $|[x,y]|$ then, we attach it along
$|[x,y)|\cup|(x,y]|$.  But this is the suspension of $|(x,y)|$, and so
is contractible by the previous lemma.  We conclude that attaching
$|[x,y]|$ does not change the homotopy type, and since $|\Poset_1|$ is
contractible, so is $X$.
\end{proof}

\medskip

There is a natural action of $\Vbr$ on the vertices of $X$.  Namely,
for $g\in\Vbr$ and~$\sigma\in\spraige_1$ with $x=[\sigma]$, define $gx
\defeq [g\ast\sigma]$.  This action preserves the
relations~$\le$ and~$\preceq$, and hence extends to an action on the
whole space.  To prove that $\Vbr$ is of type $\F_\infty$, we will
apply Brown's Criterion to the action of $\Vbr$ on $X$.

\begin{brownscriterion}\cite[Corollary~3.3]{brown87}
\label{brownscriterion}
Let $G$ be a group and $X$ a contractible $G$-CW-complex such that the
stabilizer of every cell is of type $\F_\infty$.  Let $\{X_j\}_{j\ge
1}$ be a filtration of $X$ such that each~$X_j$ is finite
$\textnormal{mod}~G$.  Suppose that the connectivity of the pair
$(X_{j+1},X_j)$ tends to $\infty$ as $j$ tends to $\infty$.  Then $G$
is of type~$\F_\infty$.
\end{brownscriterion}

For each $n\in\N$ define $X^{\le n}$ to be the full subcomplex of~$X$
spanned by vertices~$x$ with $f(x)\le n$.  Note that the $X^{\le n}$
are invariant under the action of~$\Vbr$.\pagebreak[3]

\begin{lemma}[Cocompactness]
\label{lem:cocompactness}
For each $n\ge 1$ the sublevel set $X^{\le n}$ is finite modulo~$\Vbr$.
\end{lemma}

\begin{proof}
Note first that for each $k\ge 1$, $\Vbr$ acts transitively on the set
of $(1,k)$-spraiges.  Thus there exists for each $1\le k\le n$ one
orbit of vertices $x$ in $X^{\le n}$ with~$f(x)=k$.  Given a vertex
$x$ with $f(x)=k$ there exist only finitely many cubes~$C_1, \ldots,
C_r$ in the sublevel set $X^{\le n}$ that have $x$ as their bottom.
If $C$ is a cube in $X^{\le n}$ such that its bottom is in the same
orbit as $x$, then the cube~$C$ must be in the same orbit as $C_i$ for
some $1\le i\le r$.  It follows that there can only be finitely many
orbits of cubes in the sublevel set $X^{\le n}$.
\end{proof}

\begin{lemma}[Vertex stabilizers]\label{lem:stabilizers}
Let $x$ be a vertex in $X$ with $f(x)=n$.  The stabilizer
$\Stab_{\Vbr}(x)$ is isomorphic to~$B_n$.
\end{lemma}

\begin{proof}
As a first step, identify $B_n$ with its image under the inclusion $B_n\hookrightarrow \spraige_{n,n}$ sending~$b$ to $(1_n,b,1_n)$.

Let $g\in\Stab_{\Vbr}(x)$.  Fix $\sigma\in\spraige_{1,n}$ with
$x=[\sigma]$.  We have $gx=x$, which means that $[g\ast\sigma]=[\sigma]$,
and so in particular $\sigma^{-1} \ast g \ast \sigma \in B_n$.

Define a map
\begin{align*}
\psi\colon \Stab_{\Vbr}(x) & \to B_n
\\
g \qquad & \mapsto \, \sigma^{-1} \ast g \ast \sigma\text{\,.}
\end{align*}
This is an isomorphism, with inverse $b\mapsto \sigma \ast b\ast \sigma^{-1}$. We
remark that $\psi$ depends on the choice of~$\sigma$, and
so is not canonical, but it is uniquely determined up to inner
automorphisms of~$B_n$.
\end{proof}

\begin{definition}
Let $J\subseteq\{1,\dots,n\}$.  Let $b\in B_n$ and let $\rho_b$ be the
corresponding permutation in $S_n$.  If $\rho_b$ stabilizes $J$
set-wise, call $b$ an $J$-\emph{stabilizing} braid.  Let $B_n^J\le
B_n$ be the subgroup of~$J$-stabilizing braids.
\end{definition}

\begin{corollary}[Cell stabilizers]\label{cor:cell_stabs}
Let $x$ be a vertex in $X$, with $f(x)=n$ and $x=[\sigma]$, and let
$F^{(n)}_J$ be an elementary forest.  If
$y=[\sigma\ast\lambda^{(n)}_J]$, then the stabilizer in~$\Vbr$ of the
cube $[x,y]$ is isomorphic to~$B_n^J$.  In particular all cell
stabilizers are of type~$\F_\infty$.
\end{corollary}

\begin{proof}
First observe that $g\in\Vbr$ stabilizes $[x,y]$ if and only if it
stabilizes $x$ and~$y$.  For~$g\in\Stab_{\Vbr}(x)$ let $b_g \defeq \sigma^{-1} \ast g \ast \sigma \in B_n$, where we identify~$B_n$ as a subgroup of~$\spraige_{n,n}$ as in the previous proof. Then~$g$
stabilizes~$y$ if and only if $[\sigma \ast b_g \ast\lambda^{(n)}_J]
= [\sigma\ast\lambda^{(n)}_J]$, which by left cancellation is
equivalent to~$[b_g \ast \lambda^{(n)}_J] =
[\lambda^{(n)}_J]$.  This in turn is equivalent to~$b_g \in
B_n^J$, and so the cube stabilizer equals $\psi^{-1}(B_n^J)$, where
$\psi\colon\Stab_{\Vbr}(x)\to B_n$ is the map from the previous proof.
Since~$\psi$ is an isomorphism, the first statement follows.

The second statement follows since braid groups are of
type~$\F_\infty$~\cite[Theorem~A]{squier94}, and each~$B_n^J$ has
finite index in $B_n$.
\end{proof}

\medskip

The filtration $\{X^{\le n}\}_{n\ge1}$ of $X$ has so far been shown to
satisfy all the conditions of Brown's Criterion save one, namely that
the connectivity of the pair $(X^{\le n+1},X^{\le n})$ tends to
$\infty$ as $n$ tends to $\infty$.  We will prove this in
Corollary~\ref{cor:pairs_conn}, using discrete Morse theory.  We now
describe the Morse-theoretic tools we will use.

Let $Y$ be a piecewise Euclidean cell complex, and let~$h$ be a map
from the set of vertices of $Y$ to the integers, such that each cell has a unique vertex maximizing~$h$.  Call $h$ a
\emph{height function}, and $h(y)$ the \emph{height} of $y$ for
vertices $y$ in $Y$.  For $t\in\Z$, define $Y^{\le t}$ to be
the full subcomplex of~$Y$ spanned by vertices~$y$ satisfying $h(y)\le
t$.  Similarly define $Y^{< t}$, and let $Y^{=t}$ be the set of
vertices at height~$t$.  The \emph{descending star} $\dst(y)$ of a
vertex $y$ is defined to be the open star of~$y$ in~$Y^{\le y}$.  The
\emph{descending link} $\dlk(y)$ of $y$ is given by the set of ``local
directions'' starting at $y$ and pointing into $\dst(y)$. More details can be found in \cite{bestvina97}, and the following Morse Lemma is a consequence of \cite[Corollary~2.6]{bestvina97}.

\begin{morselemma}
\label{lem:morse_lemma}
With the above setup, the following hold.
 \begin{enumerate}
	\item Suppose that for any vertex $y$ with $h(y)=t$, $\dlk(y)$
	is $(k-1)$-connected.  Then the pair $(Y^{\le t},Y^{<t})$ is
	$k$-connected, that is, the inclusion $Y^{<t}\hookrightarrow
	Y^{\le t}$ induces an isomorphism in~$\pi_j$ for $j<k$, and an
	epimorphism in $\pi_k$.

	\item Suppose that for any vertex $y$ with $h(y)\ge t$,
	$\dlk(y)$ is $(k-1)$-connected.  Then $(Y,Y^{< t})$ is
	$k$-connected.
 \end{enumerate}
\end{morselemma}

\medskip

The first part of the Morse Lemma will be applied to the Stein space,
and the second part will become convenient in
Section~\ref{sec:match_cpxes}.

Every cell of~$X$ has a unique vertex maximizing~$f$, so $f$ is a height function. Hence we can inspect the connectivity of the pair $(X^{\le n},X^{< n})$
by looking at descending links with respect to~$f$.  In the rest of
this section, we describe a convenient model for the descending links,
and then analyze their connectivity in the following sections.

Recall that we identify $\Poset_1$ with the vertex set of $X$, and
cubes in $X$ are (geometric realizations of) intervals $[y,x]$
with $x,y\in\Poset_1$ and $y\preceq x$.  For $x\in\Poset_1$, the
descending star $\dst(x)$ of $x$ in $X$ is the set of cubes $[y,x]$
with top $x$.  For such a cube $C=[y,x]$ let $\Bot(C) \defeq y$ be the
map giving the bottom vertex.  This is a bijection from the
set of such cubes to the set $D(x) \defeq \{y\in\Poset_1\mid y\preceq
x\}$.  The cube $[y',x]$ is a face of~$[y,x]$ if and only if
$y'\in[y,x]$, if and only if $y'\ge y$.  Hence $C'$ is a face of $C$
if and only if $\Bot(C')\ge \Bot(C)$, so $\Bot$ is an order-reversing
poset map.  By considering cubes~$[y,x]$ with $y\neq x$ and
restricting to $D(x) \setminus \{x\}$, we obtain a description of~$\dlk(x)$.  Namely, a simplex in~$\dlk(x)$ is a dangling spraige~$y$
with~$y\prec x$, the rank of the simplex is the number of elementary
splits needed to get from $y$ to $x$ (so the number of elementary
merges to get from~$x$ to~$y$) and the face relation is the reverse of
the relation~$<$ on $D(x)\setminus\{x\}$.  Since $X$ is a cubical
complex, $\dlk(x)$ is a simplicial complex.

\subsection*{A model for the descending link:} If $f(x)=n$, then thanks to left cancellation, $\dlk(x)$ is isomorphic to the simplicial complex $\elbraigecpx_n$ of dangling
elementary $n$-braiges $[(1_n,b,F_J^{(n-|J|)})]$ for
$J\neq\emptyset$, with the face relation given by the reverse of the
ordering~$\le$ in $\Poset_n$.  See Figure~\ref{fig:desc_lk_to_EB} for
an idea of the correspondence between~$\dlk(x)$ and $\elbraigecpx_n$.  We will
usually draw braiges as emerging from a horizontal line, as a visual
reminder of this correspondence.

\begin{figure}
\centering
\begin{tikzpicture}

  \filldraw[lightgray]
   (0,0) -- (1,1) -- (2,0) -- (1,-1) -- (0,0);
  \draw
   (2,0) -- (1,-1) -- (0,0);
  \draw[darkgray, line width=1.5pt]
   (0,0) -- (1,1) -- (2,0);
  \filldraw
   (0,0) circle (1.5pt)   (1,1) circle (1.5pt)   (2,0) circle (1.5pt)   (1,-1)
circle (1.5pt);

  \draw
   (0.75,1.4) to [out=90, in=90, looseness=2] (1.25,1.4) -- (0.75,1.4)
   (0.85,1.4) -- (0.85,1.2)   (0.95,1.4) -- (0.95,1.2)   (1.05,1.4) --
(1.05,1.2)   (1.15,1.4) -- (1.15,1.2);
  \node at (1,1.52) {$x$};

 \begin{scope}[xshift=-1.2cm,yshift=-1cm]
  \draw
   (0.75,1.4) to [out=90, in=90, looseness=2] (1.25,1.4) -- (0.75,1.4)
   (0.85,1.4) -- (0.85,1.2) -- (0.9,1.1) -- (0.95,1.2) -- (0.95,1.4)  
(1.05,1.4) -- (1.05,1.2)   (1.15,1.4) -- (1.15,1.2);
  \node at (1,1.52) {$x$};
 \end{scope}

 \begin{scope}[xshift=1.2cm,yshift=-1cm]
  \draw
   (0.75,1.4) to [out=90, in=90, looseness=2] (1.25,1.4) -- (0.75,1.4)
   (0.85,1.4) -- (0.85,1.2)   (0.95,1.4) -- (0.95,1.2)   (1.05,1.4) --
(1.05,1.2) -- (1.1,1.1) -- (1.15,1.2) -- (1.15,1.4);
  \node at (1,1.52) {$x$};
 \end{scope}

 \begin{scope}[yshift=-2.85cm]
  \draw
   (0.75,1.4) to [out=90, in=90, looseness=2] (1.25,1.4) -- (0.75,1.4)
   (0.85,1.4) -- (0.85,1.2) -- (0.9,1.1) -- (0.95,1.2) -- (0.95,1.4)  
(1.05,1.4) -- (1.05,1.2) -- (1.1,1.1) -- (1.15,1.2) -- (1.15,1.4);
  \node at (1,1.52) {$x$};
 \end{scope}

  \node at (3.5,0) {$\longleftrightarrow$};

  \draw[line width=1.5pt, gray]
   (4.8,0) -- (6.8,0);
  \filldraw[darkgray]
   (4.8,0) circle (1.5pt)   (6.8,0) circle (1.5pt);

 \begin{scope}[xshift=4.7cm, yshift=-0.2cm, scale=2]
  \draw
   (-0.15,0.5) -- (-0.15,0.3) -- (-0.1,0.2) -- (-0.05,0.3) -- (-0.05,0.5)  
(0.05,0.5) -- (0.05,0.3)   (0.15,0.5) -- (0.15,0.3);

  \draw
   (-0.2,0.5) -- (0.2,0.5);

(0.5pt)   (-0.05,0.3) circle (0.5pt)   (-0.05,0.5) circle (0.5pt)   (0.05,0.5)
circle (0.5pt)   (0.05,0.3) circle (0.5pt)   (0.15,0.5) circle (0.5pt)  
(0.15,0.3) circle (0.5pt);
 \end{scope}

 \begin{scope}[xshift=6.9cm, yshift=-0.2cm, scale=2]
  \draw
   (-0.15,0.5) -- (-0.15,0.3)   (-0.05,0.3) -- (-0.05,0.5)   (0.05,0.5) --
(0.05,0.3) -- (0.1,0.2) -- (0.15,0.3) -- (0.15,0.5);

  \draw
   (-0.2,0.5) -- (0.2,0.5);

(0.5pt)   (-0.05,0.3) circle (0.5pt)   (-0.05,0.5) circle (0.5pt)   (0.05,0.5)
circle (0.5pt)   (0.05,0.3) circle (0.5pt)   (0.15,0.5) circle (0.5pt)  
(0.15,0.3) circle (0.5pt);
 \end{scope}

 \begin{scope}[xshift=5.8cm, yshift=-1.3cm, scale=2]
  \draw
   (-0.15,0.5) -- (-0.15,0.3) -- (-0.1,0.2) -- (-0.05,0.3) -- (-0.05,0.5)  
(0.05,0.5) -- (0.05,0.3) -- (0.1,0.2) -- (0.15,0.3) -- (0.15,0.5);

  \draw
   (-0.2,0.5) -- (0.2,0.5);

(0.5pt)   (0.1,0.2) circle (0.5pt)   (-0.05,0.3) circle (0.5pt)   (-0.05,0.5)
circle (0.5pt)   (0.05,0.5) circle (0.5pt)   (0.05,0.3) circle (0.5pt)  
(0.15,0.5) circle (0.5pt)   (0.15,0.3) circle (0.5pt);
 \end{scope}
\end{tikzpicture}
\caption{The correspondence between $\dlk(x)$ and $\elbraigecpx_n$.}
\label{fig:desc_lk_to_EB}
\end{figure}

We will prove that $\elbraigecpx_n$ is highly connected in
Corollary~\ref{cor:desc_link_conn}.  Our proof relies on a complex
that we call the \emph{matching complex on a surface}, which we will
define and analyze in the next section.  Then we will return to
considering dangling elementary braiges in
Section~\ref{sec:desc_link_conn}.

We close this section with some remarks on $\Fbr$.  Restricting to
pure braids everywhere in this section does not affect any of the
proofs, so we can simply say that $X(\Fbr)$ is the contractible
cubical complex of dangling pure $1$-spraiges, understood in the same
way as $X$ (though now dangling is only via pure braids).  We will
also denote by $f$ the height function ``number of feet'' on
$X(\Fbr)$.  The filtration is still cocompact and the stabilizers are
still of type $\F_\infty$, being finite index subgroups of braid
groups.  As for descending links, the descending link of a dangling
pure $(1,n)$-spraige in $X(\Fbr)$ is isomorphic to the simplicial
complex~$\elpbraigecpx_n$ of dangling elementary pure~$n$-braiges.


\section{Matching complexes on surfaces}
\label{sec:surfaces}

Throughout this section $S$ denotes a connected surface, with
(possibly empty) boundary~$\partial S$, and $P$ denotes a finite set
of points in $S\setminus\partial S$.  By an \emph{arc}, we mean a
simple path in $S\setminus\partial S$ that intersects $P$ precisely at
its endpoints, and whose endpoints are distinct.  Our standard
reference for arc complexes is~\cite{hatcher91}.  Note that our
definition of arc is slightly different from the definition given
in~\cite{hatcher91}, in that we do not allow the endpoints of a given
arc to coincide.  Also note that in~\cite{hatcher91}, points in $P$
were allowed to be contained in $\partial S$, and we will not consider
this case here.  In Section~\ref{sec:desc_link_conn} we will only need
the special case where~$S$ is a disk, but to prove the results in this
section we need to use this degree of generality.

A major theme of this section will be the similarities between certain
complexes defined using edges in graphs, and similar complexes defined
using arcs on surfaces.  Of particular interest is the family of
\emph{complete graphs} $K_n$.  The graph $K_n$ is the graph with $n$
nodes and a single edge between any two nodes.  Later we will also be
interested in the family of subgraphs of \emph{linear graphs}~$L_n$.
The linear graph~$L_n$ has~$n+1$ nodes labeled $1$ through~$n+1$, and
$n$ edges, one connecting the node labeled $i$ to the node
labeled~$i+1$ for each~$1\le i\le n$.  Note that when dealing
with~$K_n$,~$n$ is the number of nodes, but when dealing
with~$L_n$,~$n$ is the number of edges.  This is just for the sake of
future ease of notation.

\subsection{The arc complex}
\label{sec:arc_cpx}

Let $\{\alpha_0,\dots,\alpha_k\}$ be a collection of arcs.  If the
$\alpha_i$ are all disjoint from each other except possibly at their
endpoints, and if no distinct~$\alpha_i$ and~$\alpha_j$ are homotopic
relative $P$, we call $\{\alpha_0,\dots,\alpha_k\}$ an \emph{arc
system}.  The homotopy classes, relative $P$, of arc systems form the
simplices of a simplicial complex, with the face relation given by
passing to subcollections of arcs.

\begin{definition}[Arc complex]\label{def:hatcher_cpx}
Let $\Gamma$ be a graph with $|P|$ nodes, and identify $P$ with the
set of nodes of $\Gamma$.  Call an arc in $S$ \emph{compatible} with $\Gamma$
if its endpoints are connected by an edge in $\Gamma$.  Let $\hatcharc(\Gamma)$ be the \emph{arc complex} on $(S,P)$ corresponding
to $\Gamma$, that is the simplicial complex with a $k$-simplex for
each arc system $\{\alpha_0,\dots,\alpha_k\}$ such that all the
$\alpha_i$ are compatible with $\Gamma$.
\end{definition}

We remark
that $\hatcharc(K_n)$ is a proper subcomplex of the space $\mathcal{A}(S,P)$
in~\cite{hatcher91}, since we only consider arcs with two distinct
endpoints.

It will be convenient to consider actual arcs rather than homotopy classes in many of the following arguments. This is justified by the following fact.

\begin{lemma}
\label{lem:arc_representatives}
Given finitely many homotopy classes of arcs $[\alpha_0],\ldots,[\alpha_k]$ there are representatives $\alpha_0,\ldots,\alpha_k$ such that $\abs{\alpha_i \cap \alpha_j}$ is minimal among all representatives of $[\alpha_i]$ and $[\alpha_j]$ for $0 \le i<j \le k$. In particular, any simplex is represented by disjoint arcs.
\end{lemma}

\begin{proof}
If $\abs{P} \le 2$ there is at most one arc and nothing to show. If $\abs{P} \ge 3$ we consider the points in $P$ as punctures. Then $S$ has negative Euler characteristic so we may equip it with a hyperbolic metric. The following references are stated for closed curves but also apply to arcs, see \cite[Section~1.2.7]{farb12}. For each homotopy class $[\alpha_i]$ we take $\alpha_i$ to be the geodesic within the class \cite[Proposition~1.3]{farb12}. Then any two of the arcs intersect minimally \cite[Corollary~1.9]{farb12}.
\end{proof}

\pagebreak[3]

\begin{proposition}
\label{prop:arc_cpx_contractible}
For any $n\ge2$ the complex $\hatcharc(K_n)$ is contractible.
\end{proposition}

The proof here is essentially the same as the proof of the theorem
in~\cite{hatcher91}, so we will not be overly precise.  Indeed there
is only one extra step, which we will point out when it comes.

\begin{proof}
Fix an arc $\beta$, i.e., a vertex in $\hatcharc(K_n)$.  We will
retract $\hatcharc(K_n)$ to the star of~$\beta$.  We use the ``Hatcher
flow'' introduced in~\cite{hatcher91}.  Let
$\sigma=\{\alpha_0,\dots,\alpha_k\}$ be a simplex in $\hatcharc(K_n)$
and let $p$ be a point in $\sigma$, expressed in terms of barycentric
coordinates $p=\sum_{i=0}^k c_i\alpha_i$, with $c_i\ge0$ and
$\sum_{i=0}^k c_i=1$.  Interpret $p$ geometrically by saying that each
$\alpha_i$ is thickened to a ``band'' of thickness $c_i$.  Wherever
the bands cross $\beta$, pinch them into a single band of thickness
$\theta$.  Now the Hatcher flow is as follows.  At time $t\in[0,1]$,
push $p$ to the point $p_t$ obtained by leaving $(1-t)\theta$ worth of the band in place and pushing the remaining $t\theta$-thick part of the band all
the way to one end of~$\beta$; see Figure~\ref{fig:hatcher_flow}.  The
additional consideration we have to make is, if at any point we create
a new arc whose endpoints coincide, discard this from $p_t$.  This is
allowed, since if none of the $\alpha_i$ are loops then there will
always exist at least one non-loop arc used in $p_t$.  One checks that this flow is continuous and respects the face relation, and at time $t=1$ we have deformed $\hatcharc(K_n)$ into the
star of $\beta$, so we conclude that $\hatcharc(K_n)$ is contractible.
\end{proof}

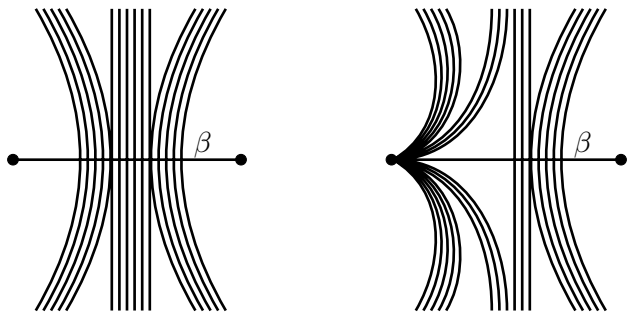
\begin{figure}
\centering
\begin{tikzpicture}

  \draw[line width=1pt]
   (0.5,0) -- (3.5,0)

   (0.8,-2) to [out=60, in=-60] (0.8,2)
   (0.9,-2) to [out=60, in=-60] (0.9,2)
   (1,-2) to [out=60, in=-60] (1,2)
   (1.1,-2) to [out=60, in=-60] (1.1,2)
   (1.2,-2) to [out=60, in=-60] (1.2,2)

   (1.8,-2) -- (1.8,2)
   (1.9,-2) -- (1.9,2)
   (2,-2) -- (2,2)
   (2.1,-2) -- (2.1,2)
   (2.2,-2) -- (2.2,2)
   (2.3,-2) -- (2.3,2)

   (2.9,-2) to [out=120, in=-120] (2.9,2)
   (3,-2) to [out=120, in=-120] (3,2)
   (3.1,-2) to [out=120, in=-120] (3.1,2)
   (3.2,-2) to [out=120, in=-120] (3.2,2)
   (3.3,-2) to [out=120, in=-120] (3.3,2);

  \filldraw 
   (0.5,0) circle (2pt)
   (3.5,0) circle (2pt);

  \node at (3,0.2) {$\beta$};

 \begin{scope}[xshift=5cm]
  \draw[line width=1pt]
   (0.5,0) -- (3.5,0)

   (0.8,-2) to [out=60, in=-30] (0.5,0) to [out=30, in=-60] (0.8,2)
   (0.9,-2) to [out=60, in=-25] (0.5,0) to [out=25, in=-60] (0.9,2)
   (1,-2) to [out=60, in=-20] (0.5,0) to [out=20, in=-60] (1,2)
   (1.1,-2) to [out=60, in=-15] (0.5,0) to [out=15, in=-60] (1.1,2)
   (1.2,-2) to [out=60, in=-10] (0.5,0) to [out=10, in=-60] (1.2,2)

   (1.8,-2) to [out=90, in=-10] (0.5,0) to [out=10, in=-90] (1.8,2)
   (1.9,-2) to [out=90, in=-5] (0.5,0) to [out=5, in=-90] (1.9,2)
   (2,-2) to [out=90, in=0] (0.5,0) to [out=0, in=-90] (2,2)
   (2.1,-2) -- (2.1,2)
   (2.2,-2) -- (2.2,2)
   (2.3,-2) -- (2.3,2)

   (2.9,-2) to [out=120, in=-120] (2.9,2)
   (3,-2) to [out=120, in=-120] (3,2)
   (3.1,-2) to [out=120, in=-120] (3.1,2)
   (3.2,-2) to [out=120, in=-120] (3.2,2)
   (3.3,-2) to [out=120, in=-120] (3.3,2);

  \filldraw 
   (0.48,0) circle (2pt)
   (3.5,0) circle (2pt);

  \node at (3,0.2) {$\beta$};
 \end{scope}
\end{tikzpicture}
\caption{The Hatcher flow.}
\label{fig:hatcher_flow}
\end{figure}

As a remark, note that the above proof yields contractibility for more
general~$\hatcharc(\Gamma)$; the only requirement is that there exists
a node of $\Gamma$
that shares an edge with every other node.

We now want to consider a subspace of $\hatcharc(K_n)$ that is related
to the matching complex of a complete graph, which we call the
matching complex on a surface.  In the next subsection, we will first
cover some background on matching complexes of graphs and establish
some results, and then we will inspect the surface version.

\subsection{Matching complexes}
\label{sec:match_cpxes}

Matching complexes of graphs are defined as follows.

\begin{definition}[Matching complex of a graph]
Let $\Gamma$ be a graph.  The \emph{matching complex} $\match(\Gamma)$
of $\Gamma$ is the simplicial complex with a $k$-simplex for every
collection~$\{e_0,\dots,e_k\}$ of $k+1$ pairwise disjoint edges, with
the face relation given by passing to subcollections.
\end{definition}

Observe that $\match(K_n)$ is non-empty if and only if $n\ge2$, and as
an exercise one can verify that it is connected for $n\ge5$.  As we
will see in Proposition~\ref{prop:matching_cpx_conn}, $\match(K_n)$ is
$(\nu(n)-1)$-connected, where we define $\nu(m)\defeq
\bigl\lfloor\frac{m+1}{3}\bigr\rfloor-1$ for any $m\in\Z$.

Before proving this, we need a technical lemma.

\begin{lemma}
\label{lem:play_with_floor}
For $m_1,\dots,m_{\ell}\in\Z$, we have
\begin{equation*}
\sum_{i=1}^{\ell}\nu(m_i)     
\ge \nu\biggl(\sum_{i=1}^{\ell}m_i-4(\ell-1)\biggr)\text{\,.}
\end{equation*}
\end{lemma}

\begin{proof}
We induct on $\ell$.  The case $\ell=1$ is trivially true.  In order
to prove the case~$\ell=2$ we need the following observation:
\begin{equation*}
\biggl\lfloor\frac{m_1}{3}\biggr\rfloor +
\biggl\lfloor\frac{m_2}{3}\biggr\rfloor \ge
\biggl\lfloor\frac{m_{1}+m_{2}+1}{3}\biggr\rfloor-1
\tag{$*$}
\end{equation*}
for any $m_{1}, m_{2}\in \Z$.  It suffices to consider
the cases where~$m_1$ and~$m_2$ are~$0$,~$1$ or~$2$, and these cases
are readily checked.  Thus we obtain
\begin{align*}
\nu(m_1)+\nu(m_2) & =
\biggl\lfloor\frac{m_1+1}{3} \biggr\rfloor
+ \biggl\lfloor\frac{m_2+1}{3}\biggr\rfloor-2\\
& \ge
\biggl\lfloor\frac{m_1+m_2+3}{3}\biggr\rfloor-3 &&
\text{(using the inequality~($*$))}
\\
& =
\biggl\lfloor\frac{m_1+m_2-3}{3}\biggr\rfloor-1
\\[1ex]
& =
\nu(m_1+m_2-4)
\\
\intertext{and this finishes the case $\ell=2$.  Finally, suppose that
$\ell>2$. Then}
\sum_{i=1}^{\ell}\nu(m_i)
& =
\sum_{i=1}^{\ell-1}\nu(m_i)+\nu(m_{\ell})
\\
& \ge
\nu\biggl(\sum_{i=1}^{\ell-1}m_i-4(\ell-2)\biggr) + \nu(m_{\ell})
&& \text{(by induction)}
\\
& \ge\nu\biggl(\sum_{i=1}^{\ell-1}m_i-4(\ell-2)+m_{\ell}-4\biggr)
&& \text{(by the $\ell=2$ case)}
\\
& =
\nu\biggl(\sum_{i=1}^{\ell}m_i-4(\ell-1)\biggr).
\tag*{\qedhere}
\end{align*}
\end{proof}

To prove that $\match(K_n)$ is $(\nu(n)-1)$-connected it will be
convenient to embed it into a contractible space, so we can use the
Morse Lemma.  Let $\hatch(K_n)$ be the simplicial complex with a
simplex for every subgraph of $K_n$ that has the same vertex set as
$K_n$ and has at least one edge, with face relation given by
inclusion.  Hence a~$0$-simplex in~$\hatch(K_n)$ is a subgraph with a
single edge, a~$1$-simplex has two edges, and so forth.  In fact, $\hatch(K_n)$ is isomorphic to
an~$\bigl(\binom{n}{2}-1\bigr)$-simplex, and so is contractible.
Think of~$\match(K_n)$ as a subcomplex of the contractible complex~$\hatch(K_n)$.  Consider a simplex~$\Gamma$ in $\hatch(K_n)$.  Let
$e(\Gamma)$ be the number of edges of $\Gamma$ and let~$r(\Gamma)$ be
the number of non-isolated nodes of $\Gamma$.  Define the
\emph{defect} of~$\Gamma$ to be the number~$d(\Gamma) \defeq
2e(\Gamma)-r(\Gamma)$.  This measures the failure of $\Gamma$ to be in
$\match(K_n)$, in that~$\match(K_n)$ is precisely the set of simplices
of~$\hatch(K_n)$ with defect~$0$.  Note that~$\match(K_n)$ already
contains every~$0$-simplex of~$\hatch(K_n)$.  Also note that a
subgraph~$\Gamma'$ of a graph~$\Gamma$ cannot have higher defect
than~$\Gamma$.  Now define a function~$h(\Gamma) \defeq
(d(\Gamma),-e(\Gamma))$, and consider its values ordered
lexicographically.  Think of~$h$ as a function on the vertex set of
$\hatch(K_n)'$, the barycentric subdivision of $\hatch(K_n)$.  Note
that adjacent vertices have distinct~$e$-values, and hence
distinct~$h$-values, so this is a height function on the vertex set of
$\hatch(K_n)'$.

Consider the descending link $\dlk(\Gamma)$ of $\Gamma$ in
$\hatch(K_n)'$.  There are two types of vertices in $\dlk(\Gamma)$,
namely graphs $\widetilde{\Gamma}>\Gamma$ with
$h(\widetilde{\Gamma})<h(\Gamma)$ and graphs $\Gamma'<\Gamma$
with~$h(\Gamma')<h(\Gamma)$.  Define the \emph{up-link} (respectively
\emph{down-link}) to be the full subcomplex of $\dlk(\Gamma)$ spanned
by vertices of the first type (respectively second type).  Observe
that $h(\widetilde{\Gamma})<h(\Gamma)$ is equivalent to
$d(\widetilde{\Gamma})=d(\Gamma)$, and $h(\Gamma')<h(\Gamma)$ is
equivalent to $d(\Gamma')<d(\Gamma)$.  Any graph in the down-link is a
subgraph of any graph in the up-link, so $\dlk(\Gamma)$ is the join of
the up-link and down-link. See Figure~\ref{fig:defect-1} for an idea of defect, up-link and down-link.

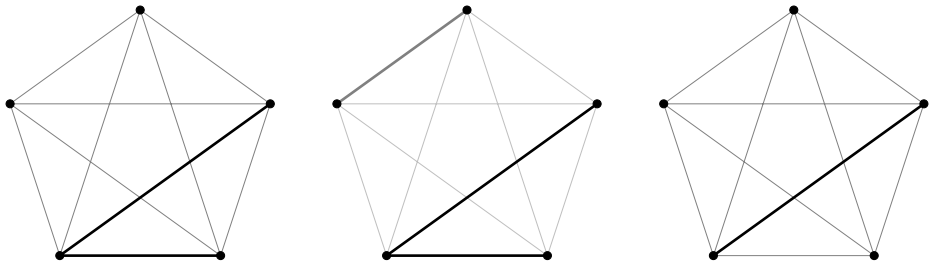
\begin{figure}[t]
\centering
\begin{tikzpicture}
        \coordinate (A) at ( 90:1.8);
	\coordinate (B) at (162:1.8);
	\coordinate (C) at (234:1.8);
	\coordinate (D) at (306:1.8);
	\coordinate (E) at ( 18:1.8);
	
	\draw[gray, very thin] (A) -- (B) -- (C) -- (D) -- (E) -- cycle
	(A) -- (C) -- (E) -- (B) -- (D) -- (A);
	
	\draw[line width=1pt] 
	(C) -- (E)
	(C) -- (D);
	
	\filldraw 
	(A) circle(1.5pt)
	(B) circle(1.5pt)
	(C) circle(1.5pt)
	(D) circle(1.5pt)
	(E) circle(1.5pt);
	
	\begin{scope}[xshift=4.3cm]
	    \coordinate (A) at ( 90:1.8);
	    \coordinate (B) at (162:1.8);
	    \coordinate (C) at (234:1.8);
	    \coordinate (D) at (306:1.8);
	    \coordinate (E) at ( 18:1.8);
	\end{scope}
	
	\draw[lightgray, very thin] (A) -- (B) -- (C) -- (D) -- (E) -- cycle
	(A) -- (C) -- (E) -- (B) -- (D) -- (A);
	
	\draw[line width=1pt] 
	(C) -- (E)
	(C) -- (D);
	\draw[line width=1pt, gray] 
	(A) -- (B);
	
	\filldraw 
	(A) circle(1.5pt)
	(B) circle(1.5pt)
	(C) circle(1.5pt)
	(D) circle(1.5pt)
	(E) circle(1.5pt);
	
	\begin{scope}[xshift=8.6cm]
	    \coordinate (A) at ( 90:1.8);
	    \coordinate (B) at (162:1.8);
	    \coordinate (C) at (234:1.8);
	    \coordinate (D) at (306:1.8);
	    \coordinate (E) at ( 18:1.8);
	\end{scope}
	    
	\draw[gray, very thin] (A) -- (B) -- (C) -- (D) -- (E) -- cycle
	(A) -- (C) -- (E) -- (B) -- (D) -- (A);
	
	\draw[line width=1pt] 
	(C) -- (E);
	
	\filldraw 
	(A) circle(1.5pt)
	(B) circle(1.5pt)
	(C) circle(1.5pt)
	(D) circle(1.5pt)
	(E) circle(1.5pt);
\end{tikzpicture}
\caption{Three vertices in
$\hatch(K_5)$.  From left to right: a graph~$\Gamma$ with defect~$1$,
a graph in the up-link of $\Gamma$ and a graph in the down-link
of~$\Gamma$.}
\label{fig:defect-1}
\end{figure}

\pagebreak[3]

\begin{proposition}
\label{prop:matching_cpx_conn}
The complex $\match(K_n)$ is $(\nu(n)-1)$-connected.
\end{proposition}

This result is well-known, see for example~\cite{athanasiadis04,
bjorner94}.  We will prove it here using a method that we will use
later to prove the main result of this section,
Theorem~\ref{thrm:surface_matching_conn}.

\begin{proof}
As a base case, $\match(K_n)$ is non-empty for $n\ge2$.  Now suppose
$n\ge5$.  Since~$\hatch(K_n)$ is contractible and its~$0$-skeleton is
already in~$\match(K_n)$, to show that~$\match(K_n)$ is
$(\nu(n)-1)$-connected it suffices by the Morse Lemma to show that for
any~$\Gamma$ with $e(\Gamma)\ge2$ and~$d(\Gamma)\ge1$, the descending
link~$\dlk(\Gamma)$ is~$(\nu(n)-1)$-connected.  First consider the
down-link.  A subgraph of~$\Gamma'<\Gamma$ fails to
be in the down-link precisely if each edge
in~$\Gamma\setminus\Gamma'$ is disjoint from every other edge
of~$\Gamma$, since then and only then do~$\Gamma$ and~$\Gamma'$ have
the same defect.  Let~$\Gamma_0$ be the subgraph of~$\Gamma$
consisting precisely of all such edges, if any exist.  The space of
all proper subgraphs of~$\Gamma$ is a combinatorial $(e(\Gamma)-2)$-sphere, and the
complement in this space of the down-link is either empty, or is
contractible with cone point~$\Gamma_0$.  Hence the down-link is
either an~$(e(\Gamma)-2)$-sphere or is contractible.  Now consider the
up-link.  The graphs in the up-link are given by adding edges
to~$\Gamma$ that are all disjoint from each other and from the edges
of~$\Gamma$.  Hence the up-link is isomorphic to~$\match(K_{n-r(\Gamma)})$, and
so is~$(\nu(n-r(\Gamma))-1)$-connected
by induction.  Since the down-link is $(e(\Gamma)-3)$-connected, this
tells us that $\dlk(\Gamma)$ is
$(e(\Gamma)+\nu(n-r(\Gamma))-2)$-connected.  Since~$e(\Gamma)\ge 2$
and $d(\Gamma)\ge 1$, we have
\begin{align*}
 e(\Gamma)+\nu(n-r(\Gamma))-2 &= \nu(n+3e(\Gamma)-r(\Gamma)-3)-1 \\
 &= \nu(n+d(\Gamma)+e(\Gamma)-3)-1 \\
 &\ge \nu(n)-1
\end{align*}
and so we conclude that $\dlk(\Gamma)$ is indeed $(\nu(n)-1)$-connected.
\end{proof}

\medskip

We now define the notion of the matching complex on a surface.  We
retain all the notation and definitions from
Section~\ref{sec:arc_cpx}, including the surface $S$ with points~$P$.
Fix a labeling of the points in~$P$ by the numbers $1,\dots,n$, and
identify~$P$ with the set of nodes of $K_n$, as above.

\begin{definition}[Matching complex on a surface]\label{def:match_cpx_surface}
Let $\matcharc(K_n)$ be the subcomplex of $\hatcharc(K_n)$ whose
simplices are given by arc systems whose arcs are pairwise disjoint
including at their endpoints.  For a subgraph $\Gamma$ of $K_n$ let
$\matcharc(\Gamma)$ be the preimage of $\match(\Gamma)$ under the
map $\matcharc(K_n)\to\match(K_n)$ that sends an arc with endpoints
labeled~$i$ and $j$ to the edge of $K_n$ with endpoints $i$ and $j$.
We call $\matcharc(\Gamma)$ the \emph{matching complex} on $(S,P)$
corresponding to $\Gamma$.
\end{definition}

Our goal now is to show that $\matcharc(K_n)$ is
$(\nu(n)-1)$-connected, just like $\match(K_n)$.  Our proof mimics the
proof of Proposition~\ref{prop:matching_cpx_conn}.  As a remark, in
the next section we will analyze the connectivity of $\matcharc(L_n)$ using different methods,
in Corollary~\ref{cor:surface_line_matching_conn}, and also give an
alternate proof that $\matcharc(K_n)$ is $(\nu(n)-1)$-connected.  The
proof that we give in this section for $\matcharc(K_n)$ is more in
line with the analogous situation for graphs, but does not generalize
readily to $\matcharc(L_n)$.

We begin by defining a height function~$h$ on the vertex set of the
barycentric subdivision~$\hatcharc(K_n)'$ of $\hatcharc(K_n)$, similar
to the height function on the vertex set of~$\hatch(K_n)'$.  In fact
we will reuse much of the notation from the $\match(K_n)$ case.  Let
$\sigma$ be a~$k$-simplex in $\hatcharc(K_n)$, represented by the arcs
$\alpha_0,\dots,\alpha_k$.  Choose these arcs to be pairwise disjoint,
except possibly at their endpoints.  Let $A(\sigma) \defeq
\alpha_0\cup\cdots\cup\alpha_k$ as a subspace of~$S$.  Also let
$a(\sigma) \defeq k+1$ be the number of arcs, let $r(\sigma)\defeq
|A\cap P|$ and define the \emph{defect} $d(\sigma) \defeq
2a(\sigma)-r(\sigma)$.  For example if every arc in $\sigma$ has the
same two points as endpoints then $d(\sigma)=2a(\sigma)-2$, the
maximum possible.  The defect is zero if and
only if $\sigma\in\matcharc(K_n)$, so in particular~$\matcharc(K_n)$
is a sublevel set.  Now define $h(\sigma)\defeq
(d(\sigma),-a(\sigma))$ with the lexicographic ordering.  In
particular, adjacent vertices in $\hatcharc(K_n)'$ have
distinct~$h$-values.  As in the case of~$\match(K_n)$, $\dlk(\sigma)$
decomposes as the join of an up-link and a down-link.  The up-link is
given by simplices $\widetilde{\sigma}>\sigma$ with
$d(\widetilde{\sigma})=d(\sigma)$, and the down-link is given by
faces~$\sigma'<\sigma$ with $d(\sigma')<d(\sigma)$; see
Figure~\ref{fig:arc-complex-link}.

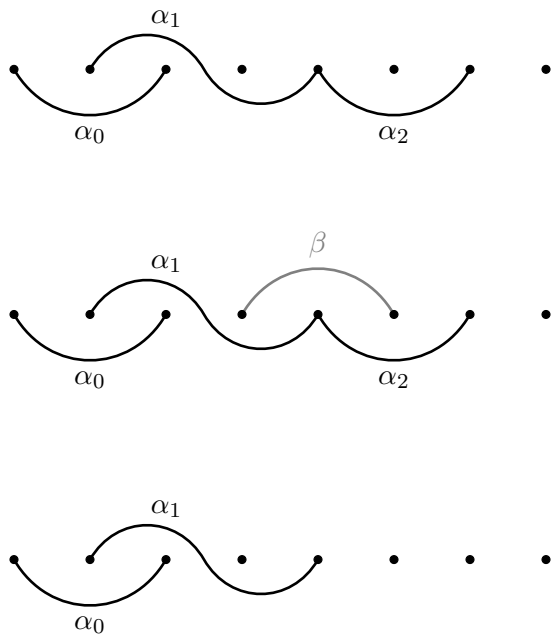
\begin{figure}
\centering
\begin{tikzpicture}
	\draw[line width=1pt]
	(0,0) to [out=-60, in=240,looseness=1.2]  
	      node[pos=0.5, below] {$\alpha_{0}$} (2,0)
	(1,0) to [out= 60, in=120,looseness=1.2] 
	      node[pos=0.65, above] {$\alpha_{1}$} (2.5,0)
	      to [out=-60, in=240,looseness=1.2] (4,0)
	(4,0) to [out=-60, in=240,looseness=1.2]  
	      node[pos=0.5, below] {$\alpha_{2}$} (6,0);
	
        \filldraw 
	(0,0) circle (1.5pt)
	(1,0) circle (1.5pt)
	(2,0) circle (1.5pt)
	(3,0) circle (1.5pt)
	(4,0) circle (1.5pt)
	(5,0) circle (1.5pt)
	(6,0) circle (1.5pt)
	(7,0) circle (1.5pt);  
	
	\begin{scope}[yshift=-3.25cm]
	    \draw[line width=1pt]
	    (0,0) to [out=-60, in=240,looseness=1.2]  
		  node[pos=0.5, below] {$\alpha_{0}$} (2,0)
	    (1,0) to [out= 60, in=120,looseness=1.2] 
		  node[pos=0.65, above] {$\alpha_{1}$} (2.5,0)
		  to [out=-60, in=240,looseness=1.2] (4,0)
	    (4,0) to [out=-60, in=240,looseness=1.2]  
		  node[pos=0.5, below] {$\alpha_{2}$} (6,0);
	    \draw[line width=1pt,gray]
	    (3,0) to [out=60, in=120,looseness=1.2]  
		  node[pos=0.5, above] {$\beta$} (5,0);
		
	    \filldraw 
	    (0,0) circle (1.5pt)
	    (1,0) circle (1.5pt)
	    (2,0) circle (1.5pt)
	    (3,0) circle (1.5pt)
	    (4,0) circle (1.5pt)
	    (5,0) circle (1.5pt)
	    (6,0) circle (1.5pt)
	    (7,0) circle (1.5pt);  
          \end{scope} 
	
	\begin{scope}[yshift=-6.5cm]
	    \draw[line width=1pt]
	    (0,0) to [out=-60, in=240,looseness=1.2]  
		  node[pos=0.5, below] {$\alpha_{0}$} (2,0)
	    (1,0) to [out= 60, in=120,looseness=1.2] 
		  node[pos=0.65, above] {$\alpha_{1}$} (2.5,0)
		  to [out=-60, in=240,looseness=1.2] (4,0);
	   
	    \filldraw 
	    (0,0) circle (1.5pt)
	    (1,0) circle (1.5pt)
	    (2,0) circle (1.5pt)
	    (3,0) circle (1.5pt)
	    (4,0) circle (1.5pt)
	    (5,0) circle (1.5pt)
	    (6,0) circle (1.5pt)
	    (7,0) circle (1.5pt);  
	\end{scope}
\end{tikzpicture}
\caption{From top to bottom: an
arc system $\sigma$ with defect~$1$, a simplex in the up-link of
$\sigma$ and a simplex in the down-link of~$\sigma$.}
\label{fig:arc-complex-link}
\end{figure}

\begin{theorem}
\label{thrm:surface_matching_conn}
The complex $\matcharc(K_n)$ is $(\nu(n)-1)$-connected.
\end{theorem}

\begin{proof}
The proof runs very similarly to the proof of
Proposition~\ref{prop:matching_cpx_conn}.  Dealing with the up-link is
the biggest difference.

As a base case, $\matcharc(K_n)$ is non-empty for $n\ge2$.  Now
suppose $n\ge5$.  Since~$\hatcharc(K_n)$ is contractible by
Proposition~\ref{prop:arc_cpx_contractible}, and the $0$-skeleton of
$\hatcharc(K_n)$ is already in $\matcharc(K_n)$, to show that
$\matcharc(K_n)$ is $(\nu(n)-1)$-connected it suffices by the Morse
Lemma to show that for any $k$-simplex~$\sigma$ in $\hatcharc(K_n)$
with $k\ge1$ and $d(\sigma)\ge1$, the descending link $\dlk(\sigma)$
is $(\nu(n)-1)$-connected.  First consider the down-link.  A
face~$\sigma'$ of~$\sigma$ fails to be in the down-link precisely when
each arc in~$\sigma\setminus\sigma'$ is disjoint from every other arc
of $\sigma$, since then and only then do $\sigma$ and $\sigma'$ have
the same defect.  Let $\sigma_0$ be the face of $\sigma$ consisting
precisely of all such arcs, if any exist.  The boundary of $\sigma$ is
an $(a(\sigma)-2)$-sphere.  The complement in this space of the
down-link is either empty, or is a cone with cone point
$\sigma_0$.  Hence the down-link is either an $(a(\sigma)-2)$-sphere
or is contractible.

Now consider the up-link.  The simplices in the up-link are given by
adding arcs to~$\sigma$ that are all disjoint from each other and from
the arcs in $\sigma$.  Consider the surface~$S\setminus A$, obtained by cutting out the arcs $\alpha_i$.  Denote the
connected components of $S\setminus A$ by $C_1,\dots,C_{\ell}$, and
let $n_i \defeq |C_i\cap P|$ for each~$1\le i\le \ell$.  Note that 
$\sum_{i=1}^{\ell}n_i + r(\sigma)= n$.  The up-link
is isomorphic to the join $\bigast_{i=1}^{\ell} \matcharc(K_{n_i})$,
which by induction is
$\bigl(\sum_{i=1}^{\ell}\nu(n_i)+\ell-2\bigr)$-connected.  Since the
down-link is~$(a(\sigma)-3)$-connected, this tells us
that~$\dlk(\sigma)$ is $\bigl(a(\sigma) - 2 + \sum_{i=1}^{\ell}
\nu(n_i) + \ell - 1\bigr)$-connected.  By
Lemma~\ref{lem:play_with_floor},
\begin{equation*}
a(\sigma) - 2 + \sum_{i=1}^{\ell} \nu(n_i) + \ell-1 \ge a(\sigma) - 2 +
\nu\biggl(\sum_{i=1}^{\ell} n_i - 4(\ell-1) \biggr) + \ell - 1
\end{equation*}
which equals
\begin{align*}
\nu\biggl(\sum_{i=1}^{\ell}n_i-4(\ell-1)+3a(\sigma)+3\ell-6\biggr)-1
& =
\nu(n-r(\sigma)-\ell+3a(\sigma)-2)-1
\\
& =
\nu(n+d(\sigma)+a(\sigma)-\ell-2)-1.
\end{align*}
Here the first equality follows because $\sum_{i=1}^{\ell}n_i =
n-r(\sigma)$, and the second equality follows from the definition of
defect, namely $d(\sigma)=2a(\sigma)-r(\sigma)$.  Since
$a(\sigma)\ge2$, this last quantity is at least
$\nu(n+d(\sigma)-\ell)-1$.

To show that $\dlk(\sigma)$ is $(\nu(n)-1)$-connected, it now suffices
to show that $d(\sigma)\ge \ell$.  By an Euler characteristic
argument, we know that $r(\sigma)-a(\sigma)+\ell\le 1+z$, where~$z$ is
defined to be the number of connected components of $A(\sigma) =
\bigcup_{i=0}^k \alpha_i$.  Clearly each connected component of
$A(\sigma)$ contains at least one arc, and since $d(\sigma)\ge 1$ at
least one component must have more than one arc.  Hence $1+z\le
a(\sigma)$, which implies that $r(\sigma)+\ell\le 2a(\sigma)$, and so
indeed $\ell\le d(\sigma)$.
\end{proof}

\subsection{Connectivity of $\matcharc(L_n)$}\label{sec:nbr_arcs}

Recall the family of linear graphs $L_n$ from
Section~\ref{sec:match_cpxes}.  In this section we analyze the
connectivity of $\matcharc(L_n)$, by following the line of approach
used in the proof of Proposition~5.2 in~\cite{putman13}.  This
approach can also be adapted to recover the connectivity of
$\matcharc(K_n)$.

We first need a lemma that allows us to make certain assumptions about
maps from spheres to~$\matcharc(\Gamma)$.  To state it we need to
recall some definitions.  By a \emph{combinatorial~$k$-sphere
(respectively~$k$-disk)} we mean a simplicial complex that can be
subdivided to be isomorphic to a subdivision of the boundary of a~$(k+1)$-simplex (respectively to a subdivision of a $k$-simplex).  An
$m$-dimensional \emph{combinatorial manifold} is an $m$-dimensional
simplicial complex in which the link of every simplex~$\sigma$ of
dimension~$k$ is a combinatorial $(m-k-1)$-sphere.  In an~$m$-dimensional \emph{combinatorial manifold with boundary} the link
of a $k$-simplex~$\sigma$ is allowed to be homeomorphic to a
combinatorial~$(m-k-1)$-disk; its \emph{boundary} consists of all the
simplices whose link is indeed a disk.

A simplicial map is called \emph{simplexwise injective} if its
restriction to any simplex is injective.

\begin{lemma}[Reduction to the simplexwise injective case]
\label{lem:injectifying}
Let $Y$ be a compact $m$-dimensional combinatorial manifold.  Let $X$ be a
simplicial complex and assume that the link of every $k$-simplex in
$X$ is $(m-2k-2)$-connected.  Let $\psi \colon Y \to X$ be a
simplicial map whose restriction to $\partial Y$ is simplexwise
injective.  Then after possibly subdividing the simplicial structure of $Y$, $\psi$ is
homotopic relative $\partial Y$ to a simplexwise injective map.
\end{lemma}

Compare the statement of the lemma to the statement of the claim in
the proof of Proposition~5.2 in~\cite{putman13}. As a remark, the assumption that $Y$ is compact is not necessary, but it makes the end of the proof simpler.

\begin{proof}
The proof is by induction on $m$ and the statement is trivial for
$m=0$.

\begin{figure}
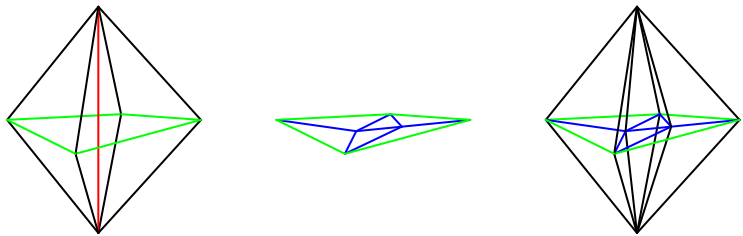

\includegraphics[scale=1.5]{injectifying-0}
\qquad
\includegraphics[scale=1.5]{injectifying-1}
\qquad
\includegraphics[scale=1.5]{injectifying-2}
\caption{Illustration of the proof of Lemma~\ref{lem:injectifying}.
The red edge is the simplex $\sigma$, that is, both of its vertices
are mapped to the same vertex under $\psi$.  The green circle is the
link of $\sigma$.  The link of~$\psi(\sigma)$ is simply connected by
assumption, so $\psi$ can be extended to a filling disk $B$ (blue).}
\label{fig:injectifying}
\end{figure}

If $\psi$ is not simplexwise injective, there exists a simplex whose
vertices do not map to pairwise distinct points.  In particular we can
choose a simplex $\sigma \subseteq Y$ of maximal dimension $k>0$ such
that for every vertex $x$ of $\sigma$ there is another vertex~$y$ of
$\sigma$ with $\psi(x) = \psi(y)$.  By assumption, $\sigma$ is not
contained in $\partial Y$.  Maximality of the dimension of $\sigma$
implies that the restriction of~$\psi$ to the $(m-k-1)$-sphere
$\lk_Y(\sigma)$ is simplexwise injective.  It also implies that
$\psi(\lk_Y(\sigma)) \subseteq \lk_X(\psi(\sigma))$.  Note further
that $\psi(\sigma)$ has dimension at most $(k-1)/2$.  Therefore its
link in $X$ is~$(m-k-1)$-connected by assumption.  Hence there is an
$(m-k)$-disk $B$ with $\partial B = \lk_Y(\sigma)$ and a map~$\varphi
\colon B \to \lk_X(\psi(\sigma))$ such that $\varphi|_{\partial B}$
coincides with $\psi |_{\lk_Y(\sigma)}$.  Inductively applying the
lemma, we may assume that $\varphi$ is simplexwise injective.

We now replace $Y$ by $Y'$, the space obtained by replacing the closed star of
$\sigma$ by~$B * \partial \sigma$.  The map $\psi' \colon Y' \to X$ is the map
that coincides with $\psi$ outside the open star of $\sigma$, coincides with
$\varphi$ on $B$ and is affine on simplices. It is clearly homotopic to~$\psi$,
since the image of~$B$ under~$\varphi$ is contained in $\lk_X(\psi(\sigma))$. 
Since the restriction of $\psi'$ to $B$ is simplexwise injective, the
restriction to any $k$-simplex of $B*\partial \sigma$ is injective. Since $Y$ is compact, by repeating this procedure finitely many times we eventually obtain a map that is simplexwise injective.
\end{proof}

\medskip

We now describe the general procedure we will use to analyze
$\matcharc(\Gamma)$ for a graph~$\Gamma$, after which we will look at
the specific graphs $K_n$ and (subgraphs of) $L_n$.

\subsection*{Our general procedure:} Pick an edge $e$ of $\Gamma$, say with endpoints $v$ and $w$.
Identify the vertex set of $\Gamma$ with the set $P$ of distinguished
points in the surface $S$.  Let
\begin{equation*}
q \colon \matcharc(\Gamma)^{(0)} \to \{0,1,2,3\}
\end{equation*}
be the function that sends an arc to $0$ if it has neither $v$ nor $w$
as endpoints,~$1$ if it has~$v$ but not $w$,~$2$ if it has $w$ but not $v$, and $3$
if it has both.  For any arc $\alpha$, say with endpoints $v_1$ and~$v_2$, the link of~$\alpha$ in $\matcharc(\Gamma)$ is isomorphic to
$\matcharc(\Gamma')$, where~$\Gamma'$ is the graph obtained from
$\Gamma$ by removing the stars of $v_1$ and~$v_2$.  Here the surface
on which the matching complex is being considered is not~$S$, but
rather~$S$ with a new boundary component obtained by ``slicing'' $S$
along $\alpha$.  The idea therefore is to build up from
$\matcharc(\Gamma)^{q=0}$ to $\matcharc(\Gamma)$ by gluing in missing
vertices (arcs) along their relative links, in increasing order of
$q$-value.  Since $\Gamma'$ has fewer vertices and edges than
$\Gamma$, these links will be highly connected by induction.  By the
Morse Lemma, it follows that the pair $(\matcharc(\Gamma),\matcharc(\Gamma)^{q=0})$ is highly connected. An important point to note is that even though $\matcharc(\Gamma)^{q=0}$ is also highly connected by induction, it is not typically as highly connected as we would like $\matcharc(\Gamma)$ to be (since we want the connectivity to go to infinity with the number of edges). For this reason we want to prove that the inclusion~$\iota \colon \matcharc(\Gamma)^{q=0}
\to \matcharc(\Gamma)$ induces the trivial map in~$\pi_k$
up to the desired connectivity bound for $\matcharc(\Gamma)$. This is accomplished as follows.
Fix an arc~$\beta$ with endpoints~$v$ and~$w$, and let $\overline{\psi} \colon
S^m \to \matcharc(\Gamma)^{q=0}$ be a simplicial map.  The goal is to show that $\psi \defeq \iota \circ \overline{\psi}$ is
homotopy equivalent to the constant map sending $S^m$ to~$\beta$, if
$m$ is not too large.  A variant of the Hatcher flow
becomes useful here; we look at arcs crossing $\beta$, choose one
closest to $w$, say $\alpha$, and ``push'' it over $w$ and off of
$\beta$, to the arc $\alpha'$.  See Figure~\ref{fig:matching_flow} for
a visualization.  We can homotope~$\psi$ to the map $\psi'$ using
$\alpha$ in lieu of $\alpha'$, assuming the mutual
link~$\lk(\alpha)\cap\lk(\alpha')$ is sufficiently highly connected,
which we can engineer to be true by induction if the structure of~$\Gamma$ is sufficiently easy to control.  A key step to making this
rigorous is being able to use Lemma~\ref{lem:injectifying}.

\medskip

We first carry out this program for the family of subgraphs of linear
graphs.  A key observation in this setting is that any node has
degree at most~$2$, and so removing the stars of two adjacent nodes
results in removing at most~$3$ edges.

It will be convenient to define $\eta(\ell) \defeq
\bigl\lfloor\frac{\ell-1}{4}\bigr\rfloor$ for $\ell\in\Z$.

\begin{theorem}\label{thrm:surface_subline_matching_conn}
Let $\Gamma_n$ be any subgraph of a linear graph, with $\Gamma_n$
having $n$ edges.  Then $\matcharc(\Gamma_n)$ is
$(\eta(n)-1)$-connected.
\end{theorem}

\begin{proof}
We induct on $n$, with the base case being that $\matcharc(\Gamma_n)$
is non-empty for~$n\ge1$, which is clear.  Now assume $n\ge5$. 
We will freely apply Lemma~\ref{lem:arc_representatives} to
represent simplices by systems of arcs.
Choose
an edge $e$ in $\Gamma_n$ with at least one endpoint of degree~$1$.
Let~$v$ and~$w$ be the endpoints of $e$, say $w$
has degree~$1$.  Let $q$ be the function defined above.  For an arc~$\alpha$ with $q(\alpha)=1$, the descending link of $\alpha$ with
respect to $q$ is isomorphic to $\matcharc(\Gamma_{n'})$, where~$\Gamma_{n'}$ is a subgraph of~$\Gamma_n$ with $n'$ edges.  Since every vertex
has degree at most~$2$, $n'\ge n-3$, so by induction~$\matcharc(\Gamma_{n'})$
is~$(\eta(n)-2)$-connected.  Similarly if
$q(\alpha)=3$ then the descending link of~$\alpha$ is isomorphic
to~$\matcharc(\Gamma_{n'})$, now with $n'\ge n-2$, so again induction
tells us that~$\matcharc(\Gamma_{n'})$ is $(\eta(n)-2)$-connected.
Note that $q(\alpha)=2$ actually does not occur in the present situation (we defined $q$ this way for the sake of consistency with the alternate proof of Theorem~\ref{thrm:surface_matching_conn} below).

The Morse Lemma now implies that the pair $(\matcharc(\Gamma_n),\matcharc(\Gamma_n)^{q=0})$ is $(\eta(n)-1)$-connected, that is, the inclusion ~$\iota\colon
\matcharc(\Gamma_n)^{q=0}\hookrightarrow\matcharc(\Gamma_n)$ induces an isomorphism in $\pi_m$ for $m \le \eta(n)-2$ and an epimorphism for $m = \eta(n) - 1$. We could now invoke induction and use that $\matcharc(\Gamma_n)^{q=0}$ is $(\eta(n)-2)$-connected to conclude that $\matcharc(\Gamma_n)$ is $(\eta(n)-2)$-connected as well. However, since we even want $\matcharc(\Gamma_n)$ to be $(\eta(n)-1)$-connected, we need a different argument and we may as well apply this for all $m$. We want to show that $\pi_m(\matcharc(\Gamma_n)^{q=0}\hookrightarrow\matcharc(\Gamma_n))$ is trivial
for $m<\eta(n)$. In other words, every sphere in $\matcharc(\Gamma_n)^{q=0}$ of dimension at most $(\eta(n)-1)$ can be collapsed in $\matcharc(\Gamma_n)$.

First we check a hypothesis on $\matcharc(\Gamma_n)$ that allows us to
apply Lemma~\ref{lem:injectifying}, namely that the link of a
$k$-simplex should be $(m-2k-2)$-connected.  A $k$-simplex~$\sigma$ is
determined by $k+1$ disjoint arcs.  Hence, the link of $\sigma$ is
isomorphic to $\matcharc(\Gamma_{n'})$ where~$n'$ is at least
$n-(3k+3)$.  By induction, this is $(\eta(n-3k-3)-1)$-connected.

Moreover,
\begin{align*}
\eta(n-3k-3)-1 &= \Big\lfloor \frac{n-3k-4}{4} \Big\rfloor -1
\\
&\ge
\frac{n-3k-4}{4} - 2
\\[1ex]
&\ge
\eta(n) - 2k - 3 \ge m-2k-2 \text{\,.}
\end{align*}

We conclude that the hypothesis of
Lemma~\ref{lem:injectifying} is satisfied.

Let $S^m$ be a combinatorial $m$-sphere.  Let
$\overline{\psi}\colon S^m\to\matcharc(\Gamma_n)^{q=0}$ be a simplicial map and
let $\psi \defeq \iota \circ \overline{\psi}$.  It suffices by simplicial
approximation~\cite[Theorem~3.4.8]{spanier66} to homotope~$\psi$ to a constant
map.  By Lemma~\ref{lem:injectifying} we may assume~$\psi$ is simplexwise
injective.  Fix an arc~$\beta$ with endpoints~$v$ and $w$. We claim that~$\psi$
can be homotoped in $\matcharc(\Gamma_n)$ to land in the star of $\beta$, which
will finish the proof.  We will proceed in a similar way to the Hatcher flow
used in the proof of Proposition~\ref{prop:arc_cpx_contractible}.  None of the
arcs in the image of~$\psi$ use $v$ or $w$ as vertices, but among the finitely
many such arcs, some might cross~$\beta$.  Pick the one, say $\alpha$,
intersecting $\beta$ at a point closest along $\beta$ to $w$, and let~$x$ be a
vertex of~$S^m$ mapping to $\alpha$.  By simplexwise injectivity, none of the
vertices in $\lk_{S^m}(x)$ map to $\alpha$. Let $\alpha'$ be the arc with the
same endpoints as $\alpha$ such that together $\alpha$ and $\alpha'$ bound a
disk whose interior contains no boundary components, punctures or points of $P$
other than $w$. See Figure~\ref{fig:matching_flow} for an example.  Note that
there is no edge from $\alpha$ to $\alpha'$, so none of the vertices
in~$\lk_{S^m}(x)$ map to $\alpha'$. Note also that
$\psi(\lk_{S^m}(x)) \subseteq \lk \alpha'$ by choice of $\alpha$.

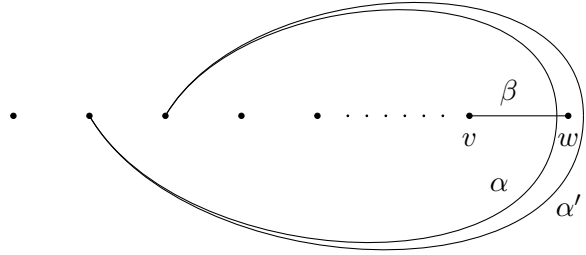
\begin{figure}
\centering
\begin{tikzpicture}

\coordinate (a) at (-2,0);
\coordinate (b) at (-1,0);
\coordinate (w) at (4.3,0);
\coordinate (v) at (3,0);
\coordinate (x) at (4.15,0);
\coordinate (y) at (4.5,0);

\draw (v) -- (w) node [pos=0.4,anchor=south] {$\beta$};
\draw (a) to [out=-60, in=-90] (x);
\draw (b) to [out= 60, in= 90] (x);
\draw (x) + (-0.5,-0.7) node [anchor=north east] {$\alpha$};
\draw (a) to [out=-60, in=-90] (y);
\draw (b) to [out= 60, in= 90] (y);
\draw (y) + (-0.5,-0.9) node [anchor=north west] {$\alpha'$};

\filldraw
	(a) circle (1pt) ++ (-1,0) circle (1pt)
	(b) circle (1pt) ++ (1,0) circle (1pt) ++ (1,0) circle (1pt)
	++ (0.40,0) circle (0.3pt) ++ (0.25,0) circle (0.3pt) ++ (0.25,0) circle
(0.3pt) ++ (0.25,0) circle (0.3pt) ++ (0.25,0) circle (0.3pt) ++ (0.25,0) circle
(0.3pt) 
	(v) circle (1pt) node[below=3pt] {$v$}
	(w) circle (1pt) node[below=3pt] {$w$};
\end{tikzpicture}
\caption{Pushing the arc $\alpha$ over the vertex $w$ to obtain the
arc~$\alpha'$, as described in the proof of
Theorem~\ref{thrm:surface_subline_matching_conn}.}
\label{fig:matching_flow}
\end{figure}
 
Define a simplicial map $\psi'\colon S^m\to\matcharc(\Gamma_n)$ that sends
the vertex $x$ to $\alpha'$ and sends all other vertices $y$ to
$\psi(y)$. We claim that
we can homotope $\psi$ to $\psi'$.  Once we do this, we will have
reduced the number of crossings with $\beta$, and so continuing this
procedure we will have homotoped our map so as to land in the star of
$\beta$, finishing the proof.
 
The mutual link $\lk(\alpha)\cap\lk(\alpha')$ is isomorphic to
$\matcharc(\Gamma_{n'})$, where $\Gamma_{n'}$ now is the graph
obtained from $\Gamma_n$ by removing $e$, and removing any edge
sharing an endpoint with an endpoint of~$\alpha$.  Here~$n'$ is the
number of edges of the resulting graph.  Since~$\Gamma_n$ is a
subgraph of a linear graph, we have thrown out at most~$4$ edges, and
so $n'\ge n-4$.  Hence by induction $\lk(\alpha)\cap\lk(\alpha')$ is
$(\eta(n)-2)$-connected, and in particular~$(m-1)$-connected.  Since
$\lk_{S^m}(x)$ is an~$(m-1)$-sphere, this tells us that there exists
an~$m$-disk $B$ with~$\partial B=\lk_{S^m}(x)$ and a simplicial
map~$\varphi\colon B \to \lk(\alpha)\cap\lk(\alpha')$ so that $\varphi$
restricted to $\partial B$ coincides with $\psi$ restricted
to~$\lk_{S^m}(x)$.  Since the image of~$B$ under~$\varphi$ is
contained in $\lk(\alpha)$, we can homotope~$\psi$,
replacing~$\psi|_{\st_{S^m}(x)}$ with $\varphi$.  Since the image
of~$B$ under~$\varphi$ is contained in $\lk(\alpha')$, we can
similarly homotope $\psi'$, replacing~$\psi'|_{\st_{S^m}(x)}$
with~$\varphi$.  These both yield the same map, so we are finished.
\end{proof}

\begin{corollary}\label{cor:surface_line_matching_conn}
$\matcharc(L_n)$ is $(\eta(n)-1)$-connected.\qed
\end{corollary}

As a remark, we expect that a better connectivity bound should be
possible.  Indeed, one can check that~$\matcharc(L_n)$ is already
connected for $n\ge4$, and that~$\match(L_n)$ is
$(\nu(n)-1)$-connected, which for large~$n$ is stronger than being $(\eta(n)-1)$-connected.
For now however, we will content ourselves with this bound.

\medskip

Using these techniques, we can also recover the connectivity of
$\matcharc(K_n)$.

\begin{proof}[Alternate proof of Theorem~\ref{thrm:surface_matching_conn}]
The base case is that $\matcharc(K_n)\neq\emptyset$ for $n\ge 2$,
which is clear.  Let $n\ge 5$.  Choose any edge $e$, with endpoints
$v$ and $w$.  Let $q$ be as above.  For an arc~$\alpha$ with
$q(\alpha)=1$, the descending link of $\alpha$ is isomorphic
to~$\matcharc(K_{n-3})$.  If $q(\alpha)=2$ or~$3$, the descending link is
isomorphic to $\matcharc(K_{n-2})$.  In any case, by induction all
descending links are $(\nu(n)-2)$-connected.  Hence we need only check
that~$\iota\colon \matcharc(K_n)^{q=0} \to \matcharc(K_n)$ induces the
trivial map in~$\pi_m$ for~$m<\nu(n)$.
 
First we check the hypothesis of Lemma~\ref{lem:injectifying}.  The
link of a $k$-simplex is a copy of~$\matcharc(K_{n-2k-2})$, which by
induction is $(\nu(n-2k-2)-1)$-connected.  We need this to be bounded
below by $m-2k-2$.  Indeed,
\begin{equation*}
\nu(n-2k-2)-1 \ge \frac{n-2k-4}{3}-2 \ge \nu(n)-2k-3 \ge m-2k-2 \text{\,.}
\end{equation*}
 
Now we consider a simplicial map $\overline{\psi}\colon S^m \to \matcharc(K_n)^{q=0}$, with $\psi \defeq \iota \circ \overline{\psi}$.  We claim that we can homotope
$\psi$ to a constant map.  By the same argument as in the proof
of Theorem~\ref{thrm:surface_subline_matching_conn}, the problem
reduces to inspecting the mutual link~$\lk(\alpha)\cap\lk(\alpha')$,
where $\alpha$ and $\alpha'$ are again as in
Figure~\ref{fig:matching_flow}.  This mutual link is isomorphic to
$\matcharc(K_{n-3})$, since compatible arcs may use any endpoints
other than the endpoints of $\alpha$, or the point $w$.  Hence by
induction $\lk(\alpha)\cap\lk(\alpha')$ is~$(\nu(n)-2)$-connected, and
by the same argument as in the proof of
Theorem~\ref{thrm:surface_subline_matching_conn}, we can eventually
homotope $\psi$ to land in the star of $\beta$, so we are done.
\end{proof}

\subsection{Cyclic graphs}
\label{sec:circle}

Let $C_n$ be the cyclic graph with $n$ nodes, labeled $1$ through~$n$
in sequence.  If $\alpha$ is an arc in $\matcharc(C_n)$ with endpoints
$1$ and $n$, the relative link of $\alpha$ is a copy of
$\matcharc(L_{n-3})$.  Gluing in these arcs, in any order, we build up
from~$\matcharc(L_{n-1})$ to $\matcharc(C_n)$.  Hence it is immediate
from Corollary~\ref{cor:surface_line_matching_conn} and the Morse
Lemma that $\matcharc(C_n)$ is $(\eta(n-1)-1)$-connected. The upshot of
this is that the methods of the present article could also be used to prove that $\Tbr$, ``braided~$T$'', is of type~$\F_\infty$. As far as we know, this group has yet to appear in the literature, so we will not say any more about $\Tbr$ here.


\section{Descending links in the Stein space}
\label{sec:desc_link_conn}

We now return to the Stein space $X$ from Section~\ref{sec:def_stein_space} and
inspect the descending links of vertices with respect to the height
function~$f$.   As explained before, the descending link of a vertex~$x$ with
$f(x)=n$ is isomorphic to the complex $\elbraigecpx_n$ of dangling elementary $n$-braiges
$[(1_n,b,F_J^{(n-|J|)})]$ with~$J\neq\emptyset$.  The idea now is to construct a
projection $\elbraigecpx_n \to \matcharc(K_n)$ and then, having calculated the
connectivity of $\matcharc(K_n)$ in the previous section, use tools of Quillen
\cite{quillen78} to obtain the connectivity of $\elbraigecpx_n$.  As usual we will wait
until the end of the section to mention the ``pure'' case.

The first key observation is that a matching on a linear graph encodes the same
information as an elementary forest. Recall that $L_{n-1}$ is the linear graph
with~$n$ vertices, labeled~$1$ through~$n$, and $n-1$ edges, one connecting~$i$
to~$i+1$ for each~$1\le i<n$. Let $\match(L_{n-1})$ be the matching complex of
$L_{n-1}$.

\begin{observation}\label{obs:matchings_to_forests}
 Elementary forests with~$n$ leaves correspond bijectively to simplices of
$\match(L_{n-1})$. Under the identification, carets correspond to edges. See
Figure~\ref{fig:matchings_to_forests} for an example. 
\end{observation}

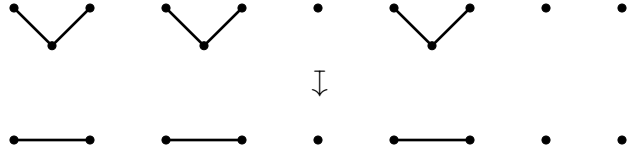
\begin{figure}[t]
\centering
\vskip2ex
\begin{tikzpicture}
  \filldraw
   (0,0) circle (1.5pt)
   (1,0) circle (1.5pt)
   (2,0) circle (1.5pt)
   (3,0) circle (1.5pt)
   (4,0) circle (1.5pt)
   (5,0) circle (1.5pt)
   (6,0) circle (1.5pt)
   (7,0) circle (1.5pt)
   (8,0) circle (1.5pt)

   (0.5,-0.5) circle (1.5pt)
   (2.5,-0.5) circle (1.5pt)
   (5.5,-0.5) circle (1.5pt);

  \draw[line width=1pt]
   (0,0) -- (0.5,-0.5) -- (1,0)
   (2,0) -- (2.5,-0.5) -- (3,0)
   (5,0) -- (5.5,-0.5) -- (6,0);
  
  \node[rotate=-90] at (4,-1) {$\mapsto$};

 \begin{scope}[yshift=-1.75cm]
  \filldraw 
   (0,0) circle (1.5pt)
   (1,0) circle (1.5pt)
   (2,0) circle (1.5pt)
   (3,0) circle (1.5pt)
   (4,0) circle (1.5pt)
   (5,0) circle (1.5pt)
   (6,0) circle (1.5pt)
   (7,0) circle (1.5pt)
   (8,0) circle (1.5pt);
  
  \draw[line width=1pt]
   (0,0) -- (1,0)   (2,0) -- (3,0)   (5,0) -- (6,0);
 \end{scope}
\end{tikzpicture}
\vskip2ex
\caption{An example of the bijective correspondence between elementary forests
with~$9$ leaves and simplices of $\match(L_8)$.}
\label{fig:matchings_to_forests}
\end{figure}

In light of the observation, we can denote an elementary~$n$-braige by
$(b,\Gamma)$, where~$b\in B_n$ and~$\Gamma$ is a simplex in~$\match(L_{n-1})$.
As usual, the equivalence class under dangling will be denoted $[(b,\Gamma)]$.

Let~$S$ be the unit disk, and fix an embedding $L_{n-1} \hookrightarrow S$ of
the linear graph with~$n-1$ edges into~$S$. Let~$P$ be the image of the vertex
set, so~$P$ is a set of~$n$ points in~$S$, labeled~$1$ through~$n$. With these
data in place we can consider~$\matcharc(K_n)$, the matching complex on $(S,P)$,
and we have an induced embedding of simplicial complexes $\match(L_{n-1})
\hookrightarrow \matcharc(K_n)$.  The braid group $B_n$ on $n$ strands is
isomorphic to the mapping class group of the~$n$-punctured disc $D_n$,
cf.~\cite{birman74}. Since $S\setminus P=D_n$, we have an action of $B_n$
on~$\matcharc(K_n)$.  In what follows it will be convenient to consider this as
a right action (much as dangling is a right action on braiges), so for~$b\in
B_n$ and $\sigma\in\matcharc(K_n)$ we will write~$(\sigma)b$ to denote the image
of~$\sigma$ under~$b$.

Define a map $\pi$ from $\elbraigecpx_n$ to $\matcharc(K_n)$ as follows. We
view~$\match(L_{n-1})$ as a subcomplex of~$\matcharc(K_n)$, so we can associate
to any elementary~$n$-braige~$(b,\Gamma)$ the arc complex $(\Gamma)b^{-1}$ in
$\matcharc(K_n)$. By construction, the map $(b,\Gamma)\mapsto (\Gamma)b^{-1}$ is
well defined on equivalence classes under dangling, so we obtain a simplicial
map \begin{align*}
\pi \colon \elbraigecpx_n & \to \matcharc(K_n) \\
[(b,\Gamma)] & \mapsto (\Gamma)b^{-1} \text{ .}
\end{align*}
Note that~$\pi$ is surjective, but not injective.

One can visualize this map by considering the merges as arcs, then ``combing
straight'' the braid and seeing where the arcs are taken, as in
Figure~\ref{fig:BMD-to-arc-cplx}. Note that the resulting simplex $(\Gamma)b^{-1}$
of $\matcharc(K_n)$ has the same dimension as the simplex~$[(b,\Gamma)]$ of
$\elbraigecpx_n$, namely one less than the number of edges in~$\Gamma$.

\begin{figure}[t]
\centering
\begin{tikzpicture}
    \clip(-.30,0.4) rectangle (1.80,-2.95);
    
    \draw[line width=3pt, white] (0.5,0) to [out=270, in=90] (0,-2);
    \draw[line width=1pt] (0.5,0) to [out=270, in=90] (0,-2);
    
    \draw[line width=3pt, white] (0,0) to [out=270, in=90] (1,-2);    
    \draw[line width=1pt] (0,0) to [out=270, in=90] (1,-2);    
    
    \draw[line width=3pt, white] (1,0) to [out=270, in=90] (0.5,-2);
    \draw[line width=1pt] (1,0) to [out=270, in=90] (0.5,-2);
    
    \draw[line width=1pt] (1.5,0) to (1.5,-2);
    
    \draw[line width=1pt] (0,-2) -- (0,-2.1) -- (0.25,-2.5) -- 
    (0.5,-2.1) -- (0.5,-2.0);
    
    \draw[line width=1pt] (1,-2) -- (1,-2.1) -- (1.25,-2.5) -- 
    (1.5,-2.1) -- (1.5,-2.0);
    
    \filldraw[red] 
    (0,-2.1) circle (1pt) 
    (0.5,-2.1) circle (1pt)
    (1,-2.1) circle (1pt) 
    (1.5,-2.1) circle (1pt);
    
    \draw[line width=1pt] (-0.25,0) -- (1.75,0);
    
\end{tikzpicture}
\qquad
\begin{tikzpicture}
    \clip(-.30,0.4) rectangle (1.80,-2.95);
    \draw[line width=3pt, white] (0.5,0) to [out=270, in=90] (0,-2);
    \draw[line width=1pt] (0.5,0) to [out=270, in=90] (0,-2);
    
    \draw[line width=3pt, white] (0,0) to [out=270, in=90] (1,-2);    
    \draw[line width=1pt] (0,0) to [out=270, in=90] (1,-2);    
    
    \draw[line width=3pt, white] (1,0) to [out=270, in=90] (0.5,-2);
    \draw[line width=1pt] (1,0) to [out=270, in=90] (0.5,-2);
    
    \draw[line width=1pt] (1.5,0) to (1.5,-2);
    
    \draw[line width=1pt] (0,-2) -- (0,-2.5)   
    (0.5,-2.5) -- (0.5,-2.0);
    
    \draw[line width=1pt] (1,-2) -- (1,-2.5) 
    (1.5,-2.5) -- (1.5,-2.0);
    
    \draw[line width=1pt,red] 
    (0,-2.5) -- (0.5,-2.5)
    (1,-2.5) -- (1.5,-2.5); 
    \filldraw[red] 
    (0,-2.5) circle (1pt) -- (0.5,-2.5) circle (1pt)
    (1,-2.5) circle (1pt) -- (1.5,-2.5) circle (1pt); 
    
    \draw[line width=1pt] (-0.25,0) -- (1.75,0);

\end{tikzpicture}
\qquad
\begin{tikzpicture}
    \clip(-.30,0.4) rectangle (1.80,-2.95);
    \draw[line width=1pt,red] 
    (0,-2.5) to [out=-30, in=210] (1,-2.5)
    (0.5,-2.5) to [out=30, in=150] (1.5,-2.5); 
    \draw[line width=3pt,white] (1,-2) -- (1,-2.4);
    
    \draw[line width=3pt, white] (0.5,-1) to [out=270, in=90] (0,-2);
    \draw[line width=1pt] (0.5,0) -- (0.5,-1) to [out=270, in=90] (0,-2);
    
    \draw[line width=3pt, white] (0,-1) to [out=270, in=90] (0.5,-2);    
    \draw[line width=1pt] (0,0) -- (0,-1) to [out=270, in=90] (0.5,-2);    
    
    \draw[line width=3pt, white] (1,0) to [out=270, in=90] (1,-2);
    \draw[line width=1pt] (1,0) to [out=270, in=90] (1,-2);
    
    \draw[line width=3pt, white] (1.5,0) to (1.5,-2);
    \draw[line width=1pt] (1.5,0) to (1.5,-2);
    
    \draw[line width=1pt] (0,-2) -- (0,-2.5)   
    (0.5,-2.5) -- (0.5,-2.0);

    \draw[line width=1pt] (1,-2) -- (1,-2.5)
    (1.5,-2.5) -- (1.5,-2.0);
    \filldraw[red] 
    (0,-2.5) circle (1pt) (0.5,-2.5) circle (1pt)
    (1,-2.5) circle (1pt) (1.5,-2.5) circle (1pt); 

    \draw[line width=1pt] (-0.25,0) -- (1.75,0);

\end{tikzpicture}
\qquad
\begin{tikzpicture}
    \clip(-.30,0.4) rectangle (1.80,-2.95);
    \draw[line width=1pt,red]
    (-0.25,-2.5) to [out=90, in=100] (0.5,-2.5)
    (-0.25,-2.5) to [out=-90, in=230] (1,-2.5)
    (0.0,-2.5) to [out=300, in=225] (0.75,-2.5)
    (0.75,-2.5) to [out=45, in=150] (1.5,-2.5); 
    \draw[line width=3pt,white] (1,-2) -- (1,-2.4);
    \draw[line width=3pt,white] (0,-2) -- (0,-2.4);

    \draw[line width=3pt, white] (0.5,0) to [out=270, in=90] (0.5,-2);
    \draw[line width=1pt] (0.5,0) to [out=270, in=90] (0.5,-2);
    
    \draw[line width=3pt, white] (0,0) to [out=270, in=90] (0,-2);    
    \draw[line width=1pt] (0,0) to [out=270, in=90] (0,-2);    
    
    \draw[line width=3pt, white] (1,0) to [out=270, in=90] (1,-2);
    \draw[line width=1pt] (1,0) to [out=270, in=90] (1,-2);
    
    \draw[line width=3pt, white] (1.5,0) to (1.5,-2);
    \draw[line width=1pt] (1.5,0) to (1.5,-2);
    
    \draw[line width=1pt] (0,-2) -- (0,-2.5)   
    (0.5,-2.5) -- (0.5,-2.0);

    \draw[line width=1pt] (1,-2) -- (1,-2.5)
    (1.5,-2.5) -- (1.5,-2.0);
    \filldraw[red] 
    (0,-2.5) circle (1pt) (0.5,-2.5) circle (1pt)
    (1,-2.5) circle (1pt) (1.5,-2.5) circle (1pt); 

    \draw[line width=1pt] (-0.25,0) -- (1.75,0);
    
\end{tikzpicture}
\caption{From braiges to arc
systems.  From left to right the pictures show the process of
``combing straight'' the braid.}
\label{fig:BMD-to-arc-cplx}
\end{figure}
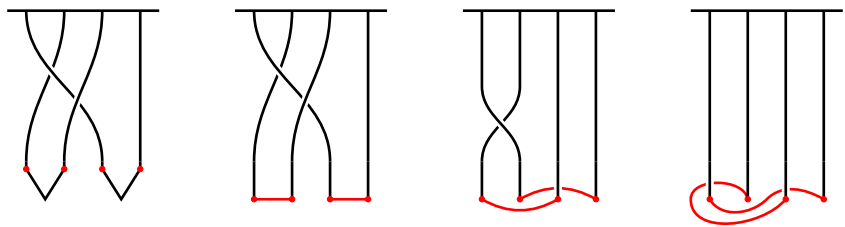

\medskip

The next lemma and proposition are concerned with the fibers of~$\pi$.

\pagebreak[3]

\begin{lemma}\label{lem:spx_fiber}
 Let~$E$ and~$\Gamma$ be simplices in~$\match(L_{n-1})$, such that~$E$ has one edge and~$\Gamma$ has~$e(\Gamma)$ edges. Let $[(b,E)]$ and $[(c,\Gamma)]$ be dangling elementary~$n$-braiges. Suppose that their images under the map~$\pi$ are contained in a simplex of $\matcharc(K_n)$.  Then there exists a simplex in $\elbraigecpx_n$ that contains~$[(b,E)]$ and~$[(c,\Gamma)]$.
\end{lemma}

\begin{proof}
We may assume that $[(b,E)]$ is not contained in $[(c,\Gamma)]$.

There is an action of $B_n$ on $\elbraigecpx_n$ (``from above''), given by
$b'[(c',\Gamma')]=[(b'c',\Gamma')]$.  One can check that for each $k\ge 0$, this
action is transitive on the~$k$-simplices of~$\elbraigecpx_n$. We can therefore assume
without loss of generality that $c=\id$, and $\Gamma$ is the subgraph
of~$L_{n-1}$ whose edges are precisely those connecting~$j$ to~$j+1$, for
$j\in\{1,3,\dots,2e(\Gamma)-1\}$.

Now there is an arc $\alpha$ representing~$\pi([(b,E)])$ that is disjoint from $\Gamma$.  
This
disjointness ensures that, after dangling, we can assume the following condition
on~$b$: for each edge of $\Gamma$, say with endpoints~$j$ and~$j+1$, $b$ can be
represented as a braid in such a way that the~$j^{\text{th}}$ and
$(j+1)^{\text{st}}$ strands of~$b$ run straight down, parallel to each other,
and no strands cross between them.  In particular $[(b,\Gamma)]=[(\id,\Gamma)]$,
so $[(b,\Gamma\cup E)]$ is a simplex in $\elbraigecpx_n$ with $[(b,E)]$
and~$[(\id,\Gamma)]$ as faces.
\end{proof}

\begin{proposition}\label{prop:fibers}
Let $\sigma$ be a $k$-simplex in $\matcharc(K_n)$ with vertices $v_0,
\ldots, v_k$.  Then
\begin{equation*}
\pi^{-1}(\sigma)=\Bigast_{j=0}^k\pi^{-1}(v_j).
\end{equation*}
In particular $\pi^{-1}(\sigma)$ is $k$-spherical.
\end{proposition}

\begin{proof}
The equation expresses an equality of abstract simplicial complexes with the same vertex set.

``$\subseteq$'':\quad This inclusion is just saying that vertices in $\pi^{-1}(\sigma)$ that are connected by an edge map to distinct vertices under $\pi$, which is clear.

\smallskip

``$\supseteq$'':\quad The $0$-skeleton of
$\bigast_{j=0}^k\pi^{-1}(v_j)$ is automatically contained in $\pi^{-1}(\sigma)$.  Now
assume that the same is true of the $r$-skeleton, for some $r\ge0$.  Let
$\tau$ be an $(r+1)$-simplex in~$\bigast_{j=0}^k\pi^{-1}(v_j)$, and
decompose $\tau$ as the join of a vertex $[(b,E)]$ and
an~$r$-simplex~$[(c,\Gamma)]$.  By induction, these are both in
$\pi^{-1}(\sigma)$, and by Lemma~\ref{lem:spx_fiber} they share a simplex
in~$\elbraigecpx_n$.  The minimal dimensional such simplex maps to
$\sigma$ under $\pi$, so we are done.
\end{proof}

Recall that in Sections~\ref{sec:match_cpxes} and~\ref{sec:nbr_arcs}
we defined the integers $\nu(n) = \bigl\lfloor
\frac{n+1}{3}\bigr\rfloor - 1$ and~$\eta(n) = \bigl\lfloor
\frac{n-1}{4}\bigr\rfloor$.

\begin{corollary}[Connectivity of descending links]
\label{cor:desc_link_conn}
$\elbraigecpx_n$ is $(\nu(n)-1)$-connected.  Hence for any vertex~$x$ in $X$ with~$f(x)=n$, $\dlk(x)$ is
$(\nu(n)-1)$-connected.
\end{corollary}

\begin{proof}
We know that $\matcharc(K_n)$ is $(\nu(n)-1)$-connected by Proposition~\ref{prop:matching_cpx_conn}.  For any
$k$-simplex $\sigma$ in $\matcharc(K_n)$, $\pi^{-1}(\sigma)$
is~$(k-1)$-connected by Proposition~\ref{prop:fibers}.  Also, $\lk(\sigma)$ is isomorphic to $\matcharc(K_{n-2k-2})$, which is $(\nu(n-2k-2)-1)$-connected and
hence $(\nu(n)-k-2)$-connected.  It follows
from~\cite[Theorem~9.1]{quillen78} that $\elbraigecpx_n$ is
$(\nu(n)-1)$-connected.
\end{proof}

In the pure case, we consider descending links of vertices in
$X(\Fbr)$.  For a vertex~$x$ with~$n+1$ feet, $\dlk(x)$ is isomorphic
to the poset $\elpbraigecpx_{n+1}$ of dangling elementary pure
$(n+1)$-braiges.  This projects onto the complex $\matcharc(L_n)$,
using an analogous projection as from $\elbraigecpx_n$ to
$\matcharc(K_n)$.  (Recall that $L_n$ is indexed by the number of
edges, not nodes.)  By the same argument as in the previous proof, we
can get the connectivity of $\elpbraigecpx_{n+1}$ from that of
$\matcharc(L_n)$.

\begin{corollary}[Pure case]\label{cor:pure_desc_lk_conn}
$\elpbraigecpx_{n+1}$ is $(\eta(n)-1)$-connected.\qed
\end{corollary}

From the above corollaries and the Morse Lemma, we conclude the following

\begin{corollary}[Connectivity of pairs in the filtration]\label{cor:pairs_conn}
For each $n\ge1$, the pair $(X^{\le n},X^{<n})$ is~$\nu(n)$-connected
and the pair $(X(\Fbr)^{\le n},X(\Fbr)^{<n})$ is~$\eta(n-1)$-connected.\qed
\end{corollary}


\section{Proof of the Main Theorem}
\label{sec:proof_main_theorem}

We are now in a position to prove our main theorem.

\begin{proof}[Proof of Main Theorem]
Consider the action of~$\Vbr$ on the complex~$X$, which is
contractible by Corollary~\ref{cor:stein_space_cible}.  We want to
apply Brown's Criterion.  By Corollary~\ref{cor:cell_stabs}, all cell
stabilizers are of type~$\F_\infty$.  By
Lemma~\ref{lem:cocompactness}, each~$X^{\le n}$ is finite
modulo~$\Vbr$, and by Corollary~\ref{cor:pairs_conn} the connectivity
of the pair $(X^{\le n},X^{<n})$ tends to~$\infty$ as~$n$ tends
to~$\infty$.  Hence Brown's Criterion tells us that~$\Vbr$ is of
type~$\F_\infty$.  A parallel argument applies to~$\Fbr$ acting on
$X(\Fbr)$, and so~$\Fbr$ is of type~$\F_\infty$.
\end{proof}


\vspace{2\bigskipamount}

\clearpage

\markleft{\uppercase{K.-U. Bux, S. Witzel, and M.~C.~B.~Zaremsky}}

\vspace{2\bigskipamount}

\clearpage

\thispagestyle{empty}
\markright{\uppercase{Erratum to: The braided Thompson's groups are of type~$\F_\infty$}}

\setlength{\baselineskip}{\stdbskip}

\begin{center}
    \textbf{\uppercase{Erratum to: The braided Thompson's groups are of type~$\F_\infty$}}
    
    \bigskip
    
    \uppercase{\small Kai-Uwe Bux, Stefan Witzel, and Matthew C.~B.~Zaremsky}
    
    \bigskip
\end{center}

\setlength{\baselineskip}{\bskipmaintext}

\noindent
There is a mistake in Lemma~3.9, which has no consequences for the rest of the article. The assumption that the link of every $k$-simplex in $X$ be $(m-2k-2)$-connected is insufficient to get the induction step to work, and needs to be replaced by $(m-k-2)$-connected. The correct statement therefore reads:

\setcounter{section}{3}
\setcounter{theorem}{8}
\begin{lemma}[Reduction to the simplexwise injective case]
\label{lem:injectifying}
Let $Y$ be a compact $m$-dimensional combinatorial manifold.  Let $X$ be a
simplicial complex and assume that the link of every $k$-simplex in
$X$ is $(m-k-2)$-connected.  Let $\psi \colon Y \to X$ be a
simplicial map whose restriction to $\partial Y$ is simplexwise
injective.  Then after possibly subdividing the simplicial structure of $Y$, $\psi$ is
homotopic relative $\partial Y$ to a simplexwise injective map.
\end{lemma}

We became aware of this mistake through discussions with S\o ren Galatius involving an equivalent result with the correct bound \cite[Theorem~2.4]{GalatiusRandalWilliams18}.

In the new formulation the old proof applies verbatim, but we extend the presentation to confirm that the induction step, which breaks down when using the old bound, now works.

\begin{proof}
The proof is by induction on $m$ and the statement is trivial for
$m=0$.

If $\psi$ is not simplexwise injective, there exists a simplex whose
vertices do not map to pairwise distinct points.  In particular we can
choose a simplex $\sigma \subseteq Y$ of maximal dimension $k>0$ such
that for every vertex $x$ of $\sigma$ there is another vertex~$y$ of
$\sigma$ with $\psi(x) = \psi(y)$.  By assumption, $\sigma$ is not
contained in $\partial Y$.  Maximality of the dimension of $\sigma$
implies that the restriction of~$\psi$ to the $(m-k-1)$-sphere
$\lk_Y(\sigma)$ is simplexwise injective.  It also implies that
$\psi(\lk_Y(\sigma)) \subseteq \lk_X(\psi(\sigma))$.  Note further
that $\psi(\sigma)$ has dimension at most $(k-1)/2 \le k-1$.  Therefore its
link in $X$ is~$(m-k-1)$-connected by assumption.  Hence there is an
$(m-k)$-disk $B$ with $\partial B = \lk_Y(\sigma)$ and a map~$\varphi
\colon B \to \lk_X(\psi(\sigma))$ such that $\varphi|_{\partial B}$
coincides with $\psi |_{\lk_Y(\sigma)}$.  Inductively applying the
lemma, we may assume that $\varphi$ is simplexwise injective.

This inductive step is indeed possible because if $\tau$ is a $d$-simplex in $\lk_X(\psi(\sigma))$ then $\psi(\sigma) \vee \tau$ is a simplex of dimension at most $(k-1)/2+d+1 \le k+d$ such that $\lk_{\lk_X(\psi(\sigma))}(\tau) = \lk_X(\psi(\sigma) \vee \tau)$. By assumption that complex is $((m-k) - d - 2)$-connected.

We now replace $Y$ by $Y'$, the space obtained by replacing the closed star of
$\sigma$ by~$B * \partial \sigma$.  The map $\psi' \colon Y' \to X$ is the map
that coincides with $\psi$ outside the open star of $\sigma$, coincides with
$\varphi$ on $B$ and is affine on simplices. It is clearly homotopic to~$\psi$,
since the image of~$B$ under~$\varphi$ is contained in $\lk_X(\psi(\sigma))$. 
Since the restriction of $\psi'$ to $B$ is simplexwise injective, the
restriction to any $k$-simplex of $B*\partial \sigma$ is injective. Since $Y$ is compact, by repeating this procedure finitely many times we eventually obtain a map that is simplexwise injective.
\end{proof}

This change is inconsequential for the rest of the article, because in both applications of Lemma~3.9 the new bound is still met. Indeed in the proof of Theorem~3.10 the estimate reads
\[
\eta(n-3k-3)-1 = \Big\lfloor \frac{n-3k-4}{4} \Big\rfloor -1 \ge \eta(n) - \Big\lfloor\frac{3}{4}k\Big\rfloor - 2 > m-k-2 \text{,}
\]
because $m < \eta(n)-1$, and in the alternate proof of Theorem~3.8 it reads
\[
\nu(n-2k-2)-1 = \Big\lfloor\frac{n-2k-1}{3}\Big\rfloor-1 \ge \nu(n)-\Big\lfloor\frac{2}{3}k\Big\rfloor-2 > m-k-2 \text{\,}
\]
because $m < \nu(n)-1$.
(Recall that $\nu(n) = \bigl\lfloor\frac{n+1}{3}\bigr\rfloor-1$ and $\eta(n) =\bigl\lfloor\frac{n-1}{4}\bigr\rfloor$.)

\medskip

The following example shows that the original formulation of the lemma is incorrect, and not just its proof.

\begin{example}
Let $Y$ consist of two triangles with vertices $a,b,c$ and $b,c,d$, respectively. Then $Y$ is a $2$-dimensional combinatorial manifold whose boundary is the circle $[a,b] \cup [b,d] \cup [d,c] \cup [c,a]$. Let $X$ be a single edge $[v,w]$. The link of each vertex of $X$ is a single point, hence contractible, while $\lk [v,w] = \emptyset$ is $(2 - 2\cdot 1 -2)$-connected (which is an empty condition) but not $(2 - 1 - 2)$-connected (which would mean non-empty). Consider the simplicial map that takes $a$ and $d$ to $v$ and $b$ and $c$ to $w$. Its restriction to the boundary is simplexwise injective. However, invariance of domain forbids that an interior point of $Y$ could have a neighborhood that is mapped injectively to $X$.

The proof of the lemma would introduce an additional vertex $x$ in the interior of $[v,w]$ and then map $[b,c]$ to $[x,w]$. However, then the induction fails since the non-empty link of $[b,c]$ cannot be mapped to the empty link of $[x,w]$.
\end{example}


\vspace{2\bigskipamount}

\clearpage

\thispagestyle{empty}
\markleft{\uppercase{M.~C.~B.~Zaremsky}}
\markright{\uppercase{Appendix: Higher generation for pure braid groups}}

\setcounter{section}{0}

\renewcommand{\thesection}{A.\arabic{section}}
\renewcommand{\theHsection}{A.\arabic{section}}

\setlength{\baselineskip}{\stdbskip}

\begin{center}
    \textbf{\uppercase{Appendix: Higher generation for pure braid groups}}
    
    \bigskip
    
    \uppercase{\small by Matthew C.~B.~Zaremsky}
    
    \bigskip
\end{center}

\setlength{\baselineskip}{\bskipmaintext}

In this appendix we use techniques and results from the main body of
the paper to derive higher generation properties for families of
subgroups of pure braid groups.  The notion of a family of subgroups
of a group being \emph{highly generating} was introduced by Abels and
Holz \cite{abels93}.  It is a very natural condition, with many strong
consequences, but to date few examples have been explicitly
constructed of highly generating subgroups for ``interesting'' groups.
One prominent existing example, given by Abels and Holz, is standard
parabolic subgroups of Coxeter groups, or standard parabolic subgroups
of groups with a~$BN$-pair.  The relevant geometry is given by Coxeter
complexes and buildings.  Higher generation is also used in~\cite{meier98} as a tool to calculate the Bieri--Neumann--Strebel--Renz
invariants of right-angled Artin groups.

As an addition to the collection of interesting examples, we produce
two classes of families of subgroups of the~$n$-string pure braid
group~$PB_n$ that we show to be highly generating.  In the first case
the geometry is given by complexes of arcs on a surface, related to
the complexes $\matcharc(L_n)$ from
Definition~\ref{def:match_cpx_surface}.  In the second case the
geometry is given by complexes of ``dangling flat braiges'', related
to the complexes $\elpbraigecpx_n$ analyzed in
Section~\ref{sec:desc_link_conn}.

\medskip

In Section~\ref{sec:higher_gen} we recall some definitions and results
from \cite{abels93}, and establish a criterion for detecting coset
complexes in Proposition~\ref{prop:fund_dom_general}.  In
Section~\ref{sec:arcs} we define the \emph{restricted arc complex} on
a surface, and in Section~\ref{sec:braiges} we define the complex of
\emph{dangling flat pure braiges}.  The relevant families of subgroups
of~$PB_n$ are defined in the paragraphs before
Lemma~\ref{lem:nerve_to_arcs} and
Corollary~\ref{cor:nerve_to_braiges}, and in
Definition~\ref{def:restrictive_fams}.  Finally in
Section~\ref{sec:conn_cpxes} we calculate the connectivity of these
complexes and deduce that the families of subgroups are highly
generating.  See Propositions~\ref{prop:restrictive_arc_conn}
and~\ref{prop:restrictive_braige_conn} for the exact bounds.


\section{Higher generation}
\label{sec:higher_gen}

Higher generation is defined using nerves of coverings of groups by
cosets.  The relevant definitions are as follows.

\begin{definition}[Nerve]\label{def:nerve}
Let $X$ be a set and $\cover$ a collection of subsets covering~$X$.
The \emph{nerve} of the cover~$\cover$, denoted~$\nerve(\cover)$, is a
simplicial complex with vertex set~$\cover$, such that pairwise
distinct vertices $U_0,\dots,U_k$ span a~$k$-simplex if and only if~$U_0\cap\cdots\cap U_k\neq \emptyset$.
\end{definition}

The type of nerve we are interested in is the following \emph{coset
complex}.

\begin{definition}[Coset complex and higher generation]\label{def:high_gen}
    Let $G$ be a group and~$\family$ a family of subgroups.  Let
    $\cosets \defeq \coprod\limits_{H\in\family}G/H$ be the covering of~$G$
    by cosets of subgroups in~$\family$.  We call~$\nerve(\cosets)$
    the \emph{coset complex} of~$G$ with respect to~$\family$, and
    denote it $\CC(G,\family)$.  We say that $\family$
    $n$-\emph{generates}~$G$ if~$\CC(G,\family)$ is~$(n-1)$-connected,
    and $\infty$-\emph{generates}~$G$ if~$\CC(G,\family)$ is
    contractible.
\end{definition}

The following theorem indicates some ways higher generation can be
used.  The first part says that~$1$-generation equals generation, and
the second part says that a~$2$-generating family yields a
decomposition of~$G$ as an amalgamated product.

\begin{theorem}\label{thrm:alg_to_geom}\cite[Theorem~2.4]{abels93}
    Let $\family=\{H_\alpha \mid \alpha\in \Lambda\}$ be a family of
    subgroups of~$G$.
    \begin{enumerate}
	\item $\family$ is~$1$-generating if and only if
	$\bigcup H_\alpha$ generates~$G$.
  
	\item $\family$ is~$2$-generating if and only if the natural
	map $\coprod\limits_\cap H_\alpha \to G$ is an
	isomorphism.
 \end{enumerate}
\end{theorem}

Here by $\coprod\limits_\cap H_\alpha$ we mean the amalgamated product
of the $H_\alpha$ over their intersections.  We remark that another
equivalent condition in part~$(1)$ is that the map~$\coprod\limits_\cap H_\alpha \to G$ be surjective.

An important observation about coset complexes is that the action of
the group on the complex has a very nice fundamental domain.

\begin{observation}[Fundamental domain]\label{obs:fund_dom}
    With the above notation, assume~$\family$ is finite.  Since
    $\bigcap\limits_{H\in\family}H\neq\emptyset$, we see that
    $\family$ itself is the vertex set of a maximal simplex
    in~$\CC(G,\family)$.  This maximal simplex, which we call~$C$, is
    a fundamental domain for the action of~$G$ on~$\CC(G,\family)$ by
    left multiplication.
\end{observation}

\begin{proof}
 For any simplex $\sigma$ in $\CC(G,\family)$, there exist $H_0,\dots,H_k\in\family$ and $g\in G$ such that the vertices of $\sigma$ are the cosets $gH_i$ for $0\le i\le k$. Then $g^{-1} \sigma$ is a face of~$C$. This shows that every~$G$-orbit intersects $C$, and indeed intersects~$C$ uniquely since if $gH_i=H_j$ then $g\in H_i=H_j$.
\end{proof}

A sort of converse of this observation is the following proposition, which allows us to detect highly generating families of subgroups as stabilizers of ``nice'' actions.

\begin{proposition}[Detecting coset complexes]\label{prop:fund_dom_general}
 Let $G$ be a group acting by simplicial automorphisms on a simplicial complex~$X$, with a single maximal simplex~$C$ as fundamental domain. Let
 $$\family\defeq \{\Stab_G(v)\mid v\text{ is a vertex of }C\}.$$
 Then $\CC(G,\family)$ is isomorphic to $X$ as a simplicial~$G$-complex.
\end{proposition}

\begin{proof}
 Define a map $\phi\colon \CC(G,\family) \to X$ by sending the coset~$g\Stab_G(v)$ to the vertex~$gv$ of~$X$. This is a $G$-invariant map between the~$0$-skeleta, and it induces a simplicial map since the vertices of a simplex in $\CC(G,\family)$ can be represented as cosets with a common left representative. Since~$C$ is a fundamental domain,~$\phi$ is bijective.
\end{proof}

A good first example is when~$X$ is a tree, on which a group~$G$ acts edge transitively and without inversion. Then Theorem~\ref{thrm:alg_to_geom} and Proposition~\ref{prop:fund_dom_general} imply that~$G$ decomposes as an amalgamated product. Namely, if~$e$ is a fundamental domain with endpoints~$v$ and~$w$, then $G=G_v\ast_{G_e}G_w$ (this is standard Bass--Serre theory). Indeed, the vertex stabilizers are not just~$2$-generating, but~$\infty$-generating.

This example is generalized by looking at groups acting on buildings.

\begin{example}[Buildings]\label{ex:bldgs}
 Let $G$ be a group acting chamber transitively on a building~$\Delta$, by type preserving automorphisms. See~\cite{abramenko08} for the relevant background. Let~$C$ be the fundamental chamber, and let~$\family\defeq \{\Stab_G(v)\mid v$ is a vertex of $C\}$. Then $\CC(G,\family)\cong\Delta$, and so $\family$ is highly generating for~$G$. More precisely, if~$\Delta$ is spherical of dimension~$n$ then $\family$ is $n$-generating, and if~$\Delta$ is not spherical then $\family$ is $\infty$-generating. If the action is not just chamber transitive, but is even \emph{Weyl transitive}, as in~\cite[Chapter~6]{abramenko08}, then the stabilizers $\Stab_G(v)$ are precisely the maximal \emph{standard parabolic subgroups}. An even stronger condition is that the action is \emph{strongly transitive}, in which case~$G$ has a~$BN$-pair, and we recover the situation in \cite[Section~3.2]{abels93}.
\end{example}

We also have examples from the world of Artin groups.

\begin{example}[Deligne complexes]\label{ex:deligne}
 Background for this example can be found in~\cite{charney95}. Let~$(A,S)$ be an Artin system with associated Coxeter system $(W,S)$. For~$T\subseteq S$ let $A_T$ (respectively $W_T$) be the subgroup generated by $T$. Let~$\widehat{\family}\defeq \{A_T \mid T\subseteq S\}$ and $\family\defeq \{A_T \mid T\subseteq S$ with $|W_T|<\infty\}$. The coset complexes~$\CC(A,\widehat{\family})$ and~$\CC(A,\family)$ are, up to homotopy equivalence, the \emph{Deligne complex} and \emph{modified Deligne complex} of $A$. The connectivity of these complexes, and hence the higher generation properties of these families of subgroups, is tied to the~$K(\pi,1)$ Conjecture described in~\cite{charney95}. Namely, $\family$ is conjecturally $\infty$-generating; see~\cite[Conjecture~2]{charney95}. This is known to hold for many Artin groups, including for braid groups.
\end{example}


\section{Some variations on arc complexes and braige complexes}\label{sec:arcs_and_braiges}

In this section we define and analyze some complexes on which the braid group and pure braid group act.  In the first subsection we look at the \emph{restricted arc complex} on a surface, and in the second subsection we look at the complex of \emph{flat dangling (pure) braiges}.  The restricted arc complex here will provide a coset complex for~$PB_n$ using arc stabilizers as subgroups. The flat dangling pure braige complex will provide a coset complex for~$PB_n$ using subgroups obtained via the ``strand cloning maps''. These subgroups are smaller than the arc stabilizers, and more visualizable when using strand pictures for braids. We will save the connectivity calculations for Section~\ref{sec:conn_cpxes}, after which we will conclude that these families of subgroups are highly generating.

\subsection{Arc complexes}\label{sec:arcs}

We maintain the definitions and notation from Section~\ref{sec:arc_cpx}. Consider~$\hatcharc(\Gamma)$ for $\Gamma$ a subgraph of $K_n$ with the same vertex set.

\medskip
\textbf{Terminological convention:} Throughout this appendix, a \emph{subgraph} $\Gamma''$ of a graph $\Gamma'$ always has the same vertex set as $\Gamma'$.
\medskip

Given an arc system $\sigma=\{\alpha_0,\dots,\alpha_k\}$ in $\hatcharc(\Gamma)$, denote by $\Gamma_\sigma$ the following subgraph of $\Gamma$. Every vertex of~$\Gamma$ is a vertex of~$\Gamma_\sigma$, and an edge~$e$ of~$\Gamma$ is in $\Gamma_\sigma$ if and only if the endpoints of~$e$ are the endpoints of some~$\alpha_i$. Call $\Gamma_\sigma$ \emph{faithful} if it has precisely~$(k+1)$ edges.  Since we only consider simplicial graphs, i.e., there are no loops or multiple edges, this condition is equivalent to saying that no distinct~$\alpha_i$,~$\alpha_j$ share both endpoints (they may share one).

The complex we are presently interested in is a complex $\restrarc(\Gamma)$, which we will call the \emph{restricted arc complex}.

\begin{definition}[Restricted arc complex]\label{def:restr_arc_cpx}
 The \emph{restricted arc complex} $\restrarc(\Gamma)$ on~$(S,P)$ corresponding to $\Gamma$ is the subcomplex of $\hatcharc(\Gamma)$ consisting of arc systems~$\sigma$ for which $\Gamma_\sigma$ is faithful. We may also write $\restrarc(S,P,\Gamma)$.
\end{definition}

We could equivalently require that the subspace of~$S$ given by the union of the arcs is a simplicial graph, i.e., has no multiple edges. In this way we can view~$\restrarc(\Gamma)$ as the complex of embeddings of subgraphs of~$\Gamma$ into~$S$ that send vertices in a prescribed way to the points of~$P$.

\medskip

\textbf{Notational convention:} Throughout this appendix, $L_n$ denotes not the linear graph with~$n$ edges, but rather the linear graph with~$n$ vertices, and hence~$n-1$ edges.
\medskip

The $\Gamma=L_n$ case is especially nice, since all of~$L_n$ can be embedded into any connected surface. In fact, every simplex of~$\restrarc(L_n)$ is a face of a maximal simplex of dimension~$n-2$.  See Figure~\ref{fig:arcs} for some examples of arc systems.

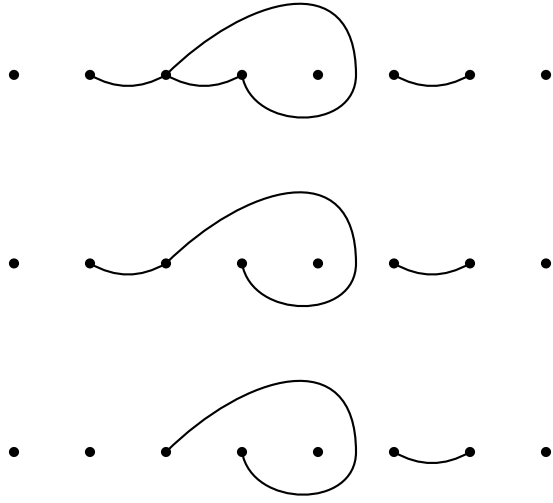
\begin{figure}[t]
\centering
\begin{tikzpicture}[line width=0.8pt]
  
  \filldraw
   (0,0) circle (1.5pt)   (1,0) circle (1.5pt)   (2,0) circle (1.5pt)   (3,0) circle (1.5pt)   (4,0) circle (1.5pt)   (5,0) circle (1.5pt)   (6,0) circle (1.5pt)   (7,0) circle (1.5pt);

  \draw
   (1,0) to [out=-30, in=-150, looseness=1] (2,0)
   (2,0) to [out=45, in=90, looseness=1.5] (4.5,0) to [out=-90, in=-80, looseness=1.3] (3,0)
   (2,0) to [out=-30, in=-150, looseness=1] (3,0)
   (5,0) to [out=-30, in=-150, looseness=1] (6,0);
    
 \begin{scope}[yshift=-2.5cm]
  \filldraw
   (0,0) circle (1.5pt)   (1,0) circle (1.5pt)   (2,0) circle (1.5pt)   (3,0) circle (1.5pt)   (4,0) circle (1.5pt)   (5,0) circle (1.5pt)   (6,0) circle (1.5pt)   (7,0) circle (1.5pt);

  \draw
   (1,0) to [out=-30, in=-150, looseness=1] (2,0)
   (2,0) to [out=45, in=90, looseness=1.5] (4.5,0) to [out=-90, in=-80, looseness=1.3] (3,0)
   (5,0) to [out=-30, in=-150, looseness=1] (6,0);
 \end{scope}

 \begin{scope}[yshift=-5cm]
  \filldraw
   (0,0) circle (1.5pt)   (1,0) circle (1.5pt)   (2,0) circle (1.5pt)   (3,0) circle (1.5pt)   (4,0) circle (1.5pt)   (5,0) circle (1.5pt)   (6,0) circle (1.5pt)   (7,0) circle (1.5pt);

  \draw
   (2,0) to [out=45, in=90, looseness=1.5] (4.5,0) to [out=-90, in=-80, looseness=1.3] (3,0)
   (5,0) to [out=-30, in=-150, looseness=1] (6,0);
 \end{scope}

\end{tikzpicture}

\caption{From top to bottom, an arc system in $\hatcharc(L_8) \setminus \restrarc(L_8)$, one in $\restrarc(L_8) \setminus \matcharc(L_8)$ and one in $\matcharc(L_8)$.}
\label{fig:arcs}
\end{figure}

\begin{remark}\label{rmk:graph_embedding}
 Embedding graphs into surfaces is an interesting enterprise in its own right, so the complex~$\restrarc(\Gamma)$ may be of further general interest. For instance, the dimension of~$\restrarc(S,P,\Gamma)$ is one less than the number of edges in a maximal subgraph of~$\Gamma$ embeddable into~$(S,P)$.
\end{remark}

\medskip

Recall that~$B_n$ acts on $\hatcharc(K_n)$, and this action stabilizes~$\matcharc(K_n)$ and~$\restrarc(K_n)$. For general~$\Gamma$,~$B_n$ will not necessarily stabilize~$\hatcharc(\Gamma)$, since general braids may not stabilize~$P$ pointwise. However, pure braids do stabilize~$P$ pointwise, and so~$PB_n$ stabilizes $\hatcharc(\Gamma)$, $\matcharc(\Gamma)$ and $\restrarc(\Gamma)$ for any~$\Gamma$.

Denote by $[m]$ the set $\{1,\dots,m\}$ for $m\in\N$. Let~$S$ be the unit disk, and fix an embedding $L_n \hookrightarrow S$ of the linear graph with $n$ vertices into~$S$. Let~$P$ be the image of the vertex set, so~$P$ is a set of~$n$ points in~$S$, labeled~$1$ through~$n$. Under this embedding, the edges of~$L_n$ yield a maximal simplex of~$\restrarc(L_n)$, which we will denote~$C$. For each~$\emptyset\neq J \subseteq [n-1]$ define $\sigma_J$ to be the face of~$C$ consisting only of those arcs with endpoints~$j,j+1$ for~$j\in J$. In particular, $\sigma_J$ is a $(|J|-1)$-simplex in $\restrarc(L_n)$.

For each $\emptyset\neq J \subseteq [n-1]$ define
$$PB_n^J \defeq \Stab_{PB_n}(\sigma_J)$$
and set $\arcfam_n \defeq \{PB_n^J \mid \emptyset\neq J \subseteq [n-1]$ with $|J|=1\}$.

\begin{lemma}\label{lem:nerve_to_arcs}
 The coset complex $\CC(PB_n,\arcfam_n)$ and the restricted arc complex~$\restrarc(L_n)$ are isomorphic as simplicial $PB_n$-complexes.
\end{lemma}

\begin{proof}
 It suffices by Proposition~\ref{prop:fund_dom_general} to show that~$C$ is a fundamental domain for the action of~$PB_n$ on~$\restrarc(L_n)$.  A maximal simplex of~$\restrarc(L_n)$ is an embedding of~$L_n$ into~$S$, such that the vertex labeled~$i$ maps to the point in~$P$ labeled~$i$, for each~$1\le i\le n$. Any such simplex is in the~$PB_n$-orbit of~$C$. Moreover, if $p\sigma_J=\sigma_K$ for $p\in PB_n$ and $\sigma_J,\sigma_K$ are faces of~$C$, then since~$p$ is pure we know that~$J=K$. We conclude that~$C$ is a fundamental domain.
\end{proof}

In Section~\ref{sec:conn_cpxes} we will calculate the connectivity of~$\restrarc(L_n)$, and deduce that~$\arcfam_n$ is highly generating for~$PB_n$. Before doing that, we describe another complex with a nice~$PB_n$ action.

\pagebreak[3]

\subsection{Flat braige complexes}\label{sec:braiges}

\begin{definition}[Flat braiges]\label{def:flat_braige}
 A \emph{flat braige} on~$n$ strands is a pair $(b,\Gamma)$, consisting of a braid $b\in B_n$ and a subgraph~$\Gamma$ of~$L_n$. If the edges of~$\Gamma$ are disjoint, we call~$(b,\Gamma)$ \emph{elementary}. If the braid is pure, then the braige is a \emph{(flat) pure braige}. See Figure~\ref{fig:braiges} for some examples.
\end{definition}

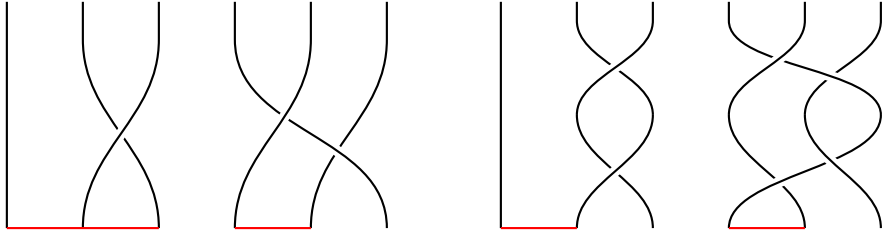
\begin{figure}[t]
\centering
\begin{tikzpicture}[line width=0.8pt]
  \draw
   (0,0) -- (0,-3)   (1,0) -- (1,-0.5) to [out=-90, in=90, looseness=1] (2,-3);
  \draw[white, line width=4pt]
   (2,0) -- (2,-0.5) to [out=-90, in=90, looseness=1] (1,-3);
  \draw
   (2,0) -- (2,-0.5) to [out=-90, in=90, looseness=1] (1,-3)
   (5,0) -- (5,-0.5) to [out=-90, in=90, looseness=1] (4,-3);
  \draw[white, line width=4pt]
   (3,0) -- (3,-0.5) to [out=-90, in=90, looseness=1] (5,-3);
  \draw
   (3,0) -- (3,-0.5) to [out=-90, in=90, looseness=1] (5,-3);
  \draw[white, line width=4pt]
   (4,0) -- (4,-0.5) to [out=-90, in=90, looseness=1] (3,-3);
  \draw
   (4,0) -- (4,-0.5) to [out=-90, in=90, looseness=1] (3,-3);
  \draw[red]
   (0,-3) -- (2,-3)   (3,-3) -- (4,-3);
 \begin{scope}[xshift=6.5cm, yscale=0.5]
  \draw
   (0,0) -- (0,-6)   (1,0) -- (1,-0.5) to [out=-90, in=90, looseness=1] (2,-3)   (1,-3) to [out=-90, in=90, looseness=1] (2,-6);
  \draw[white, line width=4pt]
   (2,0) -- (2,-0.5) to [out=-90, in=90, looseness=1] (1,-3)   (2,-3) to [out=-90, in=90, looseness=1] (1,-6);
  \draw
   (2,0) -- (2,-0.5) to [out=-90, in=90, looseness=1] (1,-3)   (2,-3) to [out=-90, in=90, looseness=1] (1,-6)
   (5,0) -- (5,-0.5) to [out=-90, in=90, looseness=1] (4,-3)   (3,-3) to [out=-90, in=90, looseness=1] (4,-6);
  \draw[white, line width=4pt]
   (3,0) -- (3,-0.5) to [out=-90, in=90, looseness=1] (5,-3)   (5,-3) to [out=-90, in=90, looseness=1] (3,-6);
  \draw
   (3,0) -- (3,-0.5) to [out=-90, in=90, looseness=1] (5,-3)   (5,-3) to [out=-90, in=90, looseness=1] (3,-6);
  \draw[white, line width=4pt]
   (4,0) -- (4,-0.5) to [out=-90, in=90, looseness=1] (3,-3)   (4,-3) to [out=-90, in=90, looseness=1] (5,-6);
  \draw
   (4,0) -- (4,-0.5) to [out=-90, in=90, looseness=1] (3,-3)   (4,-3) to [out=-90, in=90, looseness=1] (5,-6);
  \draw[red]
   (0,-6) -- (1,-6)   (3,-6) -- (4,-6);
 \end{scope}

\end{tikzpicture}

\caption{A flat braige on~$6$ strands and an elementary pure braige on~$6$ strands.}
\label{fig:braiges}
\end{figure}

Note the fundamental difference between flat braiges here and ``braiges'', as in Section~\ref{sec:general_diagrams}. For flat braiges, a ``merge'' amounts to just choosing some pairs of adjacent strands that should be stuck together at the bottom with edges. With braiges however, the merging is more subtle; strands merge two at a time, not in a square shape but in more of a triangle, and a new strand continues down out of the merge. This new strand may merge further with other strands, but one must keep track of the order of merging. However, the notions of elementary braiges are the same here and as before, since it does not matter in which order the merges occur. The spraige in Figure~\ref{fig:spraige_merging} is a good example of how, before, we kept track of the order of merging, but with flat braiges as in Figure~\ref{fig:braiges}, we do not, and so the bottom of the picture is flattened out.

Let $\braiges_n(L_n)$ be the set of all flat braiges on~$n$ strands.  There is a left action of $B_n$ on~$\braiges_n(L_n)$, via $b(c,\Gamma) \defeq (bc,\Gamma)$. We can think of $\braiges_n(L_n)$ as a simplicial complex, where $(b,\Gamma)$ is a face of~$(b',\Gamma')$ if $b=b'$ and $\Gamma'$ is a subgraph of~$\Gamma$. Restricting to pure braids, we get the set~$\pbraiges_n(L_n)$ of flat pure braiges, with an action of~$PB_n$. A nice feature of this is that~$(\id,L_n)$ is a fundamental domain for the action of~$B_n$ on~$\braiges_n(L_n)$, or~$PB_n$ on~$\pbraiges_n(L_n)$. However, it is easy to see that~$\braiges_n(L_n)$ and~$\pbraiges_n(L_n)$ stand little chance of being connected, since we can only ``move'' by changing the merges, and not the braid. To get a highly connected complex, we will consider an equivalence relation on these complexes via the notion of dangling, as in Section~\ref{sec:dangling}. First we need to define what it means for a strand in a braid to be a \emph{clone}.

\begin{definition}[Clones]\label{def:clones}
 Let $b\in B_n$. Number the strands of~$b$ from left to right at their tops by~$1,\dots,n$. Let~$\rho_b$ be the permutation induced by~$b$ under~$B_n\to S_n$. Think of~$b$ as living in~$3$-space~$\mathbb{R}^3$, with the top of the~$i^{\text{th}}$ strand at the point~$(i,1,0)$ and the bottom at~$(\rho_b(i),0,0)$, for each~$i\in[n]$. In particular all the tops and bottoms of the strands are in the~$xy$-plane. Note that for any given strand,~$b$ has a representation wherein that strand is entirely contained in the~$xy$-plane. Now suppose that for some~$i\in[n-1]$,~$b$ can be represented in such a way that the~$i^{\text{th}}$ and~$(i+1)^{\text{st}}$ strands are \emph{simultaneously} in the~$xy$-plane, and moreover, no strands of the braid other than those two intersect the closed region of the~$xy$-plane bounded by the two strands and the line segments from~$(i,1,0)$ to~$(i+1,1,0)$ and from~$(\rho_b(i),0,0)$ to~$(\rho_b(i+1),0,0)$.
 In this case we will refer to the~$(i+1)^{\text{st}}$ strand as a \emph{clone}, specifically a \emph{clone of the $i^{th}$ strand}. Note that necessarily~$\rho_b(i+1)=\rho_b(i)+1$.
\end{definition}

Our convention is to always consider the strand on the right to be the clone of the strand on the left, as opposed to the other way around. See Figure~\ref{fig:clone} for an example.

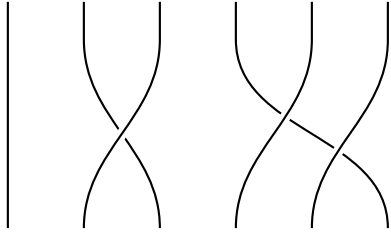
\begin{figure}[t]
\centering
\begin{tikzpicture}[line width=0.8pt]
  \draw
   (0,0) -- (0,-3)   (1,0) -- (1,-0.5) to [out=-90, in=90, looseness=1] (2,-3);
  \draw[white, line width=4pt]
   (2,0) -- (2,-0.5) to [out=-90, in=90, looseness=1] (1,-3);
  \draw
   (2,0) -- (2,-0.5) to [out=-90, in=90, looseness=1] (1,-3)
   (3,0) -- (3,-0.5) to [out=-90, in=90, looseness=1] (5,-3);
  \draw[white, line width=4pt]
   (5,0) -- (5,-0.5) to [out=-90, in=90, looseness=1] (4,-3);
  \draw
   (5,0) -- (5,-0.5) to [out=-90, in=90, looseness=1] (4,-3);
  \draw[white, line width=4pt]
   (4,0) -- (4,-0.5) to [out=-90, in=90, looseness=1] (3,-3);
  \draw
   (4,0) -- (4,-0.5) to [out=-90, in=90, looseness=1] (3,-3);
\end{tikzpicture}

\caption{The sixth strand is a clone of the fifth.}
\label{fig:clone}
\end{figure}

For each $i\in[n-1]$ there is a \emph{cloning map} $\kappa_i \colon B_{n-1} \to B_n$ given by cloning the~$i^{\text{th}}$ strand. This is not a homomorphism, but becomes one when restricted to~$\kappa_i \colon PB_{n-1} \to PB_n$. For $I=\{i_1,\dots,i_r\}\subseteq[n-r]$, with $i_1<\cdots<i_r$, define the cloning map $\kappa_I \defeq \kappa_{i_1} \circ \cdots \circ \kappa_{i_r} \colon B_{n-r} \to B_n$. The restriction $\kappa_I \colon PB_{n-r} \to PB_n$ is again a homomorphism. Now for $J=\{j_1,\dots,j_r\}\subseteq[n-1]$, with $j_1<\cdots<j_r$, let~$I_J\subseteq[n-r]$ be the set $\{j_i-(i-1) \mid 1\le i\le r\}$. The point is that a braid~$b\in B_n$ is in the image of~$\kappa_{I_J}$ if and only if for each~$j\in J$, the~$(j+1)^{\text{st}}$ strand is a clone of the~$j^{\text{th}}$ strand. Denote the subset of such braids by~$B_n^{(J)}$, and the sub\emph{group} of such pure braids by~$PB_n^{(J)}$. (The parentheses distinguish~$PB_n^{(J)}$ from the arc system stabilizer~$PB_n^J$ from the previous section.)

We can now define the equivalence relation between flat braiges, given by dangling. This is closely related to the notion of dangling in Section~\ref{sec:dangling}.

\begin{definition}[Dangling flat braiges]\label{def:dangling}
 Let $(b,\Gamma)$ be a flat braige on~$n$ strands, and number the vertices of~$\Gamma$ by~$1,\dots,n$ from left to right. Let~$J_\Gamma \subseteq [n-1]$ be the set of left endpoints of edges of~$\Gamma$. Now consider any braid~$c$ from the set~$B_n^{(J_\Gamma)}$. For each~$j\in J_\Gamma$, we know that~$\rho_c(j+1)=\rho_c(j)+1$, so there is a subgraph of~$L_n$ whose edges are precisely those connecting~$\rho_c(j)$ and $\rho_c(j+1)$ for~$j\in J_\Gamma$. Call this graph~$\Gamma^c$. The point is that, if we draw~$c$ below the braige, and ``pull'' the merges through~$c$, we get the flat braige~$(bc,\Gamma^c)$. Now declare that~$(b,\Gamma)$ is equivalent to~$(bc,\Gamma^c)$ for each~$c\in B_n^{(J_\Gamma)}$. One checks that this is an equivalence relation, called equivalence \emph{under dangling}. Denote the equivalence class of~$(b,\Gamma)$ by~$[(b,\Gamma)]$, and call it a \emph{dangling flat braige}. The idea is that the top of a braige is static, but the strands at the bottom are free to ``dangle'', modulo the 
restriction that the merges remain rigid (and oriented) during the dangling. We analogously get the notion of a \emph{dangling flat pure braige}, where we only consider~$c$ as above coming from~$PB_n^{(J_\Gamma)}$, so in particular~$\Gamma^c$ always equals~$\Gamma$ in the pure case. An example of dangling can be seen in Figure~\ref{fig:flat_dangling}, and refer back to Figure~\ref{fig:dangle} for comparison with the non-flat case.
\end{definition}

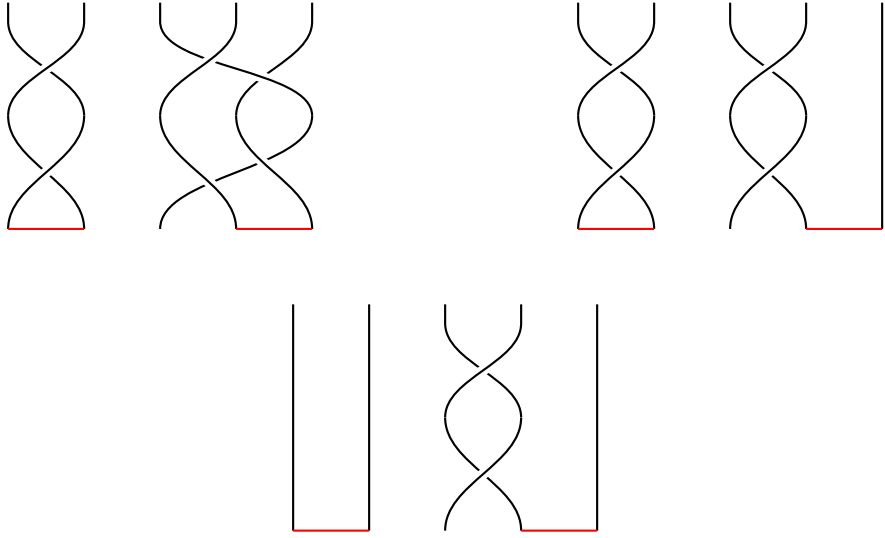
\begin{figure}[t]
\centering
\vskip0.25cm
\begin{tikzpicture}[line width=0.8pt, yscale=0.5]
  \draw
   (1,0) -- (1,-0.5) to [out=-90, in=90, looseness=1] (2,-3)   (1,-3) to [out=-90, in=90, looseness=1] (2,-6);
  \draw[white, line width=4pt]
   (2,0) -- (2,-0.5) to [out=-90, in=90, looseness=1] (1,-3)   (2,-3) to [out=-90, in=90, looseness=1] (1,-6);
  \draw
   (2,0) -- (2,-0.5) to [out=-90, in=90, looseness=1] (1,-3)   (2,-3) to [out=-90, in=90, looseness=1] (1,-6)
   (5,0) -- (5,-0.5) to [out=-90, in=90, looseness=1] (4,-3)   (5,-3) to [out=-90, in=90, looseness=1] (3,-6);
  \draw[white, line width=4pt]
   (3,0) -- (3,-0.5) to [out=-90, in=90, looseness=1] (5,-3)   (3,-3) to [out=-90, in=90, looseness=1] (4,-6);
  \draw
   (3,0) -- (3,-0.5) to [out=-90, in=90, looseness=1] (5,-3)   (3,-3) to [out=-90, in=90, looseness=1] (4,-6);
  \draw[white, line width=4pt]
   (4,0) -- (4,-0.5) to [out=-90, in=90, looseness=1] (3,-3)   (4,-3) to [out=-90, in=90, looseness=1] (5,-6);
  \draw
   (4,0) -- (4,-0.5) to [out=-90, in=90, looseness=1] (3,-3)   (4,-3) to [out=-90, in=90, looseness=1] (5,-6);
  \draw[red]
   (1,-6) -- (2,-6)   (4,-6) -- (5,-6);

 \begin{scope}[xshift=7.5cm]
  \draw
   (1,0) -- (1,-0.5) to [out=-90, in=90, looseness=1] (2,-3)   (1,-3) to [out=-90, in=90, looseness=1] (2,-6);
  \draw[white, line width=4pt]
   (2,0) -- (2,-0.5) to [out=-90, in=90, looseness=1] (1,-3)   (2,-3) to [out=-90, in=90, looseness=1] (1,-6);
  \draw
   (2,0) -- (2,-0.5) to [out=-90, in=90, looseness=1] (1,-3)   (2,-3) to [out=-90, in=90, looseness=1] (1,-6)
   (3,0) -- (3,-0.5) to [out=-90, in=90, looseness=1] (4,-3)   (3,-3) to [out=-90, in=90, looseness=1] (4,-6);
  \draw[white, line width=4pt]
   (4,0) -- (4,-0.5) to [out=-90, in=90, looseness=1] (3,-3)   (4,-3) to [out=-90, in=90, looseness=1] (3,-6);
  \draw
   (4,0) -- (4,-0.5) to [out=-90, in=90, looseness=1] (3,-3)   (4,-3) to [out=-90, in=90, looseness=1] (3,-6)
   (5,0) -- (5,-6);
  \draw[red]
   (1,-6) -- (2,-6)   (4,-6) -- (5,-6);
 \end{scope}

 \begin{scope}[xshift=3.75cm,yshift=-8cm]
  \draw
   (1,0) -- (1,-6)   (2,0) -- (2,-6)
   (3,0) -- (3,-0.5) to [out=-90, in=90, looseness=1] (4,-3)   (3,-3) to [out=-90, in=90, looseness=1] (4,-6);
  \draw[white, line width=4pt]
   (4,0) -- (4,-0.5) to [out=-90, in=90, looseness=1] (3,-3)   (4,-3) to [out=-90, in=90, looseness=1] (3,-6);
  \draw
   (4,0) -- (4,-0.5) to [out=-90, in=90, looseness=1] (3,-3)   (4,-3) to [out=-90, in=90, looseness=1] (3,-6)
   (5,0) -- (5,-6);
  \draw[red]
   (1,-6) -- (2,-6)   (4,-6) -- (5,-6);
 \end{scope}
\end{tikzpicture}

\caption{The two (elementary) braiges on the top are equivalent under pure dangling, but are \emph{not} equivalent to the third.}
\label{fig:flat_dangling}
\end{figure}

The key difference between dangling for flat braiges and dangling for braiges is a matter of which braid is considered to be the one acting. For braiges, the braid acting by dangling has as many strands as feet of the braige; for flat braiges, the braid acting is the image of this braid under a cloning map, so has as many strands as there are strands of the flat braige just above the merges.

Let~$\braigecpx_n(L_n)$ be the set of equivalence classes under dangling of flat braiges in~$\braiges_n(L_n)$. The simplicial structure of the latter induces a simplicial structure on the former, for example the faces of~$[(b,\Gamma)]$ are precisely of the form~$[(bc,\Gamma')]$, for~$c\in B_n^{(J_\Gamma)}$ and~$\Gamma'$ a subgraph of~$\Gamma^c$. Also let~$\pbraigecpx_n(L_n)$ be the set of dangling flat pure braiges. The faces of a dangling flat pure braige~$[(p,\Gamma)]$ are the dangling pure braiges of the form~$[(pc,\Gamma')]$ for~$c\in PB_n^{(J_\Gamma)}$ and $\Gamma'$ a subgraph of~$\Gamma$. Heuristically, in~$\braigecpx_n(L_n)$ we can move around not only by changing the merges, but now also by changing the braid in certain controlled ways, so~$\braigecpx_n(L_n)$ and~$\pbraigecpx_n(L_n)$ stand a chance of being connected (for large enough~$n$), and even highly connected. In the pure case we can also define $\pbraigecpx_n(\Gamma)$ for any subgraph~$\Gamma$ of~$L_n$, by only considering flat
braiges from $\pbraiges_n(\Gamma)$. We also have the subcomplexes of dangling \emph{elementary} braiges or dangling elementary pure braiges, denoted~$\elbraigecpx_n(L_n)$ and~$\elpbraigecpx_n(L_n)$ respectively. In the pure case, note that $\elpbraigecpx_n(L_n)$ is identical to the complex $\elpbraigecpx_n$ analyzed in Section~\ref{sec:desc_link_conn}; in particular we already know its connectivity. Moreover in the pure case we can use any subgraph~$\Gamma$ of~$L_n$, and get the complex~$\elpbraigecpx_n(\Gamma)$. This will be an important subcomplex for proving that~$\pbraigecpx_n(\Gamma)$ is highly connected.

The left action of~$B_n$ on~$\braiges_n(L_n)$ induces an action of~$B_n$ on~$\braigecpx_n(L_n)$; for~$c\in B_n$ we have $c[(b,\Gamma)] \defeq [(cb,\Gamma)]$. Similarly,~$PB_n$ acts from the left on~$\pbraigecpx_n(L_n)$, and indeed stabilizes~$\pbraigecpx_n(\Gamma)$ for any subgraph~$\Gamma$ of~$L_n$. The action of~$PB_n$ on~$\pbraigecpx_n(L_n)$ is of particular interest, since there is a fundamental domain consisting of a single maximal simplex, namely~$[(\id,L_n)]$. This tells us that~$\pbraigecpx_n(L_n)$ is a coset complex, using the family of stabilizers of faces of~$[(\id,L_n)]$.

\begin{lemma}[Stabilizers of dangling braiges]\label{lem:braige_stabs}
 Let~$\Gamma$ be a subgraph of~$L_n$. Then the stabilizer $\Stab_{PB_n}([(\id,\Gamma)])$ is precisely the subgroup $PB_n^{(J_\Gamma)}$.
\end{lemma}

\begin{proof}
 First let~$p\in PB_n^{(J_\Gamma)}$. Then $p[(\id,\Gamma)]=[(p,\Gamma)]=[(\id,\Gamma)]$. Now suppose $p[(\id,\Gamma)]=[(\id,\Gamma)]$, so $[(p,\Gamma)]=[(\id,\Gamma)]$. Then there exists $c\in PB_n^{(J_\Gamma)}$ such that $(p,\Gamma)=(c,\Gamma)$. But this implies that~$p=c$, so we are done.
\end{proof}

Let $\braigefam_n \defeq \{PB_n^{(J_\Gamma)} \mid \Gamma$ is a subgraph of $L_n$ with one edge$\}$.

\begin{corollary}\label{cor:nerve_to_braiges}
 $\CC(PB_n,\braigefam_n)$ is isomorphic to $\pbraigecpx_n(L_n)$ as a simplicial~$PB_n$-complex.
\end{corollary}

\begin{proof}
 This is immediate from Proposition~\ref{prop:fund_dom_general}, since~$[(\id,L_n)]$ is a fundamental domain.
\end{proof}

In the next section we will calculate the connectivity of $\restrarc(L_n)$ and $\pbraigecpx_n(L_n)$, and hence of $\CC(PB_n,\arcfam_n)$ and $\CC(PB_n,\braigefam_n)$, from which we deduce higher generation.

\medskip

We close this section by setting up a generalization of the complexes we have constructed.  Note that in the definition of $\arcfam_n$ we require~$|J|=1$, and in the definition of~$\braigefam_n$ we require~$\Gamma$ to have only one edge (this is the same as saying~$|J_\Gamma|=1$).  The subgroups in these families consist of braids that, respectively, stabilize some arc, or feature at least one cloned strand. Of course, as~$n$ grows, it becomes increasingly ``easy'' for a braid to be very complicated while still featuring a cloned strand, or stabilizing an arc. Hence, higher generation becomes an even more interesting question if we consider requirements like, e.g., \emph{all but~5 strands are clones}. (Observe that any of the standard generators of~$PB_n$ satisfy this very requirement.)

\begin{definition}[More restrictive families]\label{def:restrictive_fams}
 Let $s\in\N$.  Define 
 \begin{align*}
 \arcfam_n^s & \defeq \{PB_n^J \mid J\subseteq[n-1]  \text{ with } |J|=s\}.
 \\
 \intertext{Also define}
 \braigefam_n^s & \defeq \{PB_n^{(J_\Gamma)} \mid \Gamma \text{ is a subgraph of $L_n$ with $s$ edges}\}.
 \end{align*}
Hence $\arcfam_n^1=\arcfam_n$ and $\arcfam_n^{n-1}=\{Z(PB_n)\}$, and also $\braigefam_n^1=\braigefam_n$ and $\braigefam_n^{n-1}=\{\{1\}\}$.
\end{definition}


\section{Connectivity of the complexes}\label{sec:conn_cpxes}

For $\ell\in\Z$ define $\eta(\ell) \defeq \lfloor \frac{\ell-2}{4} \rfloor$. The main goal of this section is to prove that~$\restrarc(L_n)$ and~$\pbraigecpx_n(L_n)$ are~$(\eta(n)-1)$-connected. Note that this is slightly different from the function $\eta$ defined before Theorem~\ref{thrm:surface_subline_matching_conn}; we do this because here the symbol $L_n$ denotes a graph with $n-1$ edges and there it had $n$ edges.

\begin{theorem}[Restatement of Theorem~\ref{thrm:surface_subline_matching_conn} using current notation]\label{thrm:matching_arc_conn}
 Let $\Gamma_m$ be a subgraph of~$L_n$ with~$m$ edges. Then~$\matcharc(\Gamma_m)$ is~$(\eta(m+1)-1)$-connected.
\end{theorem}

In particular~$\matcharc(L_n)$ is~$(\eta(n)-1)$-connected.

\subsection{Connectivity of arc complexes}\label{sec:arc_conn}

Our first goal is to deduce the connectivity of~$\restrarc(L_n)$ from Theorem~\ref{thrm:matching_arc_conn}. We will use the notion of \emph{defect} from Section~\ref{sec:match_cpxes}. Let $\Gamma_m$ be a subgraph of~$L_n$ with~$m$ edges. For a $k$-simplex $\sigma=\{\alpha_0,\dots,\alpha_k\}$ in~$\restrarc(\Gamma_m)$, define $r(\sigma)$ to be the number of points in $P$ that are used as endpoints of arcs in~$\sigma$. As in Section~\ref{sec:match_cpxes}, define the \emph{defect} $d(\sigma)$ to be $2(k+1)-r(\sigma)$. Let~$h$ be the function on the barycentric subdivision $\restrarc(\Gamma_m)'$ of $\restrarc(\Gamma_m)$ given by~$h(\sigma)=(d(\sigma),-\dim(\sigma))$, ordered lexicographically. Note that~$d(\sigma)=0$ if and only if the arcs are all disjoint, even at their endpoints. Hence, thinking of~$h$ as a height function on the vertices of $\restrarc(\Gamma_m)'$, in the sense of \cite{bestvina97}, we observe that the sublevel set $(\restrarc(\Gamma_m)')^{d=0}$ is precisely $\matcharc(\Gamma_m)'$. Hence we 
can compare the homotopy types of the two complexes using discrete Morse theory, with \cite[Corollary~2.6]{bestvina97} as the guide. The key is to inspect the descending links with respect to~$h$. This is very similar to the procedure used before to deduce connectivity of~$\matcharc(K_n)$ from connectivity of~$\hatcharc(K_n)$, but we will repeat many arguments for convenience.

\begin{proposition}\label{prop:conn_r_arc_cpx}
 $\restrarc(\Gamma_m)$ is $(\eta(m+1)-1)$-connected.
\end{proposition}

\begin{proof}
 We know that $\matcharc(\Gamma_m)$ is $(\eta(m+1)-1)$-connected by Theorem~\ref{thrm:matching_arc_conn}. We claim that the inclusion $\matcharc(\Gamma_m) \to \restrarc(\Gamma_m)$ induces a surjection in homotopy~$\pi_k$ for $k\leq \eta(m+1)-1$, from which the proposition follows. To prove the claim, it suffices by \cite[Corollary~2.6]{bestvina97} to prove that for $\sigma \in \restrarc(\Gamma_m) \setminus \matcharc(\Gamma_m)$, i.e., $h(\sigma)>0$, the descending link $\dlk(\sigma)$ is $(\eta(m+1)-2)$-connected. We suppose that~$\sigma$ is a $k$-simplex, with $\sigma=\{\alpha_0,\dots,\alpha_k\}$.

 There are two types of arc systems in $\dlk(\sigma)$. First, we could have $\sigma'<\sigma$ and~$h(\sigma')<h(\sigma)$. Then~$\sigma'$ is obtained from~$\sigma$ by removing arcs and strictly decreasing the defect. Call the full subcomplex of $\dlk(\sigma)$ spanned by these~$\sigma'$ the \emph{down-link}. Second, we could have $\widetilde{\sigma}>\sigma$ and $h(\widetilde{\sigma})<h(\sigma)$. Here $\widetilde{\sigma}$ is obtained by adding new arcs to $\sigma$, so that the new arcs are all disjoint from each other and from any existing arcs, even at endpoints. Call the full subcomplex of $\dlk(\sigma)$ spanned by such $\widetilde{\sigma}$ the \emph{up-link}. Any simplex in the down-link is a face of every simplex in the up-link, so~$\dlk(\sigma)$ is the join of the down-link and up-link.

 First consider the down-link. A face~$\sigma'$ of~$\sigma$ fails to be in the down-link if and only if each arc in~$\sigma\setminus\sigma'$ is disjoint from every other arc of $\sigma$, since then and only then do $\sigma$ and~$\sigma'$ have the same defect.  Let $\sigma_0$ be the face of $\sigma$ consisting precisely of all such arcs, if any exist.  Since $d(\sigma)>0$, we know $\sigma_0\neq\sigma$. The boundary of $\sigma$ is a~$(k-1)$-sphere, and the complement in the boundary of the down-link is either empty, or is a cone with cone point $\sigma_0$.  Hence the down-link is either a~$(k-1)$-sphere or is contractible, so in particular is $(k-2)$-connected. At this point we may assume without loss of generality that the down-link is a~$(k-1)$-sphere, and so every arc in~$\sigma$ shares an endpoint with some other arc in~$\sigma$. This means that every edge of~$\Gamma_\sigma$ shares an endpoint with some other edge of~$\Gamma_\sigma$. In particular~$k\ge 1$.

 Now consider the up-link. The simplices in the up-link are given by adding arcs to~$\sigma$ that are all disjoint from each other and from the arcs in $\sigma$. Consider the connected surface~$S' \defeq S \setminus \{\alpha_0,\dots,\alpha_k\}$, obtained by cutting out the arcs~$\alpha_i$. If~$P' \defeq S' \cap P$, then $|P'|=n-r(\sigma)$. Also let $\Gamma_{m-2k-2}'$ be the subgraph of~$\Gamma_m$ obtained by removing the edges of~$\Gamma_\sigma$, and all edges sharing a vertex with any of these, so $\Gamma_{m-2k-2}'$ has at most~$m-2k-2$ edges (here we use the fact that every edge of~$\Gamma_\sigma$ shares an endpoint with some other edge of~$\Gamma_\sigma$). The up-link of~$\sigma$ is isomorphic to the matching complex $\matcharc(S',P',\Gamma_{m-2k-2}')$, which is~$(\eta(m-2k-1)-1)$-connected. Since $\dlk(\sigma)$ is the join of the down- and up-links, we conclude that~$\dlk(\sigma)$ is $(\eta(m-2k-1)+k-1)$-connected.

We have 
\begin{align*}
\eta(m-2k-1)+k-1 & \ge \frac{m-2k-3}{4}+k-2
\\
& \ge \eta(m+1)+\frac{k}{2}-\frac{5}{2} \ge \eta(m+1) -2
\end{align*}
 since $k\ge 1$, and so we are done.
\end{proof}

The next corollary is immediate, keeping in mind that with our notation~$L_n$ has~$n-1$~edges.

\begin{corollary}\label{cor:conn_r_lin_arc_cpx}
 $\restrarc(L_n)$ is $(\eta(n)-1)$-connected.\qed
\end{corollary}

\begin{corollary}\label{cor:arc_gen}
 $\CC(PB_n,\arcfam_n)$ is $(\eta(n)-1)$-connected, and hence $\arcfam_n$ is $\eta(n)$-generating for~$PB_n$.
\end{corollary}

\begin{proof}
 This is immediate from Lemma~\ref{lem:nerve_to_arcs} and Corollary~\ref{cor:conn_r_lin_arc_cpx}.
\end{proof}

We also want to show that the families $\arcfam_n^s$ from Definition~\ref{def:restrictive_fams} are highly generating. For $s>1$, the coset complex $\CC(PB_n,\arcfam_n^s)$ is obtained up to homotopy equivalence from $\CC(PB_n,\arcfam_n^{s-1})$ by removing the open stars of the vertices, i.e., the cosets $pPB_n^J$ for $|J|=s-1$. Hence the problem amounts to showing high connectivity of links. This is more or less the procedure done in the proof of Theorem~3.3 in \cite{abels93}, in the context of buildings. It is a bit harder here though; links in buildings are themselves buildings, but links in restricted arc complexes are not themselves restricted arc complexes. Nonetheless, we can get the right connectivity without too much extra work.

\begin{lemma}[Links in $\restrarc(\Gamma_m)$]\label{lem:restrarc_lks}
 Let $\sigma=\{\alpha_0,\dots,\alpha_k\}$ be a $k$-simplex in $\restrarc(\Gamma_m)$ for $\Gamma_m$ as above (with~$m$ edges). Then the link~$\lk_{\restrarc(\Gamma_m)}(\sigma)$ is $(\eta(m-k)-1)$-connected.
\end{lemma}

To make precise the terminology, here by ``link'' we mean the subcomplex of simplices $\tau$ disjoint from $\sigma$ for which there exists a simplex with $\tau$ and $\sigma$ as faces.

\begin{proof}
 Set $L\defeq \lk_{\restrarc(\Gamma_m)}(\sigma)$. An arc system~$\tau$ is in~$L$ if and only if each arc of~$\tau$ is distinct from, but compatible with, every $\alpha_i$.  For such a $\tau$, by retracting each arc $\alpha_i$ to a point, $\tau$ maps to an arc system in $\restrarc(\Gamma_{m-(k+1)})$. Here $\Gamma_{m-(k+1)}$ is a subgraph of $\Gamma_m$ with $m-(k+1)$ edges. More formally, for $0\le d\le k$ consider the homotopy equivalence of surfaces $r_d \colon S\to S_d$, obtained by collapsing $\alpha_i$ to a point, for each $0\le i\le d$. Recall $S=D_n$, and here $S_d$ is just our name for the copy of $D_{n-(d+1)}$ obtained by collapsing these arcs. Here we do not think of $D_n$ as a punctured disk, but rather as a disk with $n$ distinguished points; hence $r_d$ is really a homotopy equivalence. Also let $P_d$ be the image of $P$ under $r_d$. We have induced maps of complexes $R_d \colon L \to \restrarc(\Gamma_{m-(d+1)})$. Note that these maps are surjective, but not injective; see Figure~\ref{fig:blow_up_arc} for an example of the non-injectivity. Note however that the connectivity of $\restrarc(\Gamma_{m-(k+1)})$ is precisely the connectivity we are trying to verify for $L$.

\begin{figure}[t]
\centering
\begin{tikzpicture}[line width=0.8pt, scale=1.3]
  \filldraw
   (0,0) circle (1.5pt)   (1,0) circle (1.5pt)   (2,0) circle (1.5pt)   (3,0) circle (1.5pt);
  \draw[dashed]
   (1,0) -- (2,0);
  \node at (1.5,0.2) {$\alpha_d$};
  \draw
   (0,0) -- (1,0);
 \begin{scope}[xshift=5cm]
 \filldraw
   (0,0) circle (1.5pt)   (1,0) circle (1.5pt)   (2,0) circle (1.5pt)   (3,0) circle (1.5pt);
  \draw[dashed]
   (1,0) -- (2,0);
  \node at (1.5,0.2) {$\alpha_d$};
  \draw
   (0,0) to [out=45, in=90, looseness=1] (2.5,0)
   (2.5,0) to [out=-90, in=-80, looseness=1.3] (1,0);
 \end{scope}
\end{tikzpicture}
\caption{Distinct arcs in the link of $\sigma$ that map to the same arc under $R_d$.}
\label{fig:blow_up_arc}
\end{figure}
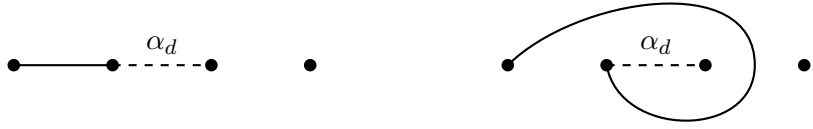

 The $r_d$ also induce epimorphisms $\phi_d \colon \Stab_{PB_n}(\sigma) \to PB_{n-(d+1)}$, with kernels $K_d\defeq \ker(\phi_d)$. Also declare~$K_{-1}$ to be the trivial subgroup. Note that $K_{-1}\le K_0\le \cdots \le K_k$. Colloquially, the pure braids~$p$ in $K_d\setminus K_{d-1}$ are precisely those that do ``twist''~$\alpha_d$ but don't twist any $\alpha_i$ for~$i>d$. For $p\in K_k$, define $D(p)\defeq \min\{d+1\mid p\in K_d\}$. We will call $D(p)$ the \emph{deviation} of $p$; note that~$D(p)=0$ if and only if $p=\id$. Now fix a map $s_{\id} \colon S_k\to S$ with $s_{\id} \circ r_k$ homotopic to the identity. This essentially amounts to fixing a choice of how to ``blow up'' each arc $\alpha_i$ to get from~$S_k$ back to~$S$. We get an induced map $\iota_{\id} \colon \restrarc(\Gamma_{m-(k+1)}) \to L$, with $R_k\circ \iota_{\id}$ equal to the identity on $\restrarc(\Gamma_{m-(k+1)})$. For each $p\in K_k$, set $\iota_p\defeq p\circ \iota_{\id}$. These maps are all injective simplicial maps that can be thought of as 
different choices of how to blow up each $\alpha_i$, and we see that $R_k\circ \iota_p$ is the identity for all $p$. Every arc system in $L$ is the image of an arc system in $\restrarc(\Gamma_{m-(k+1)})$ under some $\iota_p$, so~$L=\bigcup\limits_{p\in PB_n} \image(\iota_p)$. Also, each $\image(\iota_p)$ is isomorphic to $\restrarc(\Gamma_{m-(k+1)})$, and hence is an $(\eta(m-k)-1)$-connected subcomplex of $L$. We now need to glue these $\image(\iota_p)$ together in a clever order, always along $(\eta(m-k)-2)$-connected relative links, from which we will deduce that~$L$ is $(\eta(m-k)-1)$-connected.

 The measurement $D(p)$ provides such an order. For $0\le d\le k$ let $L^d \defeq \bigcup\limits_{D(p)\le d} \image(\iota_p)$. We claim that $L^d$ is $(\eta(m-k)-1)$-connected for all $d$. The base case $d=0$ is clear. For a given $d$, the intersection $\image(\iota_p) \cap \image(\iota_q)$ with $p\neq q$ and~$D(p)=D(q)=d+1$ is contained in $L^d$. This is because~$p$ and~$q$ must twist the arc $\alpha_d$ differently, and so if $\beta$ is an arc in $\image(\iota_p) \cap \image(\iota_q)$ then $\beta$ cannot share endpoints with $\alpha_d$. For this reason, we can build up from $L^d$ to $L^{d+1}$ by attaching the $\image(\iota_p)$ with deviation $d+1$, in any order, and the relative links will always be in $L^d$. Now, for~$p$ with $D(p)=d+1$, we attach $\image(\iota_p)$ to $L^d$ along the intersection $\image(\iota_p)\cap L^d$.  This intersection consists precisely of those arc systems in $\image(\iota_p)$ that do not use arcs sharing endpoints with $\alpha_d$. Applying~$R_k$ (so retracting each $\alpha_i$ to a 
point), this gives us the subcomplex of $\restrarc(\Gamma_{m-(k+1)})$ whose arcs are disjoint from the endpoint obtained by collapsing $\alpha_d$.  But this is just $\restrarc(\Gamma')$ for~$\Gamma'$ a subgraph of~$\Gamma_{m-(k+1)}$ with at most two fewer edges.  This is $(\eta(m-k)-2)$-connected, and so we are done.
\end{proof}

\begin{proposition}\label{prop:restrictive_arc_conn}
 For $s\in\N$, $\CC(PB_n,\arcfam_n^s)$ is $(\eta(n-(s-1))-1)$-connected, and hence~$\arcfam_n^s$ is $(\eta(n-(s-1)))$-generating for $PB_n$.
\end{proposition}

\begin{proof}
 It suffices to show that for $|J|=s-1$, the link of $PB_n^J$ in $\CC(PB_n,\arcfam_n^{s-1})$ is $(\eta(n-(s-1))-1)$-connected. Equivalently, we need the link of $\sigma_J$ in $\restrarc(L_n)$ to be $(\eta(n-(s-1))-1)$-connected. Since $\sigma_J$ is a $(|J|-1)$-simplex, its link is $(\eta(n-|J|)-1)$-connected by Lemma~\ref{lem:restrarc_lks} (since~$L_n$ has $n-1$ edges), and since~$|J|=s-1$, we conclude that indeed the link is $(\eta(n-(s-1))-1)$-connected.
\end{proof}

\subsection{Connectivity of flat braige complexes}\label{sec:braige_conn}

Now we inspect $\CC(PB_n,\braigefam_n)$, or more accurately $\pbraigecpx_n(L_n)$. To pass from the world of arcs to the world of flat braiges, we will project the flat braiges onto arcs in the following way. For each~$J\subseteq [n-1]$, let~$\sigma_J$ be the simplex of~$\matcharc(L_n)$ defined before Lemma~\ref{lem:nerve_to_arcs}. Consider the action of~$PB_n$ on~$\restrarc(L_n)$ as a right action, and define a map
\begin{align*}
 \pi \colon \pbraigecpx_n(L_n) &\to \restrarc(L_n)\\
[(p,\Gamma)] &\mapsto (\sigma_{J_\Gamma})p^{-1}
\end{align*}
where $J_\Gamma$ is as in Definition~\ref{def:dangling}. We will use~$\pi$ to also denote the restrictions $\elpbraigecpx_n(L_n) \to \matcharc(L_n)$, $\pbraigecpx_n(\Gamma) \to \restrarc(\Gamma)$ and $\elpbraigecpx_n(\Gamma) \to \matcharc(\Gamma)$ for~$\Gamma$ a subgraph of~$L_n$. As in Section~\ref{sec:desc_link_conn}, think of~$\pi$ as the procedure of combing the braid straight and watching where the arcs get moved.

\begin{proposition}[Flat braige connectivity from arc connectivity]\label{prop:braiges_to_arcs}
 For $\Gamma_m$ a subgraph of~$L_n$ with~$m$ edges, $\elpbraigecpx_n(\Gamma_m)$ is~$(\eta(m+1)-1)$-connected.
\end{proposition}

When $\Gamma_m=L_n$, this is just Corollary~\ref{cor:pure_desc_lk_conn}. Indeed the proof here is more or less the same, but we will repeat it for convenience.

\begin{proof}
 By Theorem~\ref{thrm:matching_arc_conn} $\matcharc(\Gamma_m)$ is~$(\eta(m+1)-1)$-connected. Let $\sigma=\{\alpha_0,\dots,\alpha_k\}$ be a $k$-simplex in~$\matcharc(\Gamma_m)$. The link $\lk(\sigma)$ of $\sigma$ in $\matcharc(\Gamma_m)$ is isomorphic to~$\matcharc(\Gamma')$ for $\Gamma'$ a subgraph of $\Gamma_m$ with at least $m-3(k+1)$ edges, so $\lk(\sigma)$ is $(\eta(m-3(k+1)+1)-1)$-connected, and hence $(\eta(m+1)-k-2)$-connected. It now suffices by \cite[Theorem~9.1]{quillen78} to prove that the fiber~$\pi^{-1}(\sigma)$ is $(k-1)$-connected (here we treat a simplex as a closed cell). Indeed, we will prove that $\pi^{-1}(\sigma)$ is the join of the fibers $\pi^{-1}(\alpha_i)$ of the vertices $\alpha_i$ of $\sigma$. See also Proposition~\ref{prop:fibers}.

 Let $\fiberjoin \defeq \bigjoin_{i=0}^k \pi^{-1}(\alpha_i)$ be the join of the vertex fibers. Clearly $\pi^{-1}(\sigma)\subseteq \fiberjoin$. Also, the~$0$-skeleton of $\fiberjoin$ is contained in $\pi^{-1}(\sigma)$. Now suppose that the same is true of the~$r$-skeleton for some $r\ge 0$. An~$(r+1)$-simplex in~$\fiberjoin$ is the join of a~$0$-simplex and an~$r$-simplex, both of which are contained in~$\pi^{-1}(\sigma)$. It now suffices to prove the following claim.

 \emph{Claim}: Let $[(p,E)]$ be a vertex in $\elpbraigecpx_n(\Gamma_m)$, so $p\in PB_n$ and $E$ is a one-edge subgraph of~$\Gamma_m$. Let $[(q,\Gamma)]$ be a simplex in $\elpbraigecpx_n(\Gamma_m)$ such that $\pi([(q,\Gamma)])$ does not contain $\pi([(p,E)])$ but does share a simplex with $\pi([(p,E)])$ in~$\matcharc(\Gamma_m)$. Then~$[(q,\Gamma)]$ shares a simplex with $[(p,E)]$ in $\elpbraigecpx_n(\Gamma_m)$.

 This hypothesis is rephrased in terms of arcs as: $(\Gamma)q^{-1}$ shares a simplex with~$(E)p^{-1}$. By acting from the left with~$PB_n$, we can assume without loss of generality that~$p=\id$, so we have $\pi([(p,E)])=E$. Let $\{\beta_0,\dots,\beta_\ell\} \defeq (\Gamma)q^{-1}$, chosen so that~$E$ is disjoint from the~$\beta_i$, even at endpoints (remember we are in~$\matcharc(\Gamma_m)$, not just $\restrarc(\Gamma_m)$). This is possible by the hypothesis, and implies that the dangling equivalence class $[(q,\Gamma)]$ contains a representative in which the~$(j+1)^{\text{st}}$ strand is a clone of the~$j^{\text{th}}$ strand, where $j$ and $j+1$ are the endpoints of the edge of~$E$. We can assume $(q,\Gamma)$ itself is such a representative, in which case the dangling flat braige $[(q,\Gamma\cup E)]$ is a simplex of $\elpbraigecpx_n(\Gamma_m)$ containing $[(q,\Gamma)]$ and $[(p,E)]$, proving the claim.
\end{proof}

It might be possible to mimic this proof using $\pi \colon \pbraigecpx_n(\Gamma) \to \restrarc(\Gamma)$ instead, and get the connectivity of $\pbraigecpx_n(L_n)$ right away, but the downside is that the fibers are not joins of vertex fibers. Hence one would have to do extra work to show that fibers have the right connectivity.

To calculate the connectivity of~$\pbraigecpx_n(\Gamma_m)$, we will use a similar procedure as for~$\restrarc(\Gamma_m)$. Namely, we will build up from~$\elpbraigecpx_n(\Gamma_m)$ to~$\pbraigecpx_n(\Gamma_m)$ using discrete Morse theory. A~$k$-simplex in~$\pbraigecpx_n(\Gamma_m)$ is a dangling equivalence class of a pair~$(p,\Gamma)$, for~$p\in PB_n$ and~$\Gamma$ a subgraph of~$\Gamma_m$ with~$k+1$ edges. Let~$r(\Gamma)$ be the number of vertices that are endpoints of an edge in~$\Gamma$. Then define the \emph{defect} $d(p,\Gamma)$ to be~$2(k+1)-r(\Gamma)$. Extend these definitions to the dangling equivalence classes, and observe that $\elpbraigecpx_n(\Gamma_m)$ is the~$d=0$ sublevel set of~$\pbraigecpx_n(\Gamma_m)$. We now apply Morse theory, as before.

\begin{proposition}\label{prop:conn_r_braige_cpx}
 $\pbraigecpx_n(\Gamma_m)$ is $(\eta(m+1)-1)$-connected.
\end{proposition}

\begin{proof}
 By Proposition~\ref{prop:braiges_to_arcs}, $\elpbraigecpx_n(\Gamma_m)$ is $(\eta(m+1)-1)$-connected. Mimicking the proof of Proposition~\ref{prop:conn_r_arc_cpx}, it suffices to prove that for $\sigma \in \pbraigecpx_n(\Gamma_m) \setminus \elpbraigecpx_n(\Gamma_m)$, the descending link~$\dlk(\sigma)$ is $(\eta(m+1)-2)$-connected. Let~$\sigma$ be such a~$k$-simplex, say~$\sigma=[(p,\Gamma)]$. The down-link is either~$S^{k-1}$, or contractible if~$\Gamma$ has an isolated edge. Suppose there is no such isolated edge, so the down-link is~$S^{k-1}$. Now, the up-link is obtained by dangling and then adding extra edges to the graph, such that the new edges are disjoint from~$\Gamma$ and from each other. Since~$\Gamma$ has no isolated edges, there are at most~$2(k+1)$ edges of~$\Gamma_m$ that share an endpoint with an edge of~$\Gamma$. Hence the up-link of~$\sigma$ is isomorphic to~$\elpbraigecpx_\ell(\Gamma_{m-2k-2})$ for some~$\ell$, which is $(\eta(m-2k-1)-1)$-connected. The calculation from the proof of 
Proposition~\ref{prop:conn_r_arc_cpx} now tells us that $\dlk(\sigma)$ is $(\eta(m+1)-2)$-connected.
\end{proof}

\begin{corollary}\label{cor:conn_r_lin_braige_cpx}
 $\pbraigecpx_n(L_n)$ is $(\eta(n)-1)$-connected.\qed
\end{corollary}

\begin{corollary}\label{cor:braige_gen}
 $\CC(PB_n,\braigefam_n)$ is $(\eta(n)-1)$-connected, and hence $\braigefam_n$ is~$\eta(n)$-generating for~$PB_n$. \qed
\end{corollary}

\begin{example}\label{ex:clones_hi_gen}
 For $n\ge 6$, $\CC(PB_n,\braigefam_n)$ is connected, so $PB_n$ has a generating set in which each generator features at least one cloned strand. Indeed, the standard generating set from Section~1.3.1 of \cite{kassel08} satisfies this property for~$n\ge 6$, and fails for $n<6$. For~$n\ge 10$, $\CC(PB_n,\braigefam_n)$ is simply connected, so~$PB_n$ is~$2$-generated by $\braigefam_n$. Hence there exists a presentation for~$PB_n$ in which every generator features a cloned strand, and the relations all arise from relations in the subgroups of braids with a cloned strand. Again we note that the standard presentation works precisely in this range.
\end{example}

We conclude by showing that the families $\braigefam_n^s$ for $s\in\N$, defined in Definition~\ref{def:restrictive_fams}, are highly generating as well. Just like in the arc case, for $s>1$ the coset complex $\CC(PB_n,\braigefam_n^s)$ is obtained up to homotopy equivalence from $\CC(PB_n,\braigefam_n^{s-1})$ by removing the open stars of vertices, i.e, cosets $pPB_n^{(J)}$ for~$|J|=s-1$.

\begin{lemma}[Links in $\pbraigecpx_n(\Gamma_m)$]\label{lem:braige_lks}
 Let $\sigma$ be a $k$-simplex in $\pbraigecpx_n(\Gamma_m)$ for $\Gamma_m$ as above (with~$m$ edges). Then the link~$\lk_{\pbraigecpx_n(\Gamma_m)}(\sigma)$ is $(\eta(m-k)-1)$-connected.
\end{lemma}

\begin{proof}
 Links in the flat braige case are nicer than links in the arc case, since they are actually isomorphic to smaller dangling flat braige complexes. In the arc case, namely in the proof of Lemma~\ref{lem:restrarc_lks}, we related a given link to a smaller arc complex, via a map that was not an isomorphism. In the present case, we claim that~$\lk_{\pbraigecpx_n(\Gamma_m)}(\sigma)$ is just isomorphic to $\pbraigecpx_{n-(k+1)}(\Gamma_{m-(k+1)})$, for~$\Gamma_{m-(k+1)}$ a graph with $m-(k+1)$ edges, and then the connectivity result is immediate. Say $\sigma=[(p,\Gamma_{k+1})]$ for~$\Gamma_{k+1}$ a subgraph of~$\Gamma_m$ with $k+1$ edges. Let $L\defeq \lk_{\pbraigecpx_n(\Gamma_m)}(\sigma)$.  The simplices in $L$ are dangling flat braiges of the form $\tau=[(pq,\Gamma)]$, where $q\in PB_n^{(J_{\Gamma_{k+1}})}$ and~$\Gamma$ is a subgraph of $\Gamma_m$ having no edges in common with $\Gamma_{k+1}$. The first condition ensures that~$\tau$ and $\sigma$ share a simplex, namely $[(pq,\Gamma\cup\Gamma_{k+1})]$, and the 
second condition ensures that~$\tau$ and~$\sigma$ are disjoint.  Acting from the left with $PB_n$, we can assume~$p=\id$. We have a map $\phi \colon L \to \pbraigecpx_{n-(k+1)}(\Gamma_{m-(k+1)})$, where $\Gamma_{m-(k+1)}$ is the graph with $n-(k+1)$ vertices that is obtained from $\Gamma_m$ by retracting each edge of~$\Gamma_{k+1}$ to a point.  The map $\phi$ sends $\tau=[(q,\Gamma)]$ to $[(q',\Gamma')]$, where $\Gamma'$ is the image of $\Gamma$ under the retraction $\Gamma_m\to \Gamma_{m-(k+1)}$, and $q'$ is the preimage of~$q$ under the cloning map $\kappa_{J_{\Gamma_{k+1}}}$. See Figure~\ref{fig:unclone} for an example. Since~$q'$ is uniquely determined by~$q$, we have an inverse~$\phi^{-1}$, induced by the cloning map. (This is the essential difference from the arc case, that there is only one way to ``blow up'' a braige via cloning.) Since $\phi$ and $\phi^{-1}$ are of course simplicial maps, we conclude that $\phi$ is a simplicial isomorphism, and the result follows.
\end{proof}

\begin{figure}[t]
\centering
\begin{tikzpicture}[line width=0.8pt, yscale=0.5]
  \draw
   (2,0) -- (2,-0.5) to [out=-90, in=90, looseness=1] (1,-3)   (2,-3) to [out=-90, in=90, looseness=1] (1,-6);
  \draw[white, line width=4pt]
   (1,0) -- (1,-0.5) to [out=-90, in=90, looseness=1] (2,-3)   (1,-3) to [out=-90, in=90, looseness=1] (2,-6);
  \draw
   (1,0) -- (1,-0.5) to [out=-90, in=90, looseness=1] (2,-3)   (1,-3) to [out=-90, in=90, looseness=1] (2,-6)
   (3,0) -- (3,-0.5) to [out=-90, in=90, looseness=1] (5,-3);
  \draw[white, line width=4pt]
   (5,0) -- (5,-0.5) to [out=-90, in=90, looseness=1] (4,-3)   (4,0) -- (4,-0.5) to [out=-90, in=90, looseness=1] (3,-3);
  \draw
   (5,0) -- (5,-0.5) to [out=-90, in=90, looseness=1] (4,-3)   (4,0) -- (4,-0.5) to [out=-90, in=90, looseness=1] (3,-3)   (3,-3) to [out=-90, in=90, looseness=1] (4,-6)   (4,-3) to [out=-90, in=90, looseness=1] (5,-6);
  \draw[white, line width=4pt]
    (5,-3) to [out=-90, in=90, looseness=1] (3,-6);
  \draw
    (5,-3) to [out=-90, in=90, looseness=1] (3,-6);
  \draw[dashed, red]
   (4,-6) -- (5,-6);
  \draw[red]
   (2,-6) -- (3,-6)   (3,-6) -- (4,-6);
  \node at (6.5,-3) {$\stackrel{\phi}{\longmapsto}$};

 \begin{scope}[xshift=7cm]
  \draw
   (2,0) -- (2,-0.5) to [out=-90, in=90, looseness=1] (1,-3)   (2,-3) to [out=-90, in=90, looseness=1] (1,-6);
  \draw[white, line width=4pt]
   (1,0) -- (1,-0.5) to [out=-90, in=90, looseness=1] (2,-3)   (1,-3) to [out=-90, in=90, looseness=1] (2,-6);
  \draw
   (1,0) -- (1,-0.5) to [out=-90, in=90, looseness=1] (2,-3)   (1,-3) to [out=-90, in=90, looseness=1] (2,-6)
   (3,0) -- (3,-0.5) to [out=-90, in=90, looseness=1] (4,-3);
  \draw[white, line width=4pt]
   (4,0) -- (4,-0.5) to [out=-90, in=90, looseness=1] (3,-3);
  \draw
   (4,0) -- (4,-0.5) to [out=-90, in=90, looseness=1] (3,-3)   (3,-3) to [out=-90, in=90, looseness=1] (4,-6);
  \draw[white, line width=4pt]
    (4,-3) to [out=-90, in=90, looseness=1] (3,-6);
  \draw
    (4,-3) to [out=-90, in=90, looseness=1] (3,-6);
  \draw[red]
   (2,-6) -- (3,-6)   (3,-6) -- (4,-6);
 \end{scope}
\end{tikzpicture}
\caption{The map $\phi$ takes an element of $\lk_{\pbraigecpx_5(L_5)}(\sigma)$ to an element of $\pbraigecpx_4(L_4)$. Here $\sigma$ is $[(\id,E_4)]$, for $E_4$ the subgraph with a single edge indicated by the dashed line.}
\label{fig:unclone}
\end{figure}
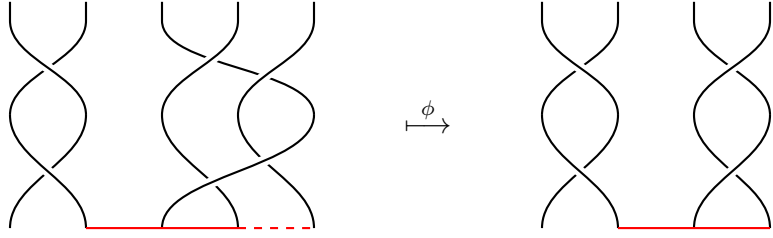

\begin{proposition}\label{prop:restrictive_braige_conn}
 For $s\in\N$, $\CC(PB_n,\braigefam_n^s)$ is $(\eta(n-(s-1))-1)$-connected, and hence~$\braigefam_n^s$ is $\eta(n-(s-1))$-generating for $PB_n$.
\end{proposition}

\begin{proof}
 As in the proof of Proposition~\ref{prop:restrictive_arc_conn}, it suffices to prove that for~$\Gamma$ with~$s-1$ edges, the link of the $(s-2)$-simplex $[(\id,\Gamma)]$ in $\pbraigecpx_n(L_n)$ is $(\eta(n-(s-1))-1)$-connected. Since $L_n$ has $n-1$ edges, this follows from Lemma~\ref{lem:braige_lks}.
\end{proof}

\begin{example}\label{ex:clones_restrictive_hi_gen}
 To generalize the previous example, we have that for any~$n\ge 6$, $\braigefam_n^{n-5}$ is~$1$-generating for~$PB_n$. This means that $PB_n$ has a set of generators such that in each generator, all but~$5$ strands are clones (indeed the standard generators have this property). Similarly for $n\ge 10$,~$\braigefam_n^{n-9}$ is~$2$-generating for~$PB_n$, so $PB_n$ has a presentation in which each relation can be realized by using only~$9$ non-clone strands. Again, the standard presentation fits the bill.
\end{example}

\begin{example}\label{ex:arcs_restrictive_hi_gen}
 In the situation of arcs, the \emph{swing presentation} for $PB_n$, described in Section~4 of \cite{margalit09}, provides an explicit example of $\arcfam_n^{n-5}$ being~$1$-generating for $n\ge 6$ and $\arcfam_n^{n-9}$ being~$2$-generating for $n\ge 10$. In this presentation the generators are Dehn twists, each of which must stabilize at least one arc of the form $\sigma_j$, as soon as $n\ge 6$. Each relation in \cite[Theorem~4.10]{margalit09} (specifically the second presentation) is a product of Dehn twists, and for $n\ge 10$ this product stabilizes at least one arc of the form $\sigma_j$. See Figure~\ref{fig:swings} for an example.
\end{example}

\begin{figure}[t]
\centering
\begin{tikzpicture}[line width=0.7pt]
  \draw[line width=1pt]
   (0,0) -- (2,0) -- (3,-1.732) -- (2,-3.464) -- (0,-3.464);
  \draw[dashed, line width=1pt]
   (0,-3.464) -- (-1,-1.732);
  \filldraw
   (0,0) circle (1.5pt)   (2,0) circle (1.5pt)   (3,-1.732) circle (1.5pt)   (2,-3.464) circle (1.5pt)   (0,-3.464) circle (1.5pt)   (-1,-1.732) circle (1.5pt);
  \draw[red]
   (1.9,-1.7) ellipse (0.3in and 0.8in);

 \begin{scope}[xshift=7cm]
  \draw[line width=1pt]
   (0,0) -- (1,0) -- (1.809,-.588) -- (2.118,-1.539)   (1.809,-2.49) -- (1,-3.078) -- (0,-3.078) -- (-.809,-2.49) -- (-1.118,-1.539) -- (-0.809,-.588);
  \draw[dashed, line width=1pt]
   (2.118,-1.539) -- (1.809,-2.49);
  \filldraw
   (0,0) circle (1.5pt)   (1,0) circle (1.5pt)   (1.809,-.588) circle (1.5pt)   (2.118,-1.539) circle (1.5pt)   (1.809,-2.49) circle (1.5pt)   (1,-3.078) circle (1.5pt)   (0,-3.078) circle (1.5pt)   (-.809,-2.49) circle (1.5pt)   (-1.118,-1.539) circle (1.5pt)   (-0.809,-.588) circle (1.5pt);
  \draw[red, rotate=-45]
   (1.6,-0.8) ellipse (0.63in and 0.9in);
  \draw[blue, rotate=-52]
   (1.5,-0.5) ellipse (0.2in and 0.75in);
  \draw[green, rotate=-18]
   (0.95,-2.6) ellipse (0.45in and 0.08in);
  \draw[teal, rotate=-18]
   (1.9,-1.3) ellipse (0.1in and 0.6in);
 \end{scope}
\end{tikzpicture}
\caption{With $6$ points, each generator must stabilize an arc. With $10$ points, each relation must stabilize an arc. The dashed lines indicate the arcs stabilized in the examples.  The relation pictured here is a \emph{lantern relation}, as in Figure~12 of \cite{margalit09}.}
\label{fig:swings}
\end{figure}
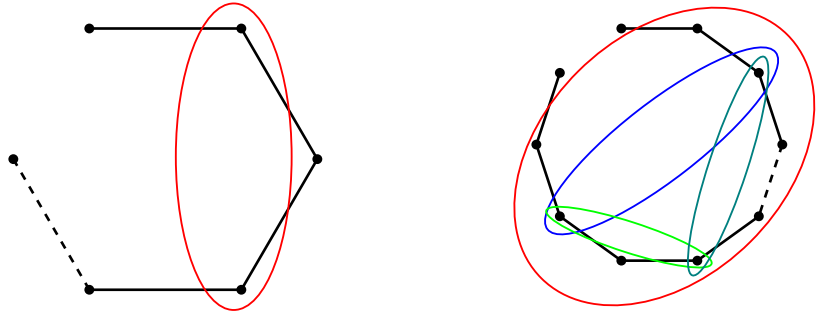


\bibliographystyle{alpha}



\bigskip

\begin{footnotesize}
\parindent0pt

\textsc{Department of Mathematics, Bielefeld University, 33501
Bielefeld, Germany}
\par\nopagebreak
\emph{E-mail address:} \texttt{bux\_2009@kubux.net}

\bigskip

\textsc{Department of Mathematics, Bielefeld University, 33501
Bielefeld, Germany}
\par\nopagebreak
\emph{E-mail address:} \texttt{mfluch@math.uni-bielefeld.de}

\bigskip

\textsc{Department of Mathematics, Bielefeld University, 33501
Bielefeld, Germany}
\par\nopagebreak
\emph{E-mail address:} \texttt{marco.marschler@math.uni-bielefeld.de}

\bigskip

\textsc{Mathematical Institute, University of M\"unster, 
  48149 M\"unster, Germany}
\par\nopagebreak
\emph{E-mail address:} \texttt{s.witzel@uni-muenster.de}

\bigskip

\textsc{Department of Mathematical Sciences, Binghamton University, Binghamton, NY 13902}\par\nopagebreak
\emph{E-mail address:} \texttt{zaremsky@math.binghamton.edu}
\end{footnotesize}

\end{document}